\documentclass[11pt,a4paper,reqno]{amsart}            

\usepackage[english]{babel}

\usepackage{amssymb,amsmath,amsthm}
\usepackage{soul}
\usepackage{comment}
\usepackage{mathrsfs}
\usepackage{enumitem}
\usepackage[immediate,debrief]{silence}
\usepackage[pagebackref]{hyperref}

\usepackage{color}

\usepackage{scrtime}

\hypersetup{
 colorlinks,
 linkcolor={blue!90!black},
 citecolor={red!80!black},
 urlcolor={blue!50!black}
}
\usepackage{tikz}
\usepackage{tikz-cd}
\usepackage{graphicx}
\usetikzlibrary{shapes.geometric,arrows,fit,matrix,positioning,trees,snakes}
\tikzset
{
    treenode/.style = {circle, draw=black, align=center, minimum size=1cm},
    subtree/.style  = {isosceles triangle, draw=black, align=center, minimum height=0.5cm, minimum width=1cm, shape border rotate=90, anchor=north}
}

\renewcommand\mathcal{\mathscr}
\theoremstyle{plain}
\newtheorem{theorem}{Theorem}[section]
\newtheorem*{theorem*}{Theorem}
\newtheorem{lemma}[theorem]{Lemma}
\newtheorem*{lemma*}{Lemma}
\newtheorem{corollary}[theorem]{Corollary}
\newtheorem{proposition}[theorem]{Proposition}

\theoremstyle{remark}
\newtheorem{remark}[theorem]{Remark}
\newtheorem*{remark*}{Remark}

\theoremstyle{definition}
\newtheorem{definition}[theorem]{Definition}
\newtheorem*{definition*}{Definition}

\theoremstyle{example}

\newtheorem*{example*}{Example}

\theoremstyle{caveat}
\newtheorem{caveat}[theorem]{Caveat}
\newtheorem*{caveat*}{Caveat}

\theoremstyle{notation}
\newtheorem{notation}[theorem]{Notation}
\newtheorem*{notation*}{Notation}

\numberwithin{equation}{section}

\newcount\quantno
\everydisplay{\quantno=0}\everycr{\quantno=0}
\newcommand\quant{\advance\quantno by1
                      \ifnum\quantno=1\qquad\else\quad\fi\forall }
\newcommand\itemno[1]{(\romannumeral #1)}

\newcommand\rest[1]{\kern-.1em
          \lower.5ex\hbox{$\scriptstyle #1$}\kern.05em}

\renewcommand\mod[1]{\vert{#1}\vert}
\newcommand\bigmod[1]{\bigl\vert{#1}\bigr|}

\newcommand\bignorm[2]{\left.{\bigl\Vert{#1}\bigr\Vert_{#2}}\right.}
\newcommand\bignormto[3]{\left.{\bigl\Vert{#1}\bigr\Vert_{#2}^{#3}}\right.}

\newcommand\prodo[2]{\left\langle#1,#2\right\rangle}

\newcommand\wrt{\,\text{\rm d}}

\newcommand\bB{\mathbf{B}}

   \newcommand\mrT{\mathrm{T}}

\newcommand\BF{\mathbb{F}}

\newcommand\BN{\mathbb{N}}
\newcommand\BR{\mathbb{R}}

\newcommand\cB{\mathcal{B}}   
 
\newcommand\cC{\mathcal{C}}   
 
 \newcommand\fD{\mathfrak{D}}

\newcommand\cF{\mathcal{F}}

\newcommand\cM{\mathcal{M}}

  \newcommand\fS{\mathfrak{S}}

\newcommand\cX{\mathcal{X}}

\newcommand\al{\alpha}
\newcommand\be{\beta}
\newcommand\ga{\gamma}    \newcommand\Ga{\Gamma}
\newcommand\de{\delta}

\newcommand\la{\lambda}   
\newcommand\om{\omega}    \newcommand\Om{\Omega}  
\newcommand\si{\sigma}    \newcommand\Si{\Sigma}
\newcommand\te{\theta}
\newcommand\vp{\varphi}

\newcommand\OV{\overline}
\newcommand\funnyk{k\hbox to 0pt{\hss\phantom{g}}}

\newcommand\lu[1]{L^1(#1)}
\newcommand\lp[1]{L^p(#1)}

\newcommand\ly[1]{L^\infty(#1)}

\newcommand\wh{\widehat}
\newcommand\wt{\widetilde}
\newcommand\whH{\widehat{\phantom{G}}\hbox to 0pt{\hss $H$}}

\newcommand\emspace{\hbox to 6pt{\hss}}
\newcommand\ds{\displaystyle}

\newcommand\rmi{\hbox{\rm (i)}}
\newcommand\rmii{\hbox{\rm (ii)}}
\newcommand\rmiii{\hbox{\rm (iii)}}
\newcommand\rmiv{\hbox{\rm (iv)}}

\newcommand\One{{\mathbf{1}}}

\newcommand\e{\mathrm{e}}

\newcommand\floor[1]{\lfloor{#1}\rfloor}

\DeclareSymbolFont{EUEX}{U}{euex}{m}{n}

\DeclareSymbolFont{euexlargesymbols}{U}{euex}{m}{n}
\DeclareMathSymbol{\intop}{\mathop}{euexlargesymbols}{"52}
     \def\int{\intop\nolimits}

\DeclareSymbolFont{euexsymbols}     {U}{euex}{m}{n}
\DeclareMathSymbol{\smallint}{\mathop}{euexsymbols}{"52}

\begin{document}

\title[Maximal operators]{Spider's webs and sharp $L^p$ bounds \\ for the Hardy--Littlewood maximal operator \\ on Gromov hyperbolic spaces}

%\thanks{\em{{2020 Mathematics Subject Classification:}} 05C05, 43A99} 


\thanks{
Chalmoukis and Santagati are members of the Gruppo Nazionale per l'Analisi Ma\-te\-matica, la Probabilit\`a e le loro Applicazioni (GNAMPA) of the
Istituto Nazionale di Alta Matematica (INdAM).
 Santagati was supported by the Australian Research Council, grant DP220100285.}

\author[N. Chalmoukis]{Nikolaos Chalmoukis}
\address[Nikolaos Chalmoukis]{Dipartimento di Matematica e Applicazioni \\ Universit\`a di Milano-Bicocca\\
via R.~Cozzi 53\\ I-20125 Milano\\ Italy}
\email{nikolaos.chalmoukis@unimib.it}

\author[S. Meda]{Stefano Meda}
\address[Stefano Meda]{Dipartimento di Matematica e Applicazioni \\ Universit\`a di Milano-Bicocca\\
via R.~Cozzi 53\\ I-20125 Milano\\ Italy}
\email{stefano.meda@unimib.it}

\author[F.\ Santagati]{Federico Santagati} \address[Federico Santagati]{ School of Mathematics and Statistics, University of New South Wales, Sydney,
Australia}
\email{f.santagati@unsw.edu.au}

\begin{abstract}
	In this paper we prove that if  $1<a\leq b<a^2$ and $X$ is a locally doubling $\de$-hyperbolic complete connected length metric measure space with 
	$(a,b)$-pinched exponential growth at infinity, then the centred Hardy--Littlewood maximal operator $\cM$ 
	is bounded on $\lp{X}$ for all $p>\tau$, and it is of weak type $(\tau,\tau)$, where $\tau := \log_ab$.  
	A key step in the proof is a new structural theorem for Gromov hyperbolic spaces with $(a,b)$-pinched exponential growth at infinity,
	consisting in a discretisation of $X$ by means of certain graphs, introduced in this paper and called spider's webs, 
	with ``good connectivity properties".  Our result applies to trees with bounded geometry, and 
	Cartan--Hadamard manifolds of pinched negative curvature, providing new boundedness results in these settings. 
    	The index $\tau$ is optimal in the sense that if $p<\tau$, then there exists $X$ satisfying the assumptions above such that $\cM$ 
	is not of weak type $(p,p)$.  Furthermore, if $b>a^2$, then there are examples of spaces~$X$ satisfying the assumptions above such that $\cM$ bounded 
	on $\lp{X}$ if and only if $p=\infty$.
\end{abstract}

%\date{\today, \thistime. Preliminary version}

\maketitle
\begin{center}
{\textsl{This paper is dedicated to the memory of Saverio Giulini}}
\end{center}
\vskip0.8cm
\section{Introduction}
\label{s: Introduction}

The purpose of this paper is to prove sharp $L^p$ bounds for the centred Hardy--Littlewood (HL) maximal operator $\cM$ on 
a comparatively large class of metric measure spaces, including all Cartan--Hadamard manifolds of pinched negative curvature and trees with
pinched exponential volume growth. 
Our main result, Theorem~\ref{t: main}, is stated below in this introduction.  A key ingredient in its proof is a discretisation procedure
of Gromov hyperbolic spaces with pinched exponential volume growth leading to certain graphs with ``good connectivity properties'', 
introduced in this paper and called \textit{spider's webs}.  One of the advantages of our strategy is that it relies on flexible techniques 
from Geometric Analysis that can be applied to quite diverse settings and hopefully to various related problems.

Suppose that $(X,d,\mu)$ is a connected metric measure space, and $\mu$ is a Borel measure.  Denote by $B_r(x)$ the open ball with centre $x$ and radius~$r$, 
i.e.  $B_r(x) := \{y\in X: d(x,y) < r\}$, and assume that the measure of each ball is positive and finite. 
For each locally integrable function $f$ on $X$, consider its \textit{centred} HL maximal function $\cM f$, defined by
$$
\cM f(x)
:= \sup_{r>0} \, \frac{1}{\mu\big(B_r(x)\big)} \, \int_{B_r(x)} \!\mod{f} \wrt \mu. 
$$
Clearly $\cM$ and the local version $\cM_0$ thereof, defined by
\begin{equation} \label{f: local max}
	\cM_0 f(x)
	:= \sup_{0<r\leq 1} \, \frac{1}{\mu\big(B_r(x)\big)} \, \int_{B_r(x)} \! \mod{f} \wrt \mu,
\end{equation}
are bounded on $\ly{X}$.  It is known, and not hard to see (see Proposition~\ref{p: weak cM0}), that if $\mu$ is locally doubling (see 
Subsection~\ref{ss: The local doubling property}), then $\cM_0$ is of weak type $(1,1)$, hence bounded on $\lp{X}$ for all $p$ in $(1,\infty]$.  
Thus, $\cM$ is bounded on $\lp{X}$ [resp. is of weak type $(p,p)$] for some $p$ in $(1,\infty)$ if and only if the ``global" maximal operator $\cM_\infty$, defined by 
\begin{equation} \label{f: global max}
	\cM_\infty f(x)
	:= \sup_{r> 1} \, \frac{1}{\mu\big(B_r(x)\big)} \, \int_{B_r(x)} \! \mod{f} \wrt \mu,
\end{equation}
is bounded on $\lp{X}$ [resp. is of weak type $(p,p)$].  

Denote by $J_{X}$ the interval of all $p$'s such that $\cM_\infty$ is bounded on $\lp{X}$.  
Standard arguments show that $J_{X} = (1,\infty]$, and $\cM_\infty$ is of weak type $(1,1)$, on all doubling metric measure spaces (see \cite[Theorem~2.2]{He}). 

The situation is less simple, but much more interesting, in the case where $\mu$ is locally doubling, but not doubling.   
The way $J_{X}$ depends on the properties of the metric measure space $(X,d,\mu)$ is rather subtle, and, we believe, not fully understood.  
We shall make some comments on this later in this introduction.  First, we describe some important contributions to the problem of determining $J_X$.   

The paper \cite{Str}, which has been a source of inspiration for our previous investigations on the subject \cite{CMS2, LMSV, MS, LS, MPSV}, is a landmark 
in this field.  J.-O.~Str\"omberg \cite[Theorem~p.~115]{Str} proved that if~$X$ is a symmetric space of the noncompact type, then $J_{X} = (1,\infty]$ 
and $\cM_\infty$ is of weak type $(1,1)$.  J.-Ph.~Anker, E.~Damek and C.~Yacoub  \cite[Corollary~3.22]{ADY} extended this result 
to Damek--Ricci spaces.
Str\"omberg also addressed the problem of determining $J_{X}$ on Riemannian surfaces $X$ with pinched negative 
Gaussian curvature, and stated the following result \cite[Remark~3, p.~126]{Str}: if $A$ and $B$ are positive numbers such that $A\leq B< 2A$, 
then there exists a perturbed hyperbolic metric on the upper half-plane with Gaussian curvature $K$ satisfying $-B^2\leq K\leq -A^2$ such that 
$J_{X} = (B/A,\infty]$: furthermore $\cM_\infty$ is  unbounded on $L^p(X)$ if $p<B/A$.  For a full proof of this result see \cite[Theorem~7.1]{MPSV},
where it is also shown that if $B>2A$, then $J_{X}$ is reduced to the point $\infty$.  %$\cM$ is bounded on~$L^p$ if and only if $p=\infty$. 

Str\"omberg's example leads naturally to conjecture that if $X$ is a Cartan--Hadamard Riemannian manifold with pinched negative
sectional curvature, i.e. $-B^2\leq K\leq -A^2$ for some positive constants~$A$ and $B$, and $A\leq B <2A$, then $J_{X}$ contains $(B/A,\infty]$.
Incidentally, this holds true (see Corollary~\ref{c: CH manifolds}), and follows from our main result.  
Note that compact perturbations of $X$ can significantly affect this kind of bounds on the sectional curvature, although 
it is reasonable to believe that the conclusion remains the same.  For this and related reasons it is desirable to rely on 
more ``robust" sets of assumptions than those involving curvature bounds.   

Other results somewhat related to the geometric framework considered by Str\"omberg can be found in \cite{L1, L2, MPSV} and in the papers cited therein.   

In particular, \cite{L1} and \cite{L2} focus on an interesting class of conic manifolds~$X$ of the form $X_0\times (0,\infty)$, where $X_0$ is a 
length metric measure space.  The distance~$d$ on~$X$ is defined much as the hyperbolic distance on the upper half-plane, with the metric $d_0$ on $X_0$ 
playing the role of the Euclidean distance on the real line (see \cite[formula (1)]{L1}) and the measure $\mu$ being the product of the measure $\mu_0$ 
on $X_0$ and of the measure $\wrt y/y^{N+1}$ on $(0,\infty)$ for some nonnegative constant~$N$.  H.-Q.~Li proves that $J_{X}$ is either $(p_0,\infty]$, 
where $p_0$ depends on the parameters describing the volume of balls on $X$, or is reduced to the point~$\infty$.  The index $p_0$ is optimal in the class of the 
conic manifolds considered.  Related results are contained in \cite{L2}.

The results in \cite{MPSV} corroborate the idea that appropriate ``perturbations'' of a given Riemannian metric preserve the $L^p$ boundedness properties of $\cM$. 
Specifically, \cite[Theorem~3.1]{MPSV} and \cite[Theorem~5.5]{MPSV} deal with conformal perturbations, and with the case of 
strictly quasi-isometric length spaces with ``controlled local geometry'', respectively.

Finally, in \cite{LMSV} the authors focus on the class $\Upsilon_{a,b}$ of all trees~$\mrT$ of \textit{$(a,b)$-bounded geometry}, i.e. trees in which every
vertex has at least $a+1$ and at most $b+1$ neighbours: the integers $a$ and $b$ are assumed to satisfy   
the inequality $2\leq a \leq b$.  They prove that if $b\leq a^2$, then $\cM$ is bounded on $\lp{\mrT}$ if $p>\tau$, where $\tau := \log_ab$ and it is of 
restricted weak type $(\tau,\tau)$.   The proof hinges on the sharp form of the Kunze--Stein phenomenon (see \cite[Theorem~1]{CMS1}).
Such trees may be considered as discrete counterparts of Riemannian manifolds with $(a,b)$-pinched exponential growth at infinity (see \eqref{f: pinched exp}
below).

It may be worth observing that the threshold index that we conjecture in the Riemannian case and that we found in this discrete setting agree.   
Indeed, on the one hand, the volume of balls in a tree $\mrT$ with $(a,b)$-bounded geometry satisfies the estimate
$$
a^r
\leq \mu_\mrT\big(B_r(x)\big) 
\leq 3\, b^r
\quant x \in \mrT \quant r \in \BN,
$$
where $\mu_\mrT$ denotes the counting measure on the vertices of $\mrT$.  
On the other hand, by comparison results \cite[Corollary~3.2~(ii)]{Sa}, if an $n$-dimensional Cartan--Hadamard Riemannian manifold $M$ satisfies $-B^2\leq K\leq -A^2$, then 
$$
c\, \e^{(n-1)Ar}
\leq \mu\big(B_r(x)\big) 
\leq C\, \e^{(n-1)Br}
\quant x \in M \quant r \geq 1,
$$
where $c$ and $C$ are positive constants, depending only on $n$, and $\mu$ denotes the Riemannian measure on $M$.
Thus, if we set $a := \e^{(n-1)A}$ and $b := \e^{(n-1)B}$, then the condition $B<2A$ transforms to $b<a^2$, the volume bounds in the discrete and in the
continuous case agree, and the threshold index $B/A$ in the continuous case can be written as $(\log b)/(\log a)$, which agrees with $\tau$.  

An interesting observation is that the results in \cite{MPSV, LMSV} are stable under \textit{strict rough isometries}
(see Definition~\ref{def: RI} below, \cite[Theorem~5.4]{LMSV} and \cite[Theorem~5.5]{MPSV}).   These are rough isometries in
the sense of M.~Kanai \cite{K} that preserve the exponential rate of the volume growth of balls when the radius tends to infinity.   
Rough isometries are also known as quasi-isometries (see, for instance, \cite[Section 7.2]{Gr}), and strict rough isometries are sometimes referred to
as $1$-quasi-isometries in the literature.  
This suggests that it may be worth looking for sets of assumptions that are ``stable" under such maps.   

Suppose that $\de$, $a$ and $b$ are nonnegative numbers satisfying the condition $1<a\leq b$.  We denote by $\cX_{a,b}^\de$ 
the class of all connected noncompact complete $\de$-hyperbolic length spaces $(X,d)$ (explicitly, $d$ is the strictly intrinsic metric associated 
to the length structure on $X$; in particular any pair of points in $X$ is connected by a shortest path; 
see \cite[Definition~8.4.1]{BBI}), endowed with a locally doubling Borel measure~$\mu$ 
with \textit{$(a,b)$-pinched exponential growth at infinity}.  By this we mean that for some~$a$ and $b$ with $1<a\leq b$ 
there exist positive constants $c$ and~$C$ such that 
\begin{equation} \label{f: pinched exp}
	c\, a^r \leq \mu\big(B_r(x)\big) \leq C \, b^r
	\quant x\in X \quant r \in [1,\infty).
\end{equation}

For notational convenience, set $\tau := \log_ab$.  Our main result is the following.

\begin{theorem} \label{t: main}
	Suppose that $1<a \leq b< a^2$ and that $\de$ is a nonnegative number.  If $X$ belongs to the class $\cX_{a,b}^\de$, then 
	the maximal operator $\cM$ is bounded on $\lp{X}$ for all $p>\tau$, and it is of weak type $(\tau,\tau)$.  
\end{theorem}

The index $\tau$ is optimal in the sense that if $\tau >1$ and $p<\tau$, then there exists a length space $X$ in the class $\cX_{a,b}^\de$ such that $\cM$ 
is not of weak type $(p,p)$  for every $p<\tau$.  For instance, one may take Str\"omberg's counterexample mentioned above as $X$.

If $b>a^2$, then the conclusion of Theorem~\ref{t: main} fails, for there
are examples of length spaces~$X$ in $\cX_{a,b}^\de$ such that $\cM$ is bounded on $\lp{X}$ if and only if $p=\infty$: see, for instance, \cite[Theorem~7.1~\rmii]{MPSV}.
Another example is given by the natural metric tree associated to the tree $\fS_{a,b}$, with $b>a^2$, considered in \cite[Proposition~3.3~\rmii]{LMSV}.  Such metric tree
(where each edge is isometric to the interval $[0,1]$) belongs to $\cX_{a,b}^0$, and $\cM_\infty$ is bounded on $\lp{\fS_{a,b}}$ if and only if $p=\infty$.

A noteworthy consequence of Theorem~\ref{t: main} is that if $M$ is a Cartan--Hadamard manifold with sectional curvatures $K_\pi$ satisfying the bound 
$$
-B^2 \leq K_\pi \leq -A^2
$$
where $0 < A \leq B < 2A$, then the centred HL maximal operator $\cM$ is bounded on $\lp{M}$ for all $p>B/A$, and it is 
of weak type $(B/A,B/A)$.

In particular, Theorem~\ref{t: main} improves Lohou\'e's results \cite{Lo} in arbitrary dimensions and, when applied to Str\"omberg's counterexample,
yields the endpoint result that $\cM$ is of weak type $(\tau,\tau)$, a fact which was previously unknown.

It is well-known that Damek--Ricci spaces $X$ are Gromov hyperbolic \cite[Theorem 4.8]{Kn} with pinched exponential growth, so that Theorem~\ref{t: main} implies 
the result of Anker et al. \cite[Corollary~3.22]{ADY} that $J_{X}=(1,\infty]$ and $\cM$ is of weak type $(1,1)$.  

Theorem~\ref{t: main} also improves previous results on trees with $(a,b)$-bounded geometry satisfying $a<b<a^2$ \cite[Theorem~3.2~\rmii]{LMSV} in two ways: 
it applies to a larger class of trees ($(a,b)$-bounded geometry implies $(a,b)$-pinched
volume exponential growth, but not conversely), and it yields a better endpoint estimate (weak type $(\tau,\tau)$ in place of restricted weak type $(\tau,\tau)$).

We observe also that higher rank noncompact symmetric spaces are not covered by Theorem~\ref{t: main}, because they are not Gromov hyperbolic.

Our analysis hinges on the (new) notion of spider's web, introduced in Definition~\ref{def: spider web}.  
Loosely speaking, a spider's web is a graph associated to a rooted tree that satisfies certain connectivity properties.  
The strategy of proof of Theorem~\ref{t: main} relies upon the following three facts:  
\begin{enumerate}
	\item[\itemno1]
		if $X$ belongs to $\cX_{a,b}^\de$, then $X$ is strictly roughly isometric to a spider's web~$\wh\Ga$, endowed with the graph distance $d_{\wh\Ga}$. 
		Moreover, the measure of the metric balls in $\wh\Ga$ (with respect to the counting measure on the set of its vertices) satisfies the bound~\eqref{f: pinched exp};
	\item[\itemno2]
		if $a\leq b<a^2$, then the maximal operator $\cM$ is bounded on $\lp{\wh\Ga}$ for $p>\tau$, and it is of weak type $(\tau,\tau)$;
	\item[\itemno3]
		the strict rough isometry in \rmi\ above can be used to ``transfer" to~$X$ the boundedness properties of $\cM$ on $\wh\Ga$ described in \rmii.  
\end{enumerate}

\noindent
It is already known that $\de$-hyperbolic spaces are strictly roughly isometric to graphs \cite[Section~5]{BI}.  However, we emphasise that it is highly nontrivial
to show that the approximating graph $\wh\Ga$ can be chosen to be a spider's web: this richer structure is needed in the proof of \rmii. 
Step \rmiii\ is a variant of \cite[Theorem~5.4]{LMSV} and of \cite[Theorem~5.5]{MPSV}, and reflects the fact that~$L^p$ boundedness properties of the maximal operator $\cM$
depend essentially on the ``large-scale geometry'' of $X$.

\smallskip
This paper is organised as follows.  Section~\ref{s: Preliminaries} is devoted to background material and preliminary results.
Section~\ref{s: Spider webs} contains more information and results concerning spider's webs associated to rooted trees.  
In Section~\ref{s: bounds on spider webs} we extend \cite[Theorem~3.2~\rmi]{LMSV} concerning the $L^p$-boundedness of the maximal operator $\cM$ on trees with
bounded geometry to spider's webs satisfying mild volume growth conditions.
In Section~\ref{s: Discretisation of Gromov} we prove that every connected noncompact complete $\de$-hyperbolic space $X$ (with distance $d$) is
strictly roughly isometric to a $\de$-hyperbolic spider's web $\wh\Ga$, with graph distance $d_{\wh\Ga}$.
Finally, in Section~\ref{s: main} we prove of our main result, and derive some consequences thereof.  

We use the ``variable constants convention'', and denote by~$c$ and $C$ constants whose value may vary from place to place and 
may depend on any factors quantified (implicitly or explicitly) before its occurrence, but not on factors quantified after.

\section[Background]{Background and preliminary results} \label{s: Preliminaries}
Suppose that $(X,d,\mu)$ is a metric measure space, where $\mu$ is a Borel measure on $(X,d)$, and denote by $\cB$ the family of all open balls in $X$.  
For each $B$ in $\cB$ and for any positive number $k$ we denote by $r_B$ the radius of~$B$ and by $k B$ the ball with the same centre as $B$ 
and radius $k r_B$.  For each $s$ in $\BR^+$, we denote by $\cB_s$ the family of all balls $B$ in~$\cB$ such that $r_B \leq s$.

\subsection{The local doubling property}
\label{ss: The local doubling property}
Assume that $0<\mu(B)<\infty$ for every~$B$ in $\cB$.  
We say that the metric measure space $X$ possesses the \emph{local doubling property} (LDP) if for every $s$ in $\BR^+$ there exists a constant $L_s$ such that
\begin{equation*}  \label{f: LDC} 
\mu \bigl(2 B\bigr)
\leq L_s \, \mu  \bigl(B\bigr)
\quant B \in \cB_s.
\end{equation*}

\begin{remark} \label{r: geom I}
	The LDP implies that for each $\tau \geq 1$ and for each $s$ in $\BR^+$ there exists a constant~$C$ such that
	\begin{equation} \label{f: doubling Dtau}
		\mu\bigl(B'\bigr)
		\leq C \, \mu(B)
	\end{equation}
	for each pair of balls $B$ and $B'$, with $B\subset B'$, $B$ in $\cB_s$, and $r_{B'}\leq \tau \, r_B$.  We shall denote by $L_{\tau,s}$ 
	the smallest constant for which (\ref{f: doubling Dtau}) holds.  
\end{remark}

\begin{remark} \label{r: geom II}
	Observe that if $X$ possesses the LDP and satisfies the growth condition~\eqref{f: pinched exp}, then the condition $0<\mu(B)<\infty$ is automatically 
	satisfied for every $B$ in $\cB$.  Indeed, the monotonicity of $\mu$ and the right hand inequality in \eqref{f: pinched exp} imply that $\mu(B)$ is 
	finite for every ball $B$.   

	Furthermore if $r_B > 1$, then $\mu(B) \geq c\, a^{r_B}>0$ by the left hand inequality in \eqref{f: pinched exp}.  If, instead $r_B<1$, then 
	$$
	c\, a \leq \mu\big((1/r_B) \, B\big)
	\leq L_{1/r_B,r_B}  \, \mu(B),
	$$
	as required.
\end{remark}

\subsection{Rough isometries}
A central notion in our investigation is the following.

\begin{definition} \label{def: RI}
Suppose that $X$ and $X'$ are two metric spaces, with distances $d$ and $d'$, respectively, and that $\la$ and $\be$ are nonnegative numbers with $\la\geq 1$.
A map $\vp : X \to X'$ is a $(\la,\be)$-rough isometry if 
$$
	\frac{1}{\la}\, d(x,y)-\be \le d'\big(\vp(x),\vp(y)\big) \le \la \, d(x,y)+\be
	\quant x,y \in X
$$
and 
\begin{equation} \label{f: RI uno}
	\sup\, \big\{d'\big(\vp(X),x'\big): x' \in X \big\}
	<\infty.
\end{equation}
If $\la = 1$, then $\vp$ is called a \textit{strict $\be$-rough isometry} (or simply a strict rough isometry): in this case 
\begin{equation} \label{f: RI due}
	d(x,y)-\be \le d'\big(\vp(x),\vp(y)\big) \le d(x,y)+\be
	\quant x,y \in X.
\end{equation}
\end{definition} 
As mentioned in the introduction, rough isometries are also known as quasi-isometries (see, for instance, \cite[Section 7.2]{Gr}).

\subsection{Gromov hyperbolicity}
Suppose that $(X,d)$ is a metric space.  The \textit{Gromov product} $(y,z)_x $ of two points $y$ and $z$ in $X$ with respect to a third point $x$ in $X$
is defined by the formula:
$$
(y,z)_x 
:= \frac{1}{2}\, \big(d(x,y)+d(x,z)-d(y,z)\big).
$$

\begin{definition} \label{def: de hyp metric space}
Suppose that $\de$ is a nonnegative number.  The metric space $(X,d)$ is $\de$-\textit{hyperbolic} if and only if the following \textit{four points condition} 
	is fulfilled: for every $x$, $y$, $z$ and $w$ in $X$
\begin{equation} \label{f: Gromov}
	(x,z)_w 
	\geq \min\big((x,y)_w,(y,z)_w\big)-\de.
\end{equation}
The space $(X,d)$ is called \textit{Gromov hyperbolic} if it is $\de$-hyperbolic for some $\de$.
\end{definition}

Note that if \eqref{f: Gromov} is satisfied for all points $x$, $y$ and $z$ and one fixed base point $w_0$, then it is satisfied for all base points $w$ 
with a constant $2\de$ (see, for instance, \cite[pp. 2--3]{CDP}).   Thus, in order to check that a space is Gromov hyperbolic, it suffices 
to check the hyperbolicity condition \eqref{f: Gromov} for one fixed base point. 

In the case where $X$ is a complete length space with strictly intrinsic metric $d$, it is well known that $(X,d)$ is $\de$-hyperbolic if and only if there exists $\de'$ such that every geodesic triangle in $X$ is $\de'$-\textit{slim}, i.e. each of its sides is contained in the $\de'$-neighbourhood of the union 
of the other two: see, for instance, \cite[Chapter~III, Proposition~1.22]{BH}.   

For a length space $X$ when we say that $X$ is $\de$-hyperbolic we refer to the constant~$\de$ appearing in the Definition~\ref{def: de hyp metric space}; 
if we want to refer to the constant $\de'$ that controls the slimness of triangles in $X$, we say that $X$ is a $\de'$-hyperbolic \textit{length} space.
In the cases where we are interested only in the slimness constant, we shall often call it $\de$ instead of $\de'$. 

\begin{remark} \label{rem: three points}
	We shall often use the following simple observation.  Suppose that $(X,d)$ is a $\de$-hyperbolic length space, and that $oxy$ is a geodesic
	triangle in $X$, with edges $[ox]$, $[oy]$ and $[xy]$.  Then there exists a point $p$ in $[xy]$ such that 
	$$
	\min\,\big(d(p,[ox]), d(p,[oy]\big) 
	< \de.
	$$

	Indeed, denote by $E$ and $F$ the set of all points in $[xy]$ at distance less than~$\de$ from $[ox]$ and from $[oy]$, respectively.  
	Observe that $E$ and $F$ are nonempty, for $E$ contains $x$ and $F$ contains $y$.  Since $X$ is a 
	$\de$-hyperbolic length space, the $\de$-neighbourhood of $[ox]\cup[oy]$ contains $[xy]$, so that  $E\cup F = [xy]$.  Since $[xy]$, $E$ and $F$ are open 
	(in the relative topology induced by $d$) and~$[xy]$ is connected, $E\cap F$ cannot be empty.  Any point in $E\cap F$ satisfies the inequality above.   
\end{remark}

	\begin{definition} \label{def: quasi-geodesic}
		Suppose that $(X,d)$ is a length space and that $\om$ and $K$ are nonnegative numbers, with $\om\geq 1$.  We say that a path $\ga$ in $X$
		is an $(\om,K)$-\textit{quasi-geodesic} if for every pair of points $p$ and $q$ in $\ga$
		$$
		\ell_X([pq]) 
		\leq \om \, d(p,q) + K;
		$$
		here $\ell_X([pq])$ denotes the length in $(X,d)$ of the segment in $\ga$ with endpoints~$p$ and $q$.  
		In particular, $(1,0)$-quasi-geodesics are just shortest paths.   

	Assume, in addition, that $(X,d)$ is a $\de$-hyperbolic length space. We say that $X$ is \textit{geodesically stable}
		if for every $\om$ in $[1,\infty)$ and each $K \ge 0$ there exists a constant $D_{\om,K}$, depending on $\de$, $\om$, and $K$ 
		such that each $(\om,K)$-quasi-geodesic $\ga$ belongs to the~$D_{\om,K}$-neighbourhood of any geodesic in $X$ joining 
		the endpoints of $\ga$.
\end{definition}

Next, we recall the well-known Morse Lemma (see, for instance, \cite[Proposition  3.1]{Bo}).

\begin{lemma} \label{lem: Morse}
	Any $\de$-hyperbolic length space is {geodesically stable}.
\end{lemma}

One of the nice features of Gromov hyperbolicity for length spaces is that it is preserved under rough isometric embeddings.  Recall the following result
(for its proof see, for instance, \cite[Chapter~III, Theorem~1.9]{BH}).

\begin{theorem}  \label{t: rqi hyp}
	Suppose that $X$ and $X'$ are complete length metric spaces and that $\vp: X' \to X$ is a $(\la,\be)$-roughly isometric embedding. 
	If $X$ is a $\de$-hyperbolic length space, then $X'$ is a $\de'$-hyperbolic length space, where $\de'$ depends on $\de$, $\la$ and $\be$. 
\end{theorem}

It may be worth observing that the conclusion of Theorem~\ref{t: rqi hyp} fails if we omit the assumption that $X$ is a length space: see \cite[Example~13, p.~89]{dLHG}.  

\subsection{Graphs, trees and spider's webs}  \label{ss: Metric graphs}
Suppose that $\Ga$ is an undirected connected graph (without self-loops and multi-edges).   
If $x$ and $y$ are vertices in $\Ga$ connected by an edge, then we say that $x$ and~$y$ are \textit{neighbours}, and write $x\sim y$.  We shall always denote
by $\mu_\Ga$ the \textit{counting measure} on the vertices of $\Ga$.   

In this paper we shall consider only graphs $\Ga$ such that for each vertex $v$ in $\Ga$ the number $\nu(v)$ of its neighbours is finite:
$\nu(v)$ is called the \textit{valence} of~$v$.  If
\begin{equation} \label{f: bounded valence}
	\sup\,\{\nu(v): v \in \Ga\} < \infty,
\end{equation}
then we say that $\Ga$ has \textit{bounded valence}.  

A path $\ga$ in $\Ga$ joining two vertices $v$ and $w$ is a finite sequence $[x_0 := v, x_1.\ldots, x_{N-1}, x_N =: w]$  of vertices
such that $x_j\sim x_{j+1}$ for every $j$ in $\{0,\ldots,N-1\}$.  The \textit{length} of $\ga$ is defined to be $N$.  Since $\Ga$ is connected, 
given two vertices~$v$ and $w$, there exists at least one path joining them.  The graph distance $d_\Ga$ between $v$ and~$w$ is just the minimum 
of the lengths of all paths joining~$v$ and $w$.  

A \textit{tree} $\mrT$ is just a graph as above without loops.  A \textit{rooted} tree is a tree $\mrT$ with a distinguished point $o$, called the root of the tree.

We now introduce the notion of spider's web, which is central to our investigation.   Consider a rooted tree $\mrT$ with root $o$.  
Given a vertex~$x$ in $\mrT\setminus\{o\}$, we write $p(x)$ for its predecessor, i.e. the unique neighbour~$y$ of $x$ such that $d_\mrT(y,o) = d_\mrT(x,o)-1$.  
For every $x$ in $\mrT$ we  set $p^0(x):=x$ and for every positive integer $k$ we inductively define $p^k(x)=p(p^{k-1}(x))$.
The (possibly empty) set of the neighbours of~$x$ such that $d_\mrT(y,o) = d_\mrT(x,o)+1$ is denoted by $s(x)$.  Each vertex in $s(x)$ (if any) 
is called a \textit{successor} of $x$.
Similarly, for every $x$ in $\mrT$ we set $s^1(x):=s(x)$ and, for every $r \ge 2$, 
$$
s^r(x)
:= \ds\bigcup_{y \in s^{r-1}(x)} s(y).
$$

Suppose that $k$ is a nonnegative integer.  We set 
\begin{equation} \label{f: Sik}
\Sigma_k
:= \big\{x \in \mrT: d_\mrT(x,o) = k \big\}.\end{equation}
$\Si_k$ is called the \textit{sphere} with centre $o$ and radius $k$ in the rooted tree $\mrT$.  We say that each point in $\Si_k$ has \textit{level} $k$.
The \textit{level} of a vertex $x$ is denoted by~$h(x)$.  Note that $h(x)$ increases as we move away from~$o$
in $\mrT$.  

\begin{definition} \label{def: spider web}
	A \textit{spider's web} $\wh\Ga$ is a graph without self-loops obtained from a rooted tree $\mathrm{T}$ by adding a (possibly empty) 
	new set of edges, according to the following rules:
	\begin{enumerate}
		\item[\itemno1] 
			two vertices from different tree levels are never joined in $\wh\Ga$ by a new edge;
		\item[\itemno2] 
			if two  vertices of the same level are connected by an edge, then their predecessors either coincide or they are neighbours.
	\end{enumerate}
	A \textit{quasi-spider's web} $\Ga$ is a graph defined much as a spider's web, with the only difference that condition \rmii\ above is 
	replaced by the following weaker condition:
	\begin{enumerate}
		\item[\itemno2$'$] 
			there exists a positive integer $m$ such that if any two vertices of $\Si_n$, with $n\geq m$ are connected by an edge, then for every integer~$k$ 
			in $\{m, \ldots,n\}$, their $k^{\mathrm{th}}$ predecessors either coincide or they are neighbours.
	\end{enumerate}
\end{definition}

Note that a rooted tree is a very special spider's web in which  vertices belonging to the same level are never connected by an edge. 
We warn the reader that a spider's web may very well be a non-planar graph.

The prototype of spider's web is the so-called dyadic spider's web $\wh\Ga_2$, which we briefly discuss now.  Let $\BF_2$ the free monoid on two generators, i.e. the set of 
finite length words with two letters $0$ and $1$.  Let an edge connect two words if and only if one is obtained by the other by adding one letter in the end (tree edges) or if 
two words have the same length and one is the immediate successor of the other in the lexicographic order or if one is the all $0$’s and the other is the all $1$’s word.

\tikzset{every picture/.style={line width=0.75pt}} %set default line width to 0.75pt        
\begin{figure}
    \centering
\begin{tikzpicture}[x=0.75pt,y=0.75pt,yscale=-1,xscale=1, scale=0.6]
%uncomment if require: \path (0,503); %set diagram left start at 0, and has height of 503

%Shape: Ellipse [id:dp01719296094624434] 
\draw   (385.12,222.69) .. controls (385.12,160.22) and (434.98,109.59) .. (496.49,109.59) .. controls (558,109.59) and (607.86,160.22) .. (607.86,222.69) .. controls (607.86,285.15) and (558,335.78) .. (496.49,335.78) .. controls (434.98,335.78) and (385.12,285.15) .. (385.12,222.69) -- cycle ;
%Shape: Ellipse [id:dp8878001931150639] 
\draw   (414.55,222.69) .. controls (414.55,176.73) and (451.23,139.47) .. (496.49,139.47) .. controls (541.75,139.47) and (578.43,176.73) .. (578.43,222.69) .. controls (578.43,268.64) and (541.75,305.9) .. (496.49,305.9) .. controls (451.23,305.9) and (414.55,268.64) .. (414.55,222.69) -- cycle ;
%Flowchart: Connector [id:dp8867926700556081] 
\draw  [color={rgb, 255:red, 0; green, 0; blue, 0 }  ,draw opacity=1 ][fill={rgb, 255:red, 0; green, 0; blue, 0 }  ,fill opacity=1 ] (493.41,222.69) .. controls (493.41,221.03) and (494.79,219.69) .. (496.49,219.69) .. controls (498.19,219.69) and (499.57,221.03) .. (499.57,222.69) .. controls (499.57,224.34) and (498.19,225.68) .. (496.49,225.68) .. controls (494.79,225.68) and (493.41,224.34) .. (493.41,222.69) -- cycle ;
%Flowchart: Connector [id:dp9480803744525813] 
\draw  [color={rgb, 255:red, 0; green, 0; blue, 0 }  ,draw opacity=1 ][fill={rgb, 255:red, 0; green, 0; blue, 0 }  ,fill opacity=1 ] (450.17,222.69) .. controls (450.17,221.03) and (451.55,219.69) .. (453.25,219.69) .. controls (454.95,219.69) and (456.32,221.03) .. (456.32,222.69) .. controls (456.32,224.34) and (454.95,225.68) .. (453.25,225.68) .. controls (451.55,225.68) and (450.17,224.34) .. (450.17,222.69) -- cycle ;
%Flowchart: Connector [id:dp7058589299851246] 
\draw  [color={rgb, 255:red, 0; green, 0; blue, 0 }  ,draw opacity=1 ][fill={rgb, 255:red, 0; green, 0; blue, 0 }  ,fill opacity=1 ] (536.66,222.69) .. controls (536.66,221.03) and (538.04,219.69) .. (539.74,219.69) .. controls (541.44,219.69) and (542.81,221.03) .. (542.81,222.69) .. controls (542.81,224.34) and (541.44,225.68) .. (539.74,225.68) .. controls (538.04,225.68) and (536.66,224.34) .. (536.66,222.69) -- cycle ;
%Flowchart: Connector [id:dp36478344261180384] 
\draw  [color={rgb, 255:red, 0; green, 0; blue, 0 }  ,draw opacity=1 ][fill={rgb, 255:red, 0; green, 0; blue, 0 }  ,fill opacity=1 ] (537.59,119.45) .. controls (537.59,117.8) and (538.96,116.46) .. (540.66,116.46) .. controls (542.36,116.46) and (543.74,117.8) .. (543.74,119.45) .. controls (543.74,121.11) and (542.36,122.44) .. (540.66,122.44) .. controls (538.96,122.44) and (537.59,121.11) .. (537.59,119.45) -- cycle ;
%Flowchart: Connector [id:dp5631983528371842] 
\draw  [color={rgb, 255:red, 0; green, 0; blue, 0 }  ,draw opacity=1 ][fill={rgb, 255:red, 0; green, 0; blue, 0 }  ,fill opacity=1 ] (391.93,174.48) .. controls (391.93,172.83) and (393.31,171.49) .. (395.01,171.49) .. controls (396.71,171.49) and (398.09,172.83) .. (398.09,174.48) .. controls (398.09,176.13) and (396.71,177.47) .. (395.01,177.47) .. controls (393.31,177.47) and (391.93,176.13) .. (391.93,174.48) -- cycle ;
%Flowchart: Connector [id:dp506749290963179] 
\draw  [color={rgb, 255:red, 0; green, 0; blue, 0 }  ,draw opacity=1 ][fill={rgb, 255:red, 0; green, 0; blue, 0 }  ,fill opacity=1 ] (592.6,274.21) .. controls (592.6,272.56) and (593.98,271.22) .. (595.68,271.22) .. controls (597.38,271.22) and (598.76,272.56) .. (598.76,274.21) .. controls (598.76,275.86) and (597.38,277.2) .. (595.68,277.2) .. controls (593.98,277.2) and (592.6,275.86) .. (592.6,274.21) -- cycle ;
%Flowchart: Connector [id:dp7182047133959191] 
\draw  [color={rgb, 255:red, 0; green, 0; blue, 0 }  ,draw opacity=1 ][fill={rgb, 255:red, 0; green, 0; blue, 0 }  ,fill opacity=1 ] (531.45,328.07) .. controls (531.45,326.41) and (532.83,325.08) .. (534.53,325.08) .. controls (536.23,325.08) and (537.61,326.41) .. (537.61,328.07) .. controls (537.61,329.72) and (536.23,331.06) .. (534.53,331.06) .. controls (532.83,331.06) and (531.45,329.72) .. (531.45,328.07) -- cycle ;
%Flowchart: Connector [id:dp13177895687771146] 
\draw  [color={rgb, 255:red, 0; green, 0; blue, 0 }  ,draw opacity=1 ][fill={rgb, 255:red, 0; green, 0; blue, 0 }  ,fill opacity=1 ] (448.89,119.3) .. controls (448.89,117.65) and (450.27,116.31) .. (451.97,116.31) .. controls (453.67,116.31) and (455.05,117.65) .. (455.05,119.3) .. controls (455.05,120.95) and (453.67,122.29) .. (451.97,122.29) .. controls (450.27,122.29) and (448.89,120.95) .. (448.89,119.3) -- cycle ;
%Flowchart: Connector [id:dp11596252284842401] 
\draw  [color={rgb, 255:red, 0; green, 0; blue, 0 }  ,draw opacity=1 ][fill={rgb, 255:red, 0; green, 0; blue, 0 }  ,fill opacity=1 ] (449.85,327.22) .. controls (449.85,325.57) and (451.23,324.23) .. (452.93,324.23) .. controls (454.63,324.23) and (456.01,325.57) .. (456.01,327.22) .. controls (456.01,328.88) and (454.63,330.22) .. (452.93,330.22) .. controls (451.23,330.22) and (449.85,328.88) .. (449.85,327.22) -- cycle ;
%Flowchart: Connector [id:dp9789846528909782] 
\draw  [color={rgb, 255:red, 0; green, 0; blue, 0 }  ,draw opacity=1 ][fill={rgb, 255:red, 0; green, 0; blue, 0 }  ,fill opacity=1 ] (550.68,283.1) .. controls (550.68,281.45) and (552.06,280.11) .. (553.76,280.11) .. controls (555.46,280.11) and (556.83,281.45) .. (556.83,283.1) .. controls (556.83,284.75) and (555.46,286.09) .. (553.76,286.09) .. controls (552.06,286.09) and (550.68,284.75) .. (550.68,283.1) -- cycle ;
%Flowchart: Connector [id:dp7020074884340285] 
\draw  [color={rgb, 255:red, 0; green, 0; blue, 0 }  ,draw opacity=1 ][fill={rgb, 255:red, 0; green, 0; blue, 0 }  ,fill opacity=1 ] (551.34,164.09) .. controls (551.34,162.44) and (552.71,161.1) .. (554.41,161.1) .. controls (556.11,161.1) and (557.49,162.44) .. (557.49,164.09) .. controls (557.49,165.74) and (556.11,167.08) .. (554.41,167.08) .. controls (552.71,167.08) and (551.34,165.74) .. (551.34,164.09) -- cycle ;
%Flowchart: Connector [id:dp20877409860252327] 
\draw  [color={rgb, 255:red, 0; green, 0; blue, 0 }  ,draw opacity=1 ][fill={rgb, 255:red, 0; green, 0; blue, 0 }  ,fill opacity=1 ] (434.82,282.52) .. controls (434.82,280.87) and (436.19,279.53) .. (437.89,279.53) .. controls (439.59,279.53) and (440.97,280.87) .. (440.97,282.52) .. controls (440.97,284.18) and (439.59,285.52) .. (437.89,285.52) .. controls (436.19,285.52) and (434.82,284.18) .. (434.82,282.52) -- cycle ;
%Flowchart: Connector [id:dp12417425795820736] 
\draw  [color={rgb, 255:red, 0; green, 0; blue, 0 }  ,draw opacity=1 ][fill={rgb, 255:red, 0; green, 0; blue, 0 }  ,fill opacity=1 ] (435.45,165.09) .. controls (435.45,163.43) and (436.83,162.09) .. (438.53,162.09) .. controls (440.23,162.09) and (441.6,163.43) .. (441.6,165.09) .. controls (441.6,166.74) and (440.23,168.08) .. (438.53,168.08) .. controls (436.83,168.08) and (435.45,166.74) .. (435.45,165.09) -- cycle ;
%Flowchart: Connector [id:dp08450911303369324] 
\draw  [color={rgb, 255:red, 0; green, 0; blue, 0 }  ,draw opacity=1 ][fill={rgb, 255:red, 0; green, 0; blue, 0 }  ,fill opacity=1 ] (593.91,173.48) .. controls (593.91,171.83) and (595.29,170.49) .. (596.99,170.49) .. controls (598.69,170.49) and (600.07,171.83) .. (600.07,173.48) .. controls (600.07,175.14) and (598.69,176.48) .. (596.99,176.48) .. controls (595.29,176.48) and (593.91,175.14) .. (593.91,173.48) -- cycle ;
%Straight Lines [id:da11027636288458997] 
\draw    (453.25,222.69) -- (539.74,222.69) ;
%Straight Lines [id:da20238583794364273] 
\draw    (554.41,167.08) -- (539.74,222.69) ;
%Straight Lines [id:da508168558231125] 
\draw    (395.01,174.48) -- (438.53,165.09) ;
%Shape: Free Drawing [id:dp8847542366970844] 
\draw  [color={rgb, 255:red, 255; green, 255; blue, 255 }  ,draw opacity=1 ][line width=4.5] [line join = round][line cap = round] (456.72,237.18) .. controls (456.28,235.41) and (454.91,233.62) .. (455.41,231.86) .. controls (455.81,230.43) and (456.78,234.55) .. (457.04,236.02) .. controls (457.52,238.67) and (455.13,233.88) .. (455.9,234.85) .. controls (456.25,235.3) and (456.46,235.84) .. (456.72,236.35) .. controls (456.91,236.75) and (456.28,235.57) .. (456.06,235.18) ;
%Shape: Free Drawing [id:dp07048420271401856] 
\draw  [color={rgb, 255:red, 255; green, 255; blue, 255 }  ,draw opacity=1 ][line width=4.5] [line join = round][line cap = round] (457.37,232.19) .. controls (456.48,231.11) and (454.62,229.8) .. (455.24,228.54) ;
%Shape: Free Drawing [id:dp15421905768324184] 
\draw  [color={rgb, 255:red, 255; green, 255; blue, 255 }  ,draw opacity=1 ][line width=0.75] [line join = round][line cap = round] (454.75,228.54) .. controls (454.36,227.74) and (453.17,227.61) .. (454.1,227.37) ;
%Shape: Free Drawing [id:dp77751044211415] 
\draw  [color={rgb, 255:red, 255; green, 255; blue, 255 }  ,draw opacity=1 ][line width=1.5] [line join = round][line cap = round] (538.07,234.19) .. controls (538.28,233.3) and (538.08,232.17) .. (538.72,231.53) .. controls (539.11,231.13) and (538.97,233.75) .. (539.05,233.19) .. controls (539.29,231.55) and (539.16,229.87) .. (539.21,228.2) ;
%Shape: Free Drawing [id:dp9431879701343574] 
\draw  [color={rgb, 255:red, 255; green, 255; blue, 255 }  ,draw opacity=1 ][line width=1.5] [line join = round][line cap = round] (454.04,228.35) .. controls (454.04,227.96) and (454.04,227.57) .. (454.04,227.18) ;
%Shape: Free Drawing [id:dp9932538405479292] 
\draw  [color={rgb, 255:red, 255; green, 255; blue, 255 }  ,draw opacity=1 ][line width=1.5] [line join = round][line cap = round] (539.48,228.85) .. controls (539.43,228.24) and (539.38,227.63) .. (539.32,227.02) ;
%Shape: Ellipse [id:dp8964649826975162] 
\draw   (346.36,222.69) .. controls (346.36,138.49) and (413.58,70.23) .. (496.49,70.23) .. controls (579.4,70.23) and (646.62,138.49) .. (646.62,222.69) .. controls (646.62,306.88) and (579.4,375.14) .. (496.49,375.14) .. controls (413.58,375.14) and (346.36,306.88) .. (346.36,222.69) -- cycle ;
%Flowchart: Connector [id:dp0032953653883435186] 
\draw  [color={rgb, 255:red, 0; green, 0; blue, 0 }  ,draw opacity=1 ][fill={rgb, 255:red, 0; green, 0; blue, 0 }  ,fill opacity=1 ] (640.95,192.43) .. controls (640.95,190.78) and (642.32,189.44) .. (644.02,189.44) .. controls (645.72,189.44) and (647.1,190.78) .. (647.1,192.43) .. controls (647.1,194.09) and (645.72,195.43) .. (644.02,195.43) .. controls (642.32,195.43) and (640.95,194.09) .. (640.95,192.43) -- cycle ;
%Flowchart: Connector [id:dp7754031052554504] 
\draw  [color={rgb, 255:red, 0; green, 0; blue, 0 }  ,draw opacity=1 ][fill={rgb, 255:red, 0; green, 0; blue, 0 }  ,fill opacity=1 ] (637.45,267.34) .. controls (637.45,265.69) and (638.83,264.35) .. (640.53,264.35) .. controls (642.23,264.35) and (643.61,265.69) .. (643.61,267.34) .. controls (643.61,268.99) and (642.23,270.33) .. (640.53,270.33) .. controls (638.83,270.33) and (637.45,268.99) .. (637.45,267.34) -- cycle ;
%Flowchart: Connector [id:dp3026976636285984] 
\draw  [color={rgb, 255:red, 0; green, 0; blue, 0 }  ,draw opacity=1 ][fill={rgb, 255:red, 0; green, 0; blue, 0 }  ,fill opacity=1 ] (616.5,311.22) .. controls (616.5,309.57) and (617.88,308.23) .. (619.58,308.23) .. controls (621.28,308.23) and (622.66,309.57) .. (622.66,311.22) .. controls (622.66,312.88) and (621.28,314.22) .. (619.58,314.22) .. controls (617.88,314.22) and (616.5,312.88) .. (616.5,311.22) -- cycle ;
%Flowchart: Connector [id:dp46079733689273317] 
\draw  [color={rgb, 255:red, 0; green, 0; blue, 0 }  ,draw opacity=1 ][fill={rgb, 255:red, 0; green, 0; blue, 0 }  ,fill opacity=1 ] (572.11,352.4) .. controls (572.11,350.75) and (573.49,349.41) .. (575.19,349.41) .. controls (576.89,349.41) and (578.27,350.75) .. (578.27,352.4) .. controls (578.27,354.05) and (576.89,355.39) .. (575.19,355.39) .. controls (573.49,355.39) and (572.11,354.05) .. (572.11,352.4) -- cycle ;
%Flowchart: Connector [id:dp18577350441638296] 
\draw  [color={rgb, 255:red, 0; green, 0; blue, 0 }  ,draw opacity=1 ][fill={rgb, 255:red, 0; green, 0; blue, 0 }  ,fill opacity=1 ] (518.29,374.61) .. controls (518.29,372.96) and (519.67,371.62) .. (521.37,371.62) .. controls (523.07,371.62) and (524.45,372.96) .. (524.45,374.61) .. controls (524.45,376.26) and (523.07,377.6) .. (521.37,377.6) .. controls (519.67,377.6) and (518.29,376.26) .. (518.29,374.61) -- cycle ;
%Flowchart: Connector [id:dp4998512789482459] 
\draw  [color={rgb, 255:red, 0; green, 0; blue, 0 }  ,draw opacity=1 ][fill={rgb, 255:red, 0; green, 0; blue, 0 }  ,fill opacity=1 ] (462.38,371.68) .. controls (462.38,370.03) and (463.76,368.69) .. (465.46,368.69) .. controls (467.16,368.69) and (468.53,370.03) .. (468.53,371.68) .. controls (468.53,373.34) and (467.16,374.67) .. (465.46,374.67) .. controls (463.76,374.67) and (462.38,373.34) .. (462.38,371.68) -- cycle ;
%Flowchart: Connector [id:dp20683648222557827] 
\draw  [color={rgb, 255:red, 0; green, 0; blue, 0 }  ,draw opacity=1 ][fill={rgb, 255:red, 0; green, 0; blue, 0 }  ,fill opacity=1 ] (464.61,74.09) .. controls (464.61,72.43) and (465.98,71.09) .. (467.68,71.09) .. controls (469.38,71.09) and (470.76,72.43) .. (470.76,74.09) .. controls (470.76,75.74) and (469.38,77.08) .. (467.68,77.08) .. controls (465.98,77.08) and (464.61,75.74) .. (464.61,74.09) -- cycle ;
%Flowchart: Connector [id:dp39450196026031914] 
\draw  [color={rgb, 255:red, 0; green, 0; blue, 0 }  ,draw opacity=1 ][fill={rgb, 255:red, 0; green, 0; blue, 0 }  ,fill opacity=1 ] (522.88,72.09) .. controls (522.88,70.44) and (524.25,69.1) .. (525.95,69.1) .. controls (527.65,69.1) and (529.03,70.44) .. (529.03,72.09) .. controls (529.03,73.74) and (527.65,75.08) .. (525.95,75.08) .. controls (524.25,75.08) and (522.88,73.74) .. (522.88,72.09) -- cycle ;
%Flowchart: Connector [id:dp3046842833715412] 
\draw  [color={rgb, 255:red, 0; green, 0; blue, 0 }  ,draw opacity=1 ][fill={rgb, 255:red, 0; green, 0; blue, 0 }  ,fill opacity=1 ] (411.57,96.03) .. controls (411.57,94.37) and (412.95,93.04) .. (414.65,93.04) .. controls (416.35,93.04) and (417.73,94.37) .. (417.73,96.03) .. controls (417.73,97.68) and (416.35,99.02) .. (414.65,99.02) .. controls (412.95,99.02) and (411.57,97.68) .. (411.57,96.03) -- cycle ;
%Flowchart: Connector [id:dp10893315548818072] 
\draw  [color={rgb, 255:red, 0; green, 0; blue, 0 }  ,draw opacity=1 ][fill={rgb, 255:red, 0; green, 0; blue, 0 }  ,fill opacity=1 ] (615.19,134.59) .. controls (615.19,132.94) and (616.57,131.6) .. (618.27,131.6) .. controls (619.97,131.6) and (621.35,132.94) .. (621.35,134.59) .. controls (621.35,136.24) and (619.97,137.58) .. (618.27,137.58) .. controls (616.57,137.58) and (615.19,136.24) .. (615.19,134.59) -- cycle ;
%Flowchart: Connector [id:dp8496697001119808] 
\draw  [color={rgb, 255:red, 0; green, 0; blue, 0 }  ,draw opacity=1 ][fill={rgb, 255:red, 0; green, 0; blue, 0 }  ,fill opacity=1 ] (576.56,95.36) .. controls (576.56,93.71) and (577.94,92.37) .. (579.64,92.37) .. controls (581.34,92.37) and (582.72,93.71) .. (582.72,95.36) .. controls (582.72,97.01) and (581.34,98.35) .. (579.64,98.35) .. controls (577.94,98.35) and (576.56,97.01) .. (576.56,95.36) -- cycle ;
%Flowchart: Connector [id:dp27493521026140544] 
\draw  [color={rgb, 255:red, 0; green, 0; blue, 0 }  ,draw opacity=1 ][fill={rgb, 255:red, 0; green, 0; blue, 0 }  ,fill opacity=1 ] (394.55,274.21) .. controls (394.55,272.56) and (395.93,271.22) .. (397.63,271.22) .. controls (399.33,271.22) and (400.7,272.56) .. (400.7,274.21) .. controls (400.7,275.86) and (399.33,277.2) .. (397.63,277.2) .. controls (395.93,277.2) and (394.55,275.86) .. (394.55,274.21) -- cycle ;
%Flowchart: Connector [id:dp4561984779352096] 
\draw  [color={rgb, 255:red, 0; green, 0; blue, 0 }  ,draw opacity=1 ][fill={rgb, 255:red, 0; green, 0; blue, 0 }  ,fill opacity=1 ] (369.67,137.91) .. controls (369.67,136.26) and (371.05,134.92) .. (372.75,134.92) .. controls (374.45,134.92) and (375.82,136.26) .. (375.82,137.91) .. controls (375.82,139.57) and (374.45,140.91) .. (372.75,140.91) .. controls (371.05,140.91) and (369.67,139.57) .. (369.67,137.91) -- cycle ;
%Flowchart: Connector [id:dp8814562407118621] 
\draw  [color={rgb, 255:red, 0; green, 0; blue, 0 }  ,draw opacity=1 ][fill={rgb, 255:red, 0; green, 0; blue, 0 }  ,fill opacity=1 ] (346.1,190.44) .. controls (346.1,188.79) and (347.48,187.45) .. (349.18,187.45) .. controls (350.88,187.45) and (352.25,188.79) .. (352.25,190.44) .. controls (352.25,192.09) and (350.88,193.43) .. (349.18,193.43) .. controls (347.48,193.43) and (346.1,192.09) .. (346.1,190.44) -- cycle ;
%Flowchart: Connector [id:dp711060233898007] 
\draw  [color={rgb, 255:red, 0; green, 0; blue, 0 }  ,draw opacity=1 ][fill={rgb, 255:red, 0; green, 0; blue, 0 }  ,fill opacity=1 ] (348.72,257.59) .. controls (348.72,255.94) and (350.1,254.6) .. (351.8,254.6) .. controls (353.5,254.6) and (354.87,255.94) .. (354.87,257.59) .. controls (354.87,259.24) and (353.5,260.58) .. (351.8,260.58) .. controls (350.1,260.58) and (348.72,259.24) .. (348.72,257.59) -- cycle ;
%Flowchart: Connector [id:dp9551985958925937] 
\draw  [color={rgb, 255:red, 0; green, 0; blue, 0 }  ,draw opacity=1 ][fill={rgb, 255:red, 0; green, 0; blue, 0 }  ,fill opacity=1 ] (368.95,309.05) .. controls (368.95,307.4) and (370.33,306.06) .. (372.03,306.06) .. controls (373.73,306.06) and (375.1,307.4) .. (375.1,309.05) .. controls (375.1,310.7) and (373.73,312.04) .. (372.03,312.04) .. controls (370.33,312.04) and (368.95,310.7) .. (368.95,309.05) -- cycle ;
%Flowchart: Connector [id:dp674744025161846] 
\draw  [color={rgb, 255:red, 0; green, 0; blue, 0 }  ,draw opacity=1 ][fill={rgb, 255:red, 0; green, 0; blue, 0 }  ,fill opacity=1 ] (405.03,347.35) .. controls (405.03,345.7) and (406.4,344.36) .. (408.1,344.36) .. controls (409.8,344.36) and (411.18,345.7) .. (411.18,347.35) .. controls (411.18,349) and (409.8,350.34) .. (408.1,350.34) .. controls (406.4,350.34) and (405.03,349) .. (405.03,347.35) -- cycle ;
%Straight Lines [id:da25946604128423423] 
\draw    (414.65,96.03) -- (451.97,119.3) ;
%Straight Lines [id:da9592096478339842] 
\draw    (467.68,74.09) -- (451.97,119.3) ;
%Straight Lines [id:da5955871437260584] 
\draw    (451.97,119.3) -- (438.53,165.09) ;
%Straight Lines [id:da09954455829595388] 
\draw    (438.53,165.09) -- (453.25,222.69) ;
%Straight Lines [id:da5925358908926945] 
\draw    (453.25,222.69) -- (438.57,278.29) -- (437.89,282.52) ;
%Straight Lines [id:da2728260994439249] 
\draw    (539.74,219.69) -- (553.76,283.1) ;
%Straight Lines [id:da35618813418633155] 
\draw    (596.99,173.48) -- (644.02,192.43) ;
%Straight Lines [id:da4005194088501758] 
\draw    (618.27,134.59) -- (596.99,173.48) ;
%Straight Lines [id:da6714737103666741] 
\draw    (576.56,95.36) -- (540.66,119.45) ;
%Straight Lines [id:da07102573992819239] 
\draw    (525.95,72.09) -- (540.66,119.45) ;
%Straight Lines [id:da8282590957396438] 
\draw    (596.99,173.48) -- (554.41,164.09) ;
%Straight Lines [id:da7890748451263029] 
\draw    (540.66,119.45) -- (554.41,164.09) ;
%Straight Lines [id:da5750653303096832] 
\draw    (595.68,274.21) -- (553.76,283.1) ;
%Straight Lines [id:da45855689062404204] 
\draw    (534.53,328.07) -- (575.19,352.4) ;
%Straight Lines [id:da2654862393064501] 
\draw    (553.76,283.1) -- (534.53,328.07) ;
%Straight Lines [id:da8033350997879224] 
\draw    (452.93,328.55) -- (453.69,329.32) -- (464.53,370.67) ;
%Straight Lines [id:da5200302705705745] 
\draw    (533.26,329.02) -- (521.37,374.61) ;
%Straight Lines [id:da515718008204028] 
\draw    (451.13,328.32) -- (451.89,329.09) -- (408.1,347.35) ;
%Straight Lines [id:da948960780523119] 
\draw    (397.63,274.21) -- (350.37,257.8) ;
%Straight Lines [id:da6380755264431096] 
\draw    (436.95,282.13) -- (397.63,274.21) ;
%Straight Lines [id:da7382641351077347] 
\draw    (436.95,282.13) -- (452.93,327.22) ;
%Straight Lines [id:da8094338631603033] 
\draw    (396.7,273.95) -- (372.03,312.04) ;
%Straight Lines [id:da7472249887804613] 
\draw    (349.18,190.44) -- (394.08,174.92) ;
%Straight Lines [id:da31256765199887904] 
\draw    (372.75,137.91) -- (395.15,173.23) ;
%Straight Lines [id:da18722566459093037] 
\draw    (595.68,274.21) -- (619.58,311.22) ;
%Straight Lines [id:da38129522751946] 
\draw    (595.68,274.21) -- (640.53,267.34) ;
%Curve Lines [id:da4925900242734058] 
\draw    (453.25,222.69) .. controls (463.22,208.12) and (533.93,208.65) .. (539.74,222.69) ;
%Flowchart: Connector [id:dp7317860500136707] 
\draw  [color={rgb, 255:red, 0; green, 0; blue, 0 }  ,draw opacity=1 ][fill={rgb, 255:red, 0; green, 0; blue, 0 }  ,fill opacity=1 ] (150.85,223.94) .. controls (150.85,222.32) and (152.19,221) .. (153.84,221) .. controls (155.49,221) and (156.83,222.32) .. (156.83,223.94) .. controls (156.83,225.57) and (155.49,226.89) .. (153.84,226.89) .. controls (152.19,226.89) and (150.85,225.57) .. (150.85,223.94) -- cycle ;
%Flowchart: Connector [id:dp5736255201877954] 
\draw  [color={rgb, 255:red, 0; green, 0; blue, 0 }  ,draw opacity=1 ][fill={rgb, 255:red, 0; green, 0; blue, 0 }  ,fill opacity=1 ] (108.86,223.94) .. controls (108.86,222.32) and (110.19,221) .. (111.84,221) .. controls (113.5,221) and (114.83,222.32) .. (114.83,223.94) .. controls (114.83,225.57) and (113.5,226.89) .. (111.84,226.89) .. controls (110.19,226.89) and (108.86,225.57) .. (108.86,223.94) -- cycle ;
%Flowchart: Connector [id:dp14576774927267344] 
\draw  [color={rgb, 255:red, 0; green, 0; blue, 0 }  ,draw opacity=1 ][fill={rgb, 255:red, 0; green, 0; blue, 0 }  ,fill opacity=1 ] (192.84,223.94) .. controls (192.84,222.32) and (194.18,221) .. (195.83,221) .. controls (197.48,221) and (198.82,222.32) .. (198.82,223.94) .. controls (198.82,225.57) and (197.48,226.89) .. (195.83,226.89) .. controls (194.18,226.89) and (192.84,225.57) .. (192.84,223.94) -- cycle ;
%Flowchart: Connector [id:dp1998180354037974] 
\draw  [color={rgb, 255:red, 0; green, 0; blue, 0 }  ,draw opacity=1 ][fill={rgb, 255:red, 0; green, 0; blue, 0 }  ,fill opacity=1 ] (193.74,122.22) .. controls (193.74,120.59) and (195.08,119.27) .. (196.73,119.27) .. controls (198.38,119.27) and (199.72,120.59) .. (199.72,122.22) .. controls (199.72,123.85) and (198.38,125.17) .. (196.73,125.17) .. controls (195.08,125.17) and (193.74,123.85) .. (193.74,122.22) -- cycle ;
%Flowchart: Connector [id:dp5526189212601856] 
\draw  [color={rgb, 255:red, 0; green, 0; blue, 0 }  ,draw opacity=1 ][fill={rgb, 255:red, 0; green, 0; blue, 0 }  ,fill opacity=1 ] (52.3,176.44) .. controls (52.3,174.82) and (53.64,173.5) .. (55.29,173.5) .. controls (56.94,173.5) and (58.28,174.82) .. (58.28,176.44) .. controls (58.28,178.07) and (56.94,179.39) .. (55.29,179.39) .. controls (53.64,179.39) and (52.3,178.07) .. (52.3,176.44) -- cycle ;
%Flowchart: Connector [id:dp27255986097163953] 
\draw  [color={rgb, 255:red, 0; green, 0; blue, 0 }  ,draw opacity=1 ][fill={rgb, 255:red, 0; green, 0; blue, 0 }  ,fill opacity=1 ] (247.17,274.72) .. controls (247.17,273.09) and (248.51,271.77) .. (250.16,271.77) .. controls (251.81,271.77) and (253.15,273.09) .. (253.15,274.72) .. controls (253.15,276.35) and (251.81,277.67) .. (250.16,277.67) .. controls (248.51,277.67) and (247.17,276.35) .. (247.17,274.72) -- cycle ;
%Flowchart: Connector [id:dp23214155675465475] 
\draw  [color={rgb, 255:red, 0; green, 0; blue, 0 }  ,draw opacity=1 ][fill={rgb, 255:red, 0; green, 0; blue, 0 }  ,fill opacity=1 ] (187.79,327.79) .. controls (187.79,326.16) and (189.13,324.84) .. (190.78,324.84) .. controls (192.43,324.84) and (193.76,326.16) .. (193.76,327.79) .. controls (193.76,329.42) and (192.43,330.74) .. (190.78,330.74) .. controls (189.13,330.74) and (187.79,329.42) .. (187.79,327.79) -- cycle ;
%Flowchart: Connector [id:dp10901570703737706] 
\draw  [color={rgb, 255:red, 0; green, 0; blue, 0 }  ,draw opacity=1 ][fill={rgb, 255:red, 0; green, 0; blue, 0 }  ,fill opacity=1 ] (107.62,122.07) .. controls (107.62,120.44) and (108.95,119.12) .. (110.61,119.12) .. controls (112.26,119.12) and (113.59,120.44) .. (113.59,122.07) .. controls (113.59,123.69) and (112.26,125.01) .. (110.61,125.01) .. controls (108.95,125.01) and (107.62,123.69) .. (107.62,122.07) -- cycle ;
%Flowchart: Connector [id:dp33683995140603773] 
\draw  [color={rgb, 255:red, 0; green, 0; blue, 0 }  ,draw opacity=1 ][fill={rgb, 255:red, 0; green, 0; blue, 0 }  ,fill opacity=1 ] (108.55,326.96) .. controls (108.55,325.33) and (109.89,324.01) .. (111.54,324.01) .. controls (113.19,324.01) and (114.53,325.33) .. (114.53,326.96) .. controls (114.53,328.59) and (113.19,329.91) .. (111.54,329.91) .. controls (109.89,329.91) and (108.55,328.59) .. (108.55,326.96) -- cycle ;
%Flowchart: Connector [id:dp2488933883171398] 
\draw  [color={rgb, 255:red, 0; green, 0; blue, 0 }  ,draw opacity=1 ][fill={rgb, 255:red, 0; green, 0; blue, 0 }  ,fill opacity=1 ] (206.46,283.48) .. controls (206.46,281.85) and (207.8,280.53) .. (209.45,280.53) .. controls (211.1,280.53) and (212.44,281.85) .. (212.44,283.48) .. controls (212.44,285.11) and (211.1,286.43) .. (209.45,286.43) .. controls (207.8,286.43) and (206.46,285.11) .. (206.46,283.48) -- cycle ;
%Flowchart: Connector [id:dp6014384525735342] 
\draw  [color={rgb, 255:red, 0; green, 0; blue, 0 }  ,draw opacity=1 ][fill={rgb, 255:red, 0; green, 0; blue, 0 }  ,fill opacity=1 ] (207.1,166.2) .. controls (207.1,164.57) and (208.43,163.25) .. (210.08,163.25) .. controls (211.73,163.25) and (213.07,164.57) .. (213.07,166.2) .. controls (213.07,167.83) and (211.73,169.15) .. (210.08,169.15) .. controls (208.43,169.15) and (207.1,167.83) .. (207.1,166.2) -- cycle ;
%Flowchart: Connector [id:dp17838047989062722] 
\draw  [color={rgb, 255:red, 0; green, 0; blue, 0 }  ,draw opacity=1 ][fill={rgb, 255:red, 0; green, 0; blue, 0 }  ,fill opacity=1 ] (93.95,282.91) .. controls (93.95,281.28) and (95.29,279.96) .. (96.94,279.96) .. controls (98.59,279.96) and (99.92,281.28) .. (99.92,282.91) .. controls (99.92,284.54) and (98.59,285.86) .. (96.94,285.86) .. controls (95.29,285.86) and (93.95,284.54) .. (93.95,282.91) -- cycle ;
%Flowchart: Connector [id:dp029985131327229242] 
\draw  [color={rgb, 255:red, 0; green, 0; blue, 0 }  ,draw opacity=1 ][fill={rgb, 255:red, 0; green, 0; blue, 0 }  ,fill opacity=1 ] (94.56,167.19) .. controls (94.56,165.56) and (95.9,164.24) .. (97.55,164.24) .. controls (99.2,164.24) and (100.54,165.56) .. (100.54,167.19) .. controls (100.54,168.81) and (99.2,170.13) .. (97.55,170.13) .. controls (95.9,170.13) and (94.56,168.81) .. (94.56,167.19) -- cycle ;
%Flowchart: Connector [id:dp0578510831816178] 
\draw  [color={rgb, 255:red, 0; green, 0; blue, 0 }  ,draw opacity=1 ][fill={rgb, 255:red, 0; green, 0; blue, 0 }  ,fill opacity=1 ] (248.44,175.46) .. controls (248.44,173.83) and (249.78,172.51) .. (251.43,172.51) .. controls (253.08,172.51) and (254.42,173.83) .. (254.42,175.46) .. controls (254.42,177.09) and (253.08,178.41) .. (251.43,178.41) .. controls (249.78,178.41) and (248.44,177.09) .. (248.44,175.46) -- cycle ;
%Straight Lines [id:da5408783269095064] 
\draw    (111.84,223.94) -- (195.83,223.94) ;
%Straight Lines [id:da5419140814079776] 
\draw    (210.08,169.15) -- (195.83,223.94) ;
%Straight Lines [id:da7712692018014478] 
\draw    (55.29,176.44) -- (97.55,167.19) ;
%Shape: Free Drawing [id:dp9144648614648596] 
\draw  [color={rgb, 255:red, 255; green, 255; blue, 255 }  ,draw opacity=1 ][line width=4.5] [line join = round][line cap = round] (115.21,238.23) .. controls (114.79,236.48) and (113.46,234.72) .. (113.94,232.99) .. controls (114.33,231.57) and (115.28,235.63) .. (115.53,237.08) .. controls (115.99,239.69) and (113.68,234.98) .. (114.42,235.93) .. controls (114.76,236.37) and (114.97,236.91) .. (115.21,237.41) .. controls (115.41,237.8) and (114.79,236.64) .. (114.58,236.26) ;
%Shape: Free Drawing [id:dp13632769671211353] 
\draw  [color={rgb, 255:red, 255; green, 255; blue, 255 }  ,draw opacity=1 ][line width=4.5] [line join = round][line cap = round] (115.85,233.31) .. controls (114.98,232.24) and (113.18,230.96) .. (113.78,229.71) ;
%Shape: Free Drawing [id:dp47176022245766736] 
\draw  [color={rgb, 255:red, 255; green, 255; blue, 255 }  ,draw opacity=1 ][line width=0.75] [line join = round][line cap = round] (113.31,229.71) .. controls (112.93,228.93) and (111.78,228.79) .. (112.67,228.56) ;
%Shape: Free Drawing [id:dp7205525626068784] 
\draw  [color={rgb, 255:red, 255; green, 255; blue, 255 }  ,draw opacity=1 ][line width=1.5] [line join = round][line cap = round] (194.21,235.28) .. controls (194.42,234.41) and (194.23,233.3) .. (194.85,232.66) .. controls (195.23,232.26) and (195.08,234.85) .. (195.16,234.3) .. controls (195.4,232.68) and (195.27,231.02) .. (195.32,229.38) ;
%Shape: Free Drawing [id:dp6164005011981001] 
\draw  [color={rgb, 255:red, 255; green, 255; blue, 255 }  ,draw opacity=1 ][line width=1.5] [line join = round][line cap = round] (112.62,229.52) .. controls (112.62,229.14) and (112.62,228.76) .. (112.62,228.38) ;
%Shape: Free Drawing [id:dp8749396807815106] 
\draw  [color={rgb, 255:red, 255; green, 255; blue, 255 }  ,draw opacity=1 ][line width=1.5] [line join = round][line cap = round] (195.59,230.02) .. controls (195.53,229.42) and (195.48,228.81) .. (195.43,228.21) ;
%Flowchart: Connector [id:dp8342193979364814] 
\draw  [color={rgb, 255:red, 0; green, 0; blue, 0 }  ,draw opacity=1 ][fill={rgb, 255:red, 0; green, 0; blue, 0 }  ,fill opacity=1 ] (294.15,194.13) .. controls (294.15,192.41) and (295.47,191.01) .. (297.1,191.01) .. controls (298.73,191.01) and (300.05,192.41) .. (300.05,194.13) .. controls (300.05,195.86) and (298.73,197.26) .. (297.1,197.26) .. controls (295.47,197.26) and (294.15,195.86) .. (294.15,194.13) -- cycle ;
%Flowchart: Connector [id:dp007618770942422692] 
\draw  [color={rgb, 255:red, 0; green, 0; blue, 0 }  ,draw opacity=1 ][fill={rgb, 255:red, 0; green, 0; blue, 0 }  ,fill opacity=1 ] (290.76,267.95) .. controls (290.76,266.22) and (292.08,264.82) .. (293.71,264.82) .. controls (295.34,264.82) and (296.66,266.22) .. (296.66,267.95) .. controls (296.66,269.68) and (295.34,271.08) .. (293.71,271.08) .. controls (292.08,271.08) and (290.76,269.68) .. (290.76,267.95) -- cycle ;
%Flowchart: Connector [id:dp6219787246323923] 
\draw  [color={rgb, 255:red, 0; green, 0; blue, 0 }  ,draw opacity=1 ][fill={rgb, 255:red, 0; green, 0; blue, 0 }  ,fill opacity=1 ] (270.38,311.19) .. controls (270.38,309.56) and (271.71,308.24) .. (273.36,308.24) .. controls (275.01,308.24) and (276.35,309.56) .. (276.35,311.19) .. controls (276.35,312.82) and (275.01,314.14) .. (273.36,314.14) .. controls (271.71,314.14) and (270.38,312.82) .. (270.38,311.19) -- cycle ;
%Flowchart: Connector [id:dp4918569822020864] 
\draw  [color={rgb, 255:red, 0; green, 0; blue, 0 }  ,draw opacity=1 ][fill={rgb, 255:red, 0; green, 0; blue, 0 }  ,fill opacity=1 ] (227.27,351.77) .. controls (227.27,350.14) and (228.61,348.82) .. (230.26,348.82) .. controls (231.91,348.82) and (233.25,350.14) .. (233.25,351.77) .. controls (233.25,353.4) and (231.91,354.72) .. (230.26,354.72) .. controls (228.61,354.72) and (227.27,353.4) .. (227.27,351.77) -- cycle ;
%Flowchart: Connector [id:dp1950021120898534] 
\draw  [color={rgb, 255:red, 0; green, 0; blue, 0 }  ,draw opacity=1 ][fill={rgb, 255:red, 0; green, 0; blue, 0 }  ,fill opacity=1 ] (175.01,373.65) .. controls (175.01,372.02) and (176.35,370.7) .. (178,370.7) .. controls (179.65,370.7) and (180.99,372.02) .. (180.99,373.65) .. controls (180.99,375.28) and (179.65,376.6) .. (178,376.6) .. controls (176.35,376.6) and (175.01,375.28) .. (175.01,373.65) -- cycle ;
%Flowchart: Connector [id:dp1557953449281143] 
\draw  [color={rgb, 255:red, 0; green, 0; blue, 0 }  ,draw opacity=1 ][fill={rgb, 255:red, 0; green, 0; blue, 0 }  ,fill opacity=1 ] (120.71,370.77) .. controls (120.71,369.14) and (122.05,367.82) .. (123.7,367.82) .. controls (125.35,367.82) and (126.69,369.14) .. (126.69,370.77) .. controls (126.69,372.4) and (125.35,373.72) .. (123.7,373.72) .. controls (122.05,373.72) and (120.71,372.4) .. (120.71,370.77) -- cycle ;
%Flowchart: Connector [id:dp3785168110571463] 
\draw  [color={rgb, 255:red, 0; green, 0; blue, 0 }  ,draw opacity=1 ][fill={rgb, 255:red, 0; green, 0; blue, 0 }  ,fill opacity=1 ] (122.88,77.51) .. controls (122.88,75.89) and (124.21,74.57) .. (125.86,74.57) .. controls (127.51,74.57) and (128.85,75.89) .. (128.85,77.51) .. controls (128.85,79.14) and (127.51,80.46) .. (125.86,80.46) .. controls (124.21,80.46) and (122.88,79.14) .. (122.88,77.51) -- cycle ;
%Flowchart: Connector [id:dp9034323024976331] 
\draw  [color={rgb, 255:red, 0; green, 0; blue, 0 }  ,draw opacity=1 ][fill={rgb, 255:red, 0; green, 0; blue, 0 }  ,fill opacity=1 ] (179.46,75.55) .. controls (179.46,73.92) and (180.8,72.6) .. (182.45,72.6) .. controls (184.1,72.6) and (185.44,73.92) .. (185.44,75.55) .. controls (185.44,77.18) and (184.1,78.5) .. (182.45,78.5) .. controls (180.8,78.5) and (179.46,77.18) .. (179.46,75.55) -- cycle ;
%Flowchart: Connector [id:dp11807544988114993] 
\draw  [color={rgb, 255:red, 0; green, 0; blue, 0 }  ,draw opacity=1 ][fill={rgb, 255:red, 0; green, 0; blue, 0 }  ,fill opacity=1 ] (71.38,99.13) .. controls (71.38,97.51) and (72.72,96.19) .. (74.37,96.19) .. controls (76.02,96.19) and (77.35,97.51) .. (77.35,99.13) .. controls (77.35,100.76) and (76.02,102.08) .. (74.37,102.08) .. controls (72.72,102.08) and (71.38,100.76) .. (71.38,99.13) -- cycle ;
%Flowchart: Connector [id:dp6540099116824407] 
\draw  [color={rgb, 255:red, 0; green, 0; blue, 0 }  ,draw opacity=1 ][fill={rgb, 255:red, 0; green, 0; blue, 0 }  ,fill opacity=1 ] (269.1,137.13) .. controls (269.1,135.51) and (270.44,134.19) .. (272.09,134.19) .. controls (273.74,134.19) and (275.08,135.51) .. (275.08,137.13) .. controls (275.08,138.76) and (273.74,140.08) .. (272.09,140.08) .. controls (270.44,140.08) and (269.1,138.76) .. (269.1,137.13) -- cycle ;
%Flowchart: Connector [id:dp19408605208034801] 
\draw  [color={rgb, 255:red, 0; green, 0; blue, 0 }  ,draw opacity=1 ][fill={rgb, 255:red, 0; green, 0; blue, 0 }  ,fill opacity=1 ] (231.59,98.48) .. controls (231.59,96.85) and (232.93,95.53) .. (234.58,95.53) .. controls (236.23,95.53) and (237.57,96.85) .. (237.57,98.48) .. controls (237.57,100.11) and (236.23,101.43) .. (234.58,101.43) .. controls (232.93,101.43) and (231.59,100.11) .. (231.59,98.48) -- cycle ;
%Flowchart: Connector [id:dp32002217307093916] 
\draw  [color={rgb, 255:red, 0; green, 0; blue, 0 }  ,draw opacity=1 ][fill={rgb, 255:red, 0; green, 0; blue, 0 }  ,fill opacity=1 ] (54.85,274.72) .. controls (54.85,273.09) and (56.19,271.77) .. (57.84,271.77) .. controls (59.49,271.77) and (60.82,273.09) .. (60.82,274.72) .. controls (60.82,276.35) and (59.49,277.67) .. (57.84,277.67) .. controls (56.19,277.67) and (54.85,276.35) .. (54.85,274.72) -- cycle ;
%Flowchart: Connector [id:dp9329131306481169] 
\draw  [color={rgb, 255:red, 0; green, 0; blue, 0 }  ,draw opacity=1 ][fill={rgb, 255:red, 0; green, 0; blue, 0 }  ,fill opacity=1 ] (30.69,140.41) .. controls (30.69,138.78) and (32.03,137.46) .. (33.68,137.46) .. controls (35.33,137.46) and (36.66,138.78) .. (36.66,140.41) .. controls (36.66,142.04) and (35.33,143.36) .. (33.68,143.36) .. controls (32.03,143.36) and (30.69,142.04) .. (30.69,140.41) -- cycle ;
%Flowchart: Connector [id:dp17545226880466602] 
\draw  [color={rgb, 255:red, 0; green, 0; blue, 0 }  ,draw opacity=1 ][fill={rgb, 255:red, 0; green, 0; blue, 0 }  ,fill opacity=1 ] (7.8,192.17) .. controls (7.8,190.54) and (9.14,189.22) .. (10.79,189.22) .. controls (12.44,189.22) and (13.78,190.54) .. (13.78,192.17) .. controls (13.78,193.8) and (12.44,195.12) .. (10.79,195.12) .. controls (9.14,195.12) and (7.8,193.8) .. (7.8,192.17) -- cycle ;
%Flowchart: Connector [id:dp790728805224241] 
\draw  [color={rgb, 255:red, 0; green, 0; blue, 0 }  ,draw opacity=1 ][fill={rgb, 255:red, 0; green, 0; blue, 0 }  ,fill opacity=1 ] (10.34,258.34) .. controls (10.34,256.71) and (11.68,255.39) .. (13.33,255.39) .. controls (14.98,255.39) and (16.32,256.71) .. (16.32,258.34) .. controls (16.32,259.97) and (14.98,261.29) .. (13.33,261.29) .. controls (11.68,261.29) and (10.34,259.97) .. (10.34,258.34) -- cycle ;
%Flowchart: Connector [id:dp9500492122649733] 
\draw  [color={rgb, 255:red, 0; green, 0; blue, 0 }  ,draw opacity=1 ][fill={rgb, 255:red, 0; green, 0; blue, 0 }  ,fill opacity=1 ] (29.99,309.05) .. controls (29.99,307.42) and (31.33,306.1) .. (32.98,306.1) .. controls (34.63,306.1) and (35.96,307.42) .. (35.96,309.05) .. controls (35.96,310.68) and (34.63,312) .. (32.98,312) .. controls (31.33,312) and (29.99,310.68) .. (29.99,309.05) -- cycle ;
%Flowchart: Connector [id:dp18656251350903785] 
\draw  [color={rgb, 255:red, 0; green, 0; blue, 0 }  ,draw opacity=1 ][fill={rgb, 255:red, 0; green, 0; blue, 0 }  ,fill opacity=1 ] (65.02,346.79) .. controls (65.02,345.16) and (66.36,343.84) .. (68.01,343.84) .. controls (69.66,343.84) and (71,345.16) .. (71,346.79) .. controls (71,348.42) and (69.66,349.74) .. (68.01,349.74) .. controls (66.36,349.74) and (65.02,348.42) .. (65.02,346.79) -- cycle ;
%Straight Lines [id:da7937805298506001] 
\draw    (74.37,99.13) -- (110.61,122.07) ;
%Straight Lines [id:da40325489416876936] 
\draw    (125.86,77.51) -- (110.61,122.07) ;
%Straight Lines [id:da18918398026366945] 
\draw    (110.61,122.07) -- (97.55,167.19) ;
%Straight Lines [id:da05702444950942309] 
\draw    (97.55,167.19) -- (111.84,223.94) ;
%Straight Lines [id:da12668763434122765] 
\draw    (111.84,223.94) -- (97.59,278.74) -- (96.94,282.91) ;
%Straight Lines [id:da8059600074686101] 
\draw    (195.83,221) -- (209.45,283.48) ;
%Straight Lines [id:da2822435337379374] 
\draw    (251.43,175.46) -- (297.1,194.13) ;
%Straight Lines [id:da047675359035670106] 
\draw    (272.09,137.13) -- (251.43,175.46) ;
%Straight Lines [id:da3628181295724566] 
\draw    (231.59,98.48) -- (196.73,122.22) ;
%Straight Lines [id:da13418051263586928] 
\draw    (182.45,75.55) -- (196.73,122.22) ;
%Straight Lines [id:da837867729886925] 
\draw    (251.43,175.46) -- (210.08,166.2) ;
%Straight Lines [id:da9138945606593271] 
\draw    (196.73,122.22) -- (210.08,166.2) ;
%Straight Lines [id:da004244575679796636] 
\draw    (250.16,274.72) -- (209.45,283.48) ;
%Straight Lines [id:da8746675989734131] 
\draw    (190.78,327.79) -- (230.26,351.77) ;
%Straight Lines [id:da04329296497108548] 
\draw    (209.45,283.48) -- (190.78,327.79) ;
%Straight Lines [id:da5490925238995809] 
\draw    (111.54,328.27) -- (112.27,329.03) -- (122.8,369.78) ;
%Straight Lines [id:da7183541545179195] 
\draw    (189.55,328.73) -- (178,373.65) ;
%Straight Lines [id:da23510459326660282] 
\draw    (109.79,328.04) -- (110.53,328.79) -- (68.01,346.79) ;
%Straight Lines [id:da09979417039005445] 
\draw    (57.84,274.72) -- (11.95,258.55) ;
%Straight Lines [id:da5705754771096326] 
\draw    (96.02,282.52) -- (57.84,274.72) ;
%Straight Lines [id:da2587913365486214] 
\draw    (96.02,282.52) -- (111.54,326.96) ;
%Straight Lines [id:da9390831741874117] 
\draw    (56.93,274.46) -- (32.98,312) ;
%Straight Lines [id:da5137243649609124] 
\draw    (10.79,192.17) -- (54.39,176.88) ;
%Straight Lines [id:da6990612315156804] 
\draw    (33.68,140.41) -- (55.43,175.21) ;
%Straight Lines [id:da7367446280537062] 
\draw    (250.16,274.72) -- (273.36,311.19) ;
%Straight Lines [id:da6845869736869883] 
\draw    (250.16,274.72) -- (293.71,267.95) ;

% Text Node
\draw (85,410) node [anchor=north west][inner sep=0.75pt]   [align=left] {\hspace{-0.3cm}Rooted binary tree};
% Text Node
\draw (444,409) node [anchor=north west][inner sep=0.75pt]   [align=left] {\hspace{0.7cm}$\wh \Ga_2$};

\end{tikzpicture}
\end{figure}

In the picture below you can see part of the rooted binary tree and of the associated dyadic spider's web $\wh\Ga_2$.  
There are two natural metrics on a spider's web $\wh\Ga$: the tree metric~$d_{\mrT}$ and the graph metric $d_{\wh\Ga}$.  Obviously
$d_\mrT \geq d_{\wh\Ga}$.  Notice that $d_\mrT$ and~$d_{\wh\Ga}$ may be nonequivalent metrics.  For instance, consider the points 
$x_n := 0 \underbrace{1 \cdots 1}_{n \mathrm{times}}$ and $y_n := 1 \overbrace{0 \cdots 0}^{n \mathrm{times}}$ in $\wh\Ga_2$.  Then 
$$
d_{\mrT} (x_n,y_n) 
= 2n+2 
\qquad \hbox{and}\qquad
d_{\wh\Ga} (x_n,y_n) 
= 1. 
$$

\subsection{Metric graphs embedded in length spaces}
In the sequel we shall encounter connected graphs $\Ga$ (without self-loops and multi-edges) 
whose set of vertices is a discrete subset of a metric measure space $(X,d,\mu)$.  In particular, we shall be concerned with the case where $X$ is a complete 
length space with length function $\ell_X$, $d$ is the associated strictly intrinsic metric \cite[Definition~8.4.1]{BBI}, and~$\mu$ is a locally doubling 
Borel measure on $(X,d)$.  For each pair of neighbours $x$ and $y$ in $\Ga$, consider a geodesic $\ga_{x,y}$  in the length space $X$ connecting them.

We denote by $\wt \Ga$ the \textit{metric graph} defined as follows.  The vertices of $\wt\Ga$ agree with those of~$\Ga$; the edge in $\wt\Ga$ connecting two 
neighbours $x$ and $y$ in~$\Ga$ is the geodesic segment $\ga_{x,y}$ chosen above.   
If~$z$ and $w$ are points in $\ga_{x,y}$, then we set
\begin{equation}\label{d: Ga wtGA}
d_{\wt \Ga} (z,w)
:= d(z,w). 
\end{equation}
In particular, $d_{\wt \Ga} (x,y) = \ell_X(\ga_{x,y})$.
Now, if $v$ and $w$ are any two points in~$\wt\Ga$ (thus, $v$ and $w$ may be either vertices of $\Ga$ or points in the interior of geodesic segments), then
we define $d_{\wt\Ga}(v,w)$ to be the infimum of the lengths of all paths in~$\wt\Ga$ joining $v$ and $w$.  
Clearly $d_{\wt\Ga}$ is a distance on $\wt\Ga$.  

\begin{definition} \label{def: bounded geometry graph}
We say that the metric graph $\wt\Ga$ associated to the graph~$\Ga$ has \textit{bounded geometry} if~$\Ga$ has bounded valence and
$$
	0 < \inf\,\{\ell_X(\si): \hbox{$\si$ edge in $\wt\Ga$}\} 
	\quad\hbox{and}\quad 
	\sup\,\{\ell_X(\si): \hbox{$\si$ edge in $\wt\Ga$}\} 
	<\infty.
$$
\end{definition}

Next we describe the \textit{metric graph} $\Ga_0$ associated to the connected graph~$\Ga$.  Loosely speaking, $\Ga_0$, as a set, agrees with $\wt\Ga$,
but each edge of $\Ga_0$ (the image of a geodesic segment in $X$) is now declared to be isometric to the interval $[0,1]$.
More precisely, for each pair $x$ and $y$ of neighbours in~$\Ga$, consider the geodesic $\ga_{x,y}$ in $X$ joining $x$ and~$y$ (which is an ``edge" in $\wt\Ga$), 
and its arc-length parametrisation $s:[0,\ell_X(\ga_{x,y})] \to \ga_{x,y}$ such that $s(0) = x$ and $s\big(\ell_X(\ga_{x,y})\big) = y$.  Define
$$
\iota(t)  
:= s\big(t\cdot \ell_X(\ga_{x,y})\big)
\quant t \in [0,1], 
$$
and set
\begin{equation} \label{f: iota}
d_{\Ga_0} \big(\iota(t_1),\iota(t_2)\big)
:= \mod{t_1-t_2}
\quant t_1,t_2\in [0,1]. 
\end{equation}
There is a natural notion of admissible paths in $\Ga_0$ and a corresponding notion of length.  Now $d_{\Ga_0}$ is just the metric associated to such 
length structure.  

\begin{remark} \label{rem: induced distance}
It is straightforward to check that if $x$ and $y$ are neighbours in~$\Ga$, then 
\begin{equation} \label{f: induced distance} 
	1
	= d_\Ga(x,y)
	= d_{\Ga_0}(x,y)
	= \frac{1}{\ell_X(\ga_{x,y})}\, \, d_{\wt\Ga}(x,y).
\end{equation}
If $\wt\Ga$ has bounded geometry, then there exist positive constants $A_1$ and $A_2$ such that 
\begin{equation}\label{eq: leng ngb}
A_1\leq \ell_X(\ga_{x,y}) \leq A_2,
\end{equation}
	 so that 
\begin{equation} \label{f: induced distance II} 
	A_1 \, d_\Ga(x,y) \leq  d_{\wt\Ga}(x,y) \leq A_2 \, d_\Ga(x,y) \quant x,y \in \Ga.  
\end{equation}
\end{remark}
		
\begin{proposition} \label{p: Ga Ga_0 wtGa}
	Suppose that $\Ga$, $\Ga_0$ and $\wt\Ga$ are as described above.  If~$\wt\Ga$ has bounded geometry
	and $(\wt\Ga, d_{\wt\Ga})$ is a $\de$-hyperbolic length space for some nonnegative number $\de$, then the following hold:
	\begin{enumerate} 
		\item[\itemno1] $(\wt\Ga, d_{\wt\Ga})$ and $(\Ga_0,d_{\Ga_0})$ are bi-Lipschitz equivalent. More precisely
 \begin{align*}
      A_1\,d_{\Ga_0} (v,w) 
	\leq d_{\wt\Ga} (v,w)  
	\leq A_2 \, d_{\Ga_0} (v,w) \quant v,w \in \wt \Ga, 
 \end{align*} 
			where $A_1$ and $A_2$ are as  in \eqref{eq: leng ngb};
		\item[\itemno2]
			$(\Ga_0,d_{\Ga_0})$ is a $\de'$-hyperbolic length space for some nonnegative number~$\de'$; 
		\item[\itemno3]
			$(\Ga,d_\Ga)$ is a Gromov hyperbolic metric space (the four points condition holds). 
	\end{enumerate}
\end{proposition}

\begin{proof}
	First we prove \rmi.  Clearly, if $x$ and $y$ are neighbours in $\Ga$, then 
	$$
	d_{\Ga_0}(x,y) 
	= 1 
	= \frac{1}{\ell_X(\ga_{x,y})} \, d_{\wt\Ga} (x,y).
	$$  
	If $\ga_{x,y}$ is the edge in $\wt\Ga$ joining $x$ and $y$, denote by $s:[0,\ell_X(\ga_{x,y})] \to \ga_{x,y}$ is its arc-length parametrisation 
	(with $s(0) = x$ and $s\big(\ell_X(\ga_{x,y})\big) = y$).  Write~$L$ in place of $\ell_X(\ga_{x,y})$ for short.  Then
	for each pair of numbers $t_1$ and $t_2$ in $[0,L]$, we have that 
	$$
	\begin{aligned}
	d_{\Ga_0} \big(s(t_1),s(t_2)\big)
		& =  d_{\Ga_0} \Big(s\Big(\frac{t_1}{L}\cdot L\Big),s\Big(\frac{t_2}{L}\cdot L \Big)\Big) \\
		& =  \frac{1}{L} \,\, \mod{t_1-t_2} \\
		& =  \frac{1}{L} \,\, d_{\wt\Ga} \big(s(t_1),s(t_2)\big); 
	\end{aligned}
	$$
	we have used \eqref{f: iota} in the second to last equality above. 

	Since, by assumption, $\wt\Ga$ has bounded geometry, \eqref{eq: leng ngb} holds for every pair $x$ and $y$ of neighbours.
	As a consequence, we see that for each pair of points~$v$ and $w$ belonging to an edge in $\Ga_0$ the following estimate holds
	$$
	A_1 \, d_{\Ga_0} (v,w) 
	\leq d_{\wt\Ga} (v,w)  
	\leq A_2 \, d_{\Ga_0} (v,w).  
	$$
	It is straightforward to check that this estimate extends to all pairs of points~$v$ and $w$ in $\wt\Ga$, 
	thereby concluding the proof of \rmi.
 
	Observe that \rmii\  follows from \rmi\ and \cite[Theorem~8.4.16]{BBI}.

	 Finally, we prove \rmiii.  Since $(\Ga_0,d_{\Ga_0})$ is a $\de'$-hyperbolic length space, 
	in particular a $\de'$-hyperbolic space, the four point 
	condition \eqref{f: Gromov} holds (with $\de'$ in place of $\de$).  Since the metric graph $d_\Ga$ on $\Ga$ agrees with the restriction 
	of $d_{\Ga_0}$ to $\Ga$, the four points condition holds in $(\Ga,d_\Ga)$.  Thus, $(\Ga,d_\Ga)$ is a $\de'$-hyperbolic metric space, as required. 
\end{proof}

\noindent
\begin{caveat}  \label{cav: metric graphs}
	\rm{In the sequel we shall consistently use the notation concerning graphs introduced in this section.  In particular, given a 
	metric space $(X,d)$, where $d$ is a strictly intrinsic metric, $\Ga$ {will always denote} an \textit{abstract graph}, whose vertices are points of $X$, 
	endowed with the graph distance $d_\Ga$.  We associate to $\Ga$ the \textit{metric graphs} $(\wt \Ga, d_{\wt \Ga})$ and $(\Ga_0,d_{\Ga_0})$ as follows:
	\begin{enumerate} 
	\item[\itemno1] 
		{$\wt\Ga$} is obtained from $\Ga$ by connecting any pair $x$ and $y$ of neighbours in $\Ga$ by 
		a geodesic $\ga_{x,y}$ in $X$ joining $x$ and $y$, and declaring that $d_{\wt\Ga}(z,w) = d(z,w)$ for every pair of points 
		$z$ and $w$ in $\ga_{x,y}$;
	\item[\itemno2] 
		$\Ga_0$ agrees with $\wt\Ga$ as a set, 
		but $d_{\Ga_0}(z,w) = d(z,w)/\ell_X(\ga_{x,y})$, where $x$, $y$, $z$ and $w$ are as in \rmi.  In particular, $d_{\Ga_0}(x,y) = 1$.
	\end{enumerate}
	\noindent
	The symbol $\wh\Ga$ will always denote a spider's web.}
\end{caveat}

\begin{notation} \label{nota: rfloor}
	\rm{Suppose that $\Ga$ and $\Ga_0$ are as above, and denote by $o$ a distinguished point in $\Ga$.  Assume that $x$ belongs to $\Ga_0 \setminus \Ga$ 
	and that $\ga$ is a geodesic in~$\Ga_0$ starting at $o$ and containing $x$.  Denote by $\floor x$ the vertex on the segment $[ox] \subset \ga$ closest to $x$.   
	Observe that 
	\begin{equation} \label{f: pavimento e soffitto}
	d_{\Ga_0}(\floor x, o) \le d_{\Ga_0}(x, o) \leq d_{\Ga_0}(\floor x, o)+1; 
	\end{equation}
	If $x$ is in $\Ga$, then we set $x :=\lfloor x \rfloor$.  The definition of $\floor x$
	depends on~$o$ and the choice of the geodesic $\ga$; in order to avoid cumbersome formulae we shall suppress this dependence in our notation.} 
\end{notation}

\section[Spider webs]{More on spider's webs} \label{s: Spider webs}

Suppose that $\wh\Ga$ is a spider's web with metric $d_{\wh\Ga}$.
We say that a geodesic~$\ga$ in $\wh\Ga$ is \textit{ascending} if it 
is of the form $[x,p(x),\ldots,p^n(x)]$, where $x$ is a vertex and $n\leq h(x)$.  Similarly  we say that a geodesic~$\ga$ in $\wh\Ga$ is \textit{descending} if it 
is of the form $[y_0,y_1,\ldots,y_n]$, where $y_j = p(y_{j+1})$ for every $j$ in $\{0,\ldots,n-1\}$.  Finally,
we say that a geodesic $\ga$ connecting two points belonging to the same level $\Si_k$ (see \eqref{f: Sik} for the notation) is \textit{horizontal} if every vertex in $\ga$
belongs to $\Si_k$.  

Our strategy to prove $L^p$ bounds for $\cM_\infty$ on spider's webs will require estimates of the volume of balls in $\wh\Ga$
with respect to the graph distance $d_{\wh\Ga}$.  As a preliminary step, we describe some of the geodesics in $\wh\Ga$.
Interestingly, if the metric space $(\wh\Ga,d_{\wh\Ga})$ is Gromov hyperbolic, then it turns out that such geodesics have a form similar to that 
of the geodesics in the hyperbolic plane.  

\begin{definition} \label{f: standard}
	Suppose that $(\wh\Ga,d_{\wh\Ga})$ is a spider's web.  We say that a geodesic $\ga$ in $\wh\Ga$ is \textit{standard} if it is the union  
	of an ascending geodesic, a horizontal geodesic and a descending geodesic.  Each of these three geodesic components may  be reduced to a point.
\end{definition}

\begin{proposition} \label{p: geodesics}
	Suppose that $(\wh\Ga,d_{\wh\Ga})$ is a spider's web.
	The following hold:
	\begin{enumerate}
		\item[\itemno1]
			for each pair of vertices~$x$ and $y$ in~$\wh\Ga$ there exists a standard geodesic joining them;
		\item[\itemno2] 
			if $(\wh\Ga,d_{\wh\Ga})$ is $\de$-hyperbolic for some nonnegative number~$\de$,  
			then the length of the horizontal part of any standard geodesic is at most~$4\de+~1$. 
	\end{enumerate}
\end{proposition}

\begin{proof}
	First we prove \rmi.  Suppose that $\ga_0$, of the form $[x,x_1,\ldots,x_{\ell-1},y]$, is a shortest path joining $x$ and $y$.
	For convenience, set $x_0 := x$ and $x_\ell := y$.  Each point in $\ga_0$ belongs to one of the sets $E_1$ and $E_2$, defined by
\begin{align*}
	E_1 
	&:= \big\{x_j \in \ga_{0}\setminus\{y\}: h(x_j) \neq h(x_{j+1}) \big\}, \\ 
	E_2 
	&:= \big\{x_j \in \ga_{0}\setminus\{y\}: h(x_j) = h(x_{j+1}) \} \cup \{y\big\}.  
\end{align*}
	Observe that $\ell(\ga_0) = |E_1|+|E_2|-1.$ 

	We shall prove that there exists a path $\ga$ of the form described in the statement of the proposition such that $\ell(\ga) \leq \ell(\ga_0)$.  

	Denote by $n$ the smallest nonnegative integer $k$ such that $\Si_k$ (see \eqref{f: Sik} for the notation) has nonempty intersection with~$\ga_0$.  
	Then $h(v)\geq n$ for every~$v$ in $\ga_0$.

	Since $\ga_0$ starts at $x$ and has nonempty intersection with $\Si_n$, it must ``move towards $o$" at least $h(x)-n$ steps; here 
	$o$ denotes the root of the tree $\mrT$ associated to $\wh\Ga$.  
	Similarly, since $\ga_0$ ends at $y$, after intersecting $\Si_n$, it must ``move away from $o$" at least $h(y)-n$ steps.  Therefore 
	\begin{equation} \label{f: ga0 I}
	\mod{E_1}\geq h(x)+h(y)-2n.
	\end{equation}
	Now, given a vertex $v$ in $\ga_0$, the point $p^{h(v)-n}(v)$ belongs to $\Si_n$, it is denoted by $\pi(v)$, and it is called the \textit{projection} of $v$ onto $\Si_n$.
	
	Note the following two elementary facts:
	\begin{enumerate}
		\item[(a)] 
			if $x_j$ belongs to $E_1$, then $\pi(x_j) = \pi(x_{j+1})$, for $x_j = p(x_{j+1})$ if $h(x_j) > h(x_{j+1})$ and $p(x_j) = x_{j+1}$ 
			if $h(x_j) < h(x_{j+1})$;
		\item[(b)] 
			if $x_j$ and $x_k$ are two consecutive elements in $E_2$ (hence $j<k$), then one of the following cases occurs:
			{\begin{enumerate}
				\item[(b1)] 
					$x_k$ agrees with $x_{j+1}$ (i.e. $x_j$ and $x_k$ are consecutive points in $\ga_0$) and therefore $h(x_j)=h(x_k)$; 
				\item[(b2)] 
					$x_k$ is a successor of $x_{j+1}$; 
				\item[(b3)] 
					$x_k$ is a predecessor of $x_{j+1}$. 
			\end{enumerate}}
			\noindent
			Since $\wh \Ga$ is a spider's web, $\pi(x_j)$ and $\pi(x_k)$ either agree or are neighbours, so that their distance in $\wh \Ga$ is at most $1$.  
	\end{enumerate} 
	Now, (a) implies that $\pi(\ga_0)$ agrees with $\pi(E_2)$.  Denote by $z_1,\ldots, z_p$ the points in $E_2$, and consider 
	$\pi(z_1), \ldots, \pi(z_p)$, which is an enumeration of the points of $\pi(E_2)$.  We single out an ordered subset 
	$\{\xi_1,\ldots,\xi_q\}$ of $\{\pi(z_1), \ldots, \pi(z_p)\}$, as follows.  Set $\xi_1 := \pi(z_1)$, and $\xi_2 := \pi(z_j)$, where $j$ is the first integer 
	greater than $1$ such that $\pi(z_j)$ is distinct from $\xi_1$.  Notice that $\xi_1$ and $\xi_2$ are neighbours, by (b) above.
	Then proceed iteratively.

	Clearly $q\leq  |\pi(E_2)|$.  Furthermore $\wt \ga_0 := [\xi_1,\ldots,\xi_q]$ is a path in $\wh\Ga$ contained in $\Si_n$ connecting $\pi(x)$ and $\pi(y)$.  Note that 
	\begin{equation} \label{f: ga0 II}
	\ell(\wt\ga_0) 
	\leq  |\pi(E_2)|-1 
	\leq |E_2|-1.
	\end{equation}
	Now denote by $\ga$ the path joining $x$ and $y$ consisting of the ascending geodesic $[x,p(x),\ldots,\pi(x)]$, 
	the horizontal path $\wt\ga_0$ and the descending geodesic $[\pi(y),\ldots,p(y),y]$.  Clearly \eqref{f: ga0 I} and \eqref{f: ga0 II} imply that
	$$
	\ell(\ga) 
	=    h(x)+h(y)-2n  + \ell(\wt\ga_0)
	\le   |E_1| +  |E_2|-1
	=    \ell(\ga_0).   
	$$
	Thus, $\ga$ is a geodesic, as required.

\tikzset{every picture/.style={line width=0.75pt}} %set default line width to 0.75pt        
\begin{figure}
\centering
\begin{tikzpicture}[x=0.75pt,y=0.75pt,yscale=-1,xscale=1]

%Straight Lines [id:da9701132289548664] 
\draw [color={rgb, 255:red, 208; green, 2; blue, 27 }  ,draw opacity=1 ]   (170.32,50.02) -- (89.92,250.08) ;
%Straight Lines [id:da5672080517687881] 
\draw [color={rgb, 255:red, 208; green, 2; blue, 27 }  ,draw opacity=1 ]   (188.56,99.88) -- (240.72,249.96) ;
%Straight Lines [id:da2586527772990479] 
\draw [color={rgb, 255:red, 208; green, 2; blue, 27 }  ,draw opacity=1 ]   (188.56,99.88) -- (150.06,99.88) ;
%Straight Lines [id:da2241934604521405] 
\draw    (170.32,50.02) -- (150.06,99.88) ;
%Straight Lines [id:da6096043473111667] 
\draw    (170.32,50.02) -- (188.56,99.88) ;
%Straight Lines [id:da9379780717666417] 
\draw [color={rgb, 255:red, 0; green, 0; blue, 0 }  ,draw opacity=1 ]   (410.32,51.52) -- (329.92,251.58) ;
%Straight Lines [id:da7381797612083654] 
\draw [color={rgb, 255:red, 74; green, 144; blue, 226 }  ,draw opacity=1 ]   (436.2,121.2) -- (480.72,251.46) ;
%Straight Lines [id:da1926614631455077] 
\draw [color={rgb, 255:red, 0; green, 0; blue, 0 }  ,draw opacity=1 ]   (428.56,101.38) -- (390.06,101.38) ;
%Straight Lines [id:da904844691061468] 
\draw    (410.32,51.52) -- (390.06,101.38) ;
%Straight Lines [id:da695247340500754] 
\draw    (410.32,51.52) -- (428.56,101.38) ;
%Straight Lines [id:da39182152624186584] 
\draw [color={rgb, 255:red, 74; green, 144; blue, 226 }  ,draw opacity=1 ]   (363.2,168.2) -- (329.92,251.58) ;
%Straight Lines [id:da13765428279805558] 
\draw [color={rgb, 255:red, 74; green, 144; blue, 226 }  ,draw opacity=1 ]   (363.2,168.2) -- (385.2,168.2) ;
%Straight Lines [id:da3106879464105352] 
\draw [color={rgb, 255:red, 74; green, 144; blue, 226 }  ,draw opacity=1 ]   (388.2,140.2) -- (385.2,168.2) ;
%Straight Lines [id:da35055640423462253] 
\draw [color={rgb, 255:red, 74; green, 144; blue, 226 }  ,draw opacity=1 ]   (388.2,140.2) -- (402.2,140.2) ;
%Straight Lines [id:da508588952551152] 
\draw [color={rgb, 255:red, 74; green, 144; blue, 226 }  ,draw opacity=1 ]   (409.31,101.38) -- (402.2,140.2) ;
%Straight Lines [id:da10383592664047825] 
\draw [color={rgb, 255:red, 74; green, 144; blue, 226 }  ,draw opacity=1 ]   (409.31,101.38) -- (419.2,121.2) ;
%Straight Lines [id:da32048262309301845] 
\draw [color={rgb, 255:red, 74; green, 144; blue, 226 }  ,draw opacity=1 ]   (419.2,121.2) -- (436.2,121.2) ;
%Straight Lines [id:da7425091896838862] 
\draw    (428.56,101.38) -- (436.2,121.2) ;

% Text Node
\draw (75.56,246) node [anchor=north west][inner sep=0.75pt]   [align=left] {\textit{x}};
% Text Node
\draw (242.72,246) node [anchor=north west][inner sep=0.75pt]   [align=left] {\textit{y}};
% Text Node
\draw (152,34) node [anchor=north west][inner sep=0.75pt]   [align=left] {$x\wedge y$};
% Text Node
\draw (134,88) node [anchor=north west][inner sep=0.75pt]   [align=left] {$\wt x$};
% Text Node
\draw (192,88) node [anchor=north west][inner sep=0.75pt]   [align=left] {$\wt y$};
% Text Node
\draw (313.56,247.5) node [anchor=north west][inner sep=0.75pt]   [align=left] {\textit{x}};
% Text Node
\draw (482.72,254.46) node [anchor=north west][inner sep=0.75pt]   [align=left] {\textit{y}};
% Text Node
\draw (393,34) node [anchor=north west][inner sep=0.75pt]   [align=left] {$x\wedge y$};

\end{tikzpicture}
\caption{The red line represents a standard geodesic connecting 
$x$ to $y$, while the blue line depicts an alternative path between them.}
\end{figure}
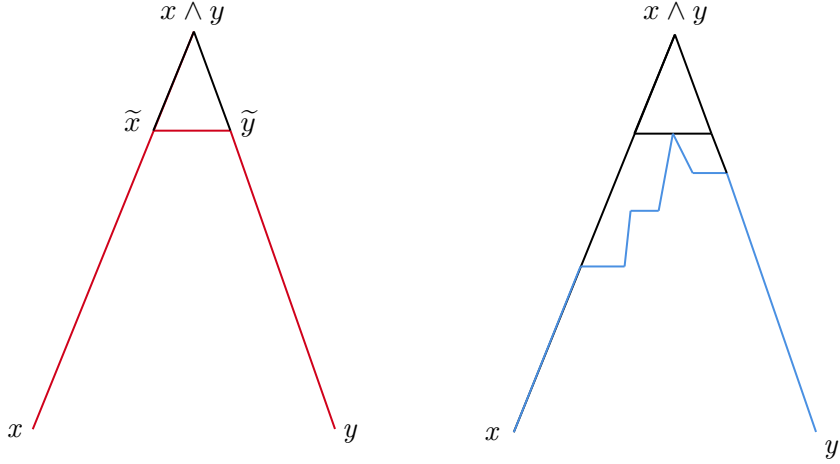

	Next we prove \rmii.  Suppose that $\ga$ is a standard geodesic joining $x$ and~$y$.  The horizontal part $\wt\ga$ of $\ga$ is a geodesic contained
	in $\Si_n$ for some nonnegative integer $n$, with endpoints  $\wt x := p^{h(x)-n}(x)$ and $\wt y := p^{h(y)-n}(y)$. 

	Consider the point $w:=x\wedge y$ (the \textit{confluent} of $x$ and $y$ with respect to the root $o$, i.e., the last point in common between the geodesics
	joining $o$ to $x$ and $y$).
	Denote by $L$ the length of $\wt{\ga}$ and consider the point $p$ in $\wt\ga$ such that $d_{\wh\Ga}(p,\wt y) =~L/2$ in the case where~$L$ is even, and
	$d_{\wh\Ga}(p,\wt y) = \lfloor L/2\rfloor+1$ in the case where $L$ is odd.  

	The spider's web $\wh\Ga$ is, by assumption, $\de$-hyperbolic.  Then the four points condition \eqref{f: Gromov}, applied to $w$, $\wt x$, $\wt y$ and $p$, yields 
	$$
	\begin{aligned}
		& d_{\wh\Ga}(p,\wt x) +  d_{\wh\Ga}(p,\wt y) - d_{\wh\Ga}(\wt x,\wt y) \\
		& \geq \min\big(d_{\wh\Ga}(p,w) +  d_{\wh\Ga}(p,\wt x) - d_{\wh\Ga}(w,\wt x), d_{\wh\Ga}(p,w) +  d_{\wh\Ga}(p,\wt y) - d_{\wh\Ga}(w,\wt y) \big) - 2 \de.
	\end{aligned}
	$$
	Notice that $d_{\wh\Ga}(p,w) = d_{\wh\Ga}(w,\wt x) = d_{\wh\Ga}(w,\wt y)$, and that $d_{\wh\Ga}(p,\wt x) +  d_{\wh\Ga}(p,\wt y) = L$.  
	Thus, the inequality above simplifies to
	$$
		\min\big(d_{\wh\Ga}(p,\wt x),  d_{\wh\Ga}(p,\wt y) \big) \leq 2 \de.
	$$
	Now our choice of $p$ implies that $d_{\wh\Ga}(p,\wt x) \leq 2\de$, so that $d_{\wh\Ga}(p,\wt y) \leq 2\de+1$ and $L \leq 4\de +1$, as required.  

	This concludes the proof of the proposition.  
\end{proof}

Suppose that $\wh\Ga$ is a spider's web associated to the rooted tree $\mrT$, that
$x$ is a vertex in $\wh\Ga$ and $r$ is a nonnegative integer.  We denote by $B_r^\mrT(x)$ and $B_r^{\wh\Ga}(x)$ the balls with centre $x$ and radius $r$
in $\wh\Ga$ with respect to the distances $d_\mrT$ and $d_{\wh\Ga}$, respectively. 

The existence of standard geodesics joining any two points in a spider's web, established in Proposition~\ref{p: geodesics},  
has some noteworthy geometric consequences, one of which, concerning the volume of geodesic balls, is discussed in the following corollary.

\begin{corollary} \label{c: geodesics}
	Suppose that $\mrT$ is a rooted tree satisfying \eqref{f: pinched exp}, and that ${\wh\Ga}$ is a spider's web associated to $\mrT$.  The following hold:
	\begin{enumerate}
		\item[\itemno1]
			$d_{\wh\Ga} \leq d_\mrT$ and $\mu_{\wh\Ga}\big({B_r^{\wh\Ga}(x)}) \geq \mu_{\wh\Ga}\big(B_r^\mrT(x)\big)$; 
		\item[\itemno2] if, in addition, $\wh \Ga$ has  bounded valence, then there exist positive constants $c$ and $C$ such that 
			$$
			c\, a^r \leq \mu_{\wh\Ga}\big({B_r^{\wh\Ga}(x)}\big) \leq C\, b^r
			\quant r \in \BN. 
			$$ 
	\end{enumerate}
\end{corollary}

\begin{proof}
	First we prove \rmi.  
	Clearly every curve in $\mrT$ joining two vertices $x$ and $y$ is also an admissible curve joining $x$ and $y$ in ${\wh\Ga}$, whence 
	$d_{\wh\Ga} \leq d_\mrT$.  The second statement in \rmi\ is a direct consequence of this. 

	Next we prove \rmii.  The left hand inequality is an obvious consequence of~\rmi.  It remains to prove the right hand inequality.
	
	Consider a geodesic $\ga$ in ${\wh\Ga}$ of length $r$ joining $x$ and~$y$.  According to Proposition~\ref{p: geodesics}, we may assume that
	$\ga$ has one of the following forms:
	\begin{enumerate}
		\item[(a)]
			$\ga$ moves away from the root $o$ of $\mrT$ at each step; 
		\item[(b)]
			$\ga$ moves $h$ steps joining points at the same level as $x$, until it reaches a point $z$ and then moves 
			$r-h$ steps away from $o$;
		\item[(c)]
			$\ga$ moves $\upsilon$ steps towards $o$, then goes horizontally $h$ steps, and in the end it moves  $r-\upsilon-h$ steps away from $o$.  
	\end{enumerate}

	The geodesics $\ga$ of the form (a) are precisely the geodesics in $\mrT$ starting from $x$ and moving downwards at each step.  Thus,
	the set of points in ${\wh\Ga}$ we can reach with such geodesics is $s^r(x)$  (see just above \eqref{f: Sik} for the definition of $s^r(x)$).  

	Similarly, the vertices that we can reach with geodesics of the form (b) are precisely those in $s^{r-h}(z)$.  Clearly $s^{r-h}(z)$ is contained
	in $B_{r-h}^\mrT(z)$, which, by the right hand inequality in \eqref{f: pinched exp}, has cardinality at most $C\, b^{r-h}$.   

	By Proposition~\ref{p: geodesics}, the estimate $h\leq 4\de +1$ holds, so that the number of vertices
	of ${\wh\Ga}$ reachable from $x$ with a geodesic starting with a horizontal segment is at most
	$$
	C\, \sum_{h=1}^{H}\, b^{r-h}   
	\leq C\, b^r,
	$$
	where $H := \lfloor 4\de\rfloor +2$ and $C$ depends also on the maximum of the valence function on $\wh\Ga$.  

	A similar reasoning shows that the number of vertices of ${\wh\Ga}$ reachable from~$x$ with a geodesic of the form (c) is, at most, 
	$$
	C\, \sum_{\upsilon=1}^r\, \sum_{h=0}^{\min(r-\upsilon,H)} \, b^{r-\upsilon-h}
	\leq C\, b^r.
	$$

	Altogether, we see that the cardinality of the sphere $S_r^{\wh\Ga}(x)$ is, at most, $C\, b^r$.  Therefore
	$$
	\mu_{\wh\Ga}\big(B_r^{\wh\Ga}(x)\big) 
	= \sum_{j=0}^r \, \mu_{\wh\Ga}\big(S_j^{\wh\Ga}(x)\big)  
	\leq C\, \sum_{j=0}^r \, b^j
	\leq C\, b^r,
	$$
	as required.  
\end{proof}

\section[Bounds spider webs]{$L^p$ bounds for $\cM_\infty$ on spider's webs} \label{s: bounds on spider webs}

In this section we prove $L^p$ bounds for the centred HL maximal operator on spider's webs satisfying some mild additional assumptions.  Our approach is
a variant of the strategy that A.~Naor and T.~Tao follow to prove that $\cM$ is of weak type $(1,1)$ on homogeneous trees
(see \cite{NT}, in particular Lemma~5.1 and the proof of Theorem~1.5).  Such strategy was then generalised in \cite[Theorem~4.1]{ST} and used in \cite[Lemma~3.3]{LMSV}.  
For different proofs of the weak type $(1,1)$ estimate for $\cM$, see \cite{RT} and \cite{CMS2}.

We consider the metric measure space $\big(\wh\Ga,d_{\wh\Ga},\mu_{\wh\Ga}\big)$.  Notice that in this case $\cM_0f = f$ (see \eqref{f: local max} for the definition 
of $\cM_0$), so that so that $\cM$ is bounded on $\lp{\wh\Ga}$, or it is of weak type $(p,p)$, if and only if so is the operator $\cM_\infty$.  

In order to avoid cumbersome notation, in this section the cardinality of the set $E$ will be denoted by $\mod{E}$.  For a positive integer $r$, set 
$$
{\bB}_r(x) := \big\{y \in \wh\Ga: d_{\wh\Ga}(y,x) \leq r\big\}.
$$
Given a function $f$ on ${\wh\Ga}$, denote by $A_r f(x)$ the mean of $f$ over ${\bB}_r(x)$, i.e.
$$
A_rf(x)
:= \frac{1}{\bigmod{{\bB}_r(x)}} \, \int_{{\bB}_r(x)} f \wrt \mu_{\wh\Ga}.  
$$
Note the trivial bound 
$$
\cM_\infty f 
\leq \sum_{r=1}^\infty\, A_r\mod{f}.  
$$
The next theorem is the main result of this section.  Recall that if $a$ and~$b$ are positive numbers, then $\tau$ stands for $\log_ab$. 
\begin{theorem} \label{t: NT for spider web}
	Suppose that $a\leq b < a^2$, and that ${\wh\Ga}$ is a spider's web (associated to the tree $\mrT$) satisfying the growth condition \eqref{f: pinched exp}. 
	The following hold:
	\begin{enumerate}
		\item[\itemno1] 
			there exists a positive constant~$C$ such that 
			\begin{equation} \label{f: cond m} 
				a^{-r} \,\bigmod{\big\{(x,y) \in E\times F: d_{\wh\Ga}(x,y) \leq r\big\}} \leq  C\, (\sqrt b/a)^r \, \sqrt{|E| \cdot |F|}
			\end{equation}
			for every nonnegative integer $r$ and every pair $E$, $F$ of subsets of ${\wh\Ga}$; 	
		\item[\itemno2]
			$\cM_\infty$ is of weak type $(\tau,\tau)$, and bounded on $\lp{{\wh\Ga}}$ for every $p$ in $(\tau,\infty]$. 
	\end{enumerate}
\end{theorem}

\begin{proof}
	First we prove \rmi.  Set $U_{r} := \{ (x,y) \in E \times F: d_{\wh\Ga}(x,y)\leq r\}$, and observe that $\ds U_r = \sum_{p=0}^r G_p$, where 
	$G_{p} := \{ (x,y) \in E \times F: d_{\wh\Ga}(x,y)=p\}$.  We shall prove that 
	\begin{equation} \label{f: reduction r to p}
		\mod{G_p} \leq C\, b^{p/2} \, \sqrt{\mod{E}\cdot \mod{F}}
	\end{equation}
	for every $p$ in $\{0,\ldots, r\}$, which implies
	$$
		\mod{U_r} \leq C\, b^{r/2} \, \sqrt{\mod{E}\cdot \mod{F}}.
	$$
	Since the left hand side of \eqref{f: cond m} is just~$a^{-r} \, \mod{U_r}$, part \rmi\ follows directly from the previous estimate.

	It remains to prove \eqref{f: reduction r to p}.   

	Suppose that $x$ belongs to $\Si_j$.  We estimate the number of the vertices in~$\Si_k$ at distance $p$ from $x$, i.e. the cardinality
	of $S_p(x) \cap \Si_k$.

	Denote by $\ga$ a geodesic of length $p$ joining $x$ and a point in $\Si_k$.  Recall that by Proposition~\ref{p: geodesics} we may assume that $\ga$ 
	is a standard geodesic, i.e. it has one of the special forms described in (a), (b) and (c) in the proof of Corollary~\ref{c: geodesics}.  
	In particular, denote by $\ga_0$ the horizontal part (possibly reduced to a point) of~$\ga$, and set $\ell := d_{\wh\Ga}(x,\ga_0)$.  
	Then $0\leq \ell \leq \min(j,p)$.  Thus $\ga$ moves first $\ell$ steps towards $o$, then  $h$ steps horizontally, and finally 
	$p-\ell-h$ steps away from $o$.  Note also that $p-\ell-h$ is equal to the difference between $k$, the level of a point in $\Si_k$, and the level of $\ga_0$, i.e.
	$j-\ell$.  Thus, $p= \ell + h + k-(j-\ell)$, so that 
	\begin{equation} \label{f: j and k}
	\ell =\frac{1}{2} \, (p+j-k-h).
	\end{equation}
	Since $p$, $j$ and $k$ are fixed parameters, \eqref{f: j and k} exhibits $\ell$ as a function of $h$.   
	Recall that~$h$ is a nonnegative integer~$\leq H$, where we have set $H := 4\de+1$ (see Proposition~\ref{p: geodesics}).  

	Now, given an integer $h$ in $[0,H]$, how many points in $\Si_k$ are reachable from $x$ with a standard geodesic of length $p$ and horizontal part 
	of length~$h$?  We start from $x$ and move $\ell = \big(p+k-j-h\big)/2$ steps towards $o$ until we reach $p^{\ell}(x)$, which belongs to $\Si_{j-\ell}$.  
	Next we move horizontally $h$ steps, reaching a point $z_\ga$.  The number of points $z_\ga$ in $\Si_{j-\ell}$ that we can reach in this way is bounded from 
	above by the volume of the ball with centre $p^{\ell}(x)$ and radius $h$, which has, at most, $C\, b^h$ points.  Clearly 
	$$
	d_{\wh\Ga} \big(z_\ga, \Si_k\big) 
	= p-\ell-h
	= \frac{1}{2}\, \big(p+k-j-h\big).
	$$   
	Thus, the number of points in $\Si_k$ reachable from $z_\ga$ is bounded by the measure of the ball centred at $z_\ga$ and of radius $\big(p+k-j-h\big)/2$,
	which has, at most, $C\, b^{(p+k-j-h)/2}$ points.  

	Altogether, the number of points in $\Si_k$ reachable from $x$ with a standard geodesic of length $p$ and horizontal part of length~$h$ is
	bounded above by 
	$$
	C\,b^h\, b^{(p+k-j-h)/2}
	= C\, b^{(p+k-j+h)/2}
	$$ 
	points.  Hence 
	\begin{equation} \label{f: product I}
	\bigmod{S_p(x)\cap \Si_k}
	\leq C\, \sum_{h=0}^H \, b^{(p+k-j+h)/2}
	= C\, b^{(p+k-j)/2}.  
	\end{equation}
	
 	A similar reasoning yields that if $y$ belongs to $\Si_k$, then
	\begin{equation} \label{f: product II}
		\bigmod{S_p(y)\cap \Si_j}
		\leq C\, b^{(p+j-k)/2}.
	\end{equation}
	Furthermore, for every pair $(j,k)$ of nonnegative integers define 
	$$
	E_j:=E \cap \Si_j,
	\qquad
	F_k:=F \cap \Si_k,
	$$
	and 
	$$
	G_{j,k,p}
	:= \big\{ (x,y) \in E_j \times F_k: d_{\wh\Ga}(x,y)=p\big\}.  
	$$ 
	Observe that 
	$$
	\mod{G_{j,k,p}}
	=    \sum_{x \in E_j } \sum_{y \in F_k} \One_{S_p(x)}(y) 
	=    \sum_{y \in F_k } \sum_{x \in E_j} \One_{S_p(y)}(x).  
	$$
	Define $e_j:=\mod{E_j}/b^j$ and $d_k:=\mod{F_{k}}/b^{k}$.  Now \eqref{f: product I} and \eqref{f: product II} imply that
	$$
	\begin{aligned}
		\mod{G_{j,k,p}}
		& \le  C \,  \min\, \big(\mod{E_j} \, b^{(p-j+k)/2},\mod{F_k}\, b^{(p+j-k)/2}\big) \\
		& =    C \, b^{(p+j+k)/2}\, \min\, (e_j ,d_k).
	\end{aligned}
	$$
	Therefore
	$$
	\mod{G_p}
		=   \sum_{j,k=0}^\infty \, \mod{G_{j,k,p}}
		\leq C\, b^{p/2} \, \sum_{j,k=0}^\infty \, b^{(j+k)/2} \,  \min\, \big(e_j,d_k\big). 
	$$
	Now we can apply almost \textit{verbatim} the argument in the last part of the proof of \cite[p. 759--760]{NT}, and conclude that the last double series
	is dominated by $8\sqrt{\mod{E} \cdot \mod{F}}$.  The estimate \eqref{f: reduction r to p} follows directly from this.  

	\medskip
	Next we prove \rmii.  It suffices to consider nonnegative functions $f$ in~$L^{\tau}({\wh\Ga})$.  We want to prove that there exists a constant $C>0$ such that
	$$
	\bigmod{\{\cM_\infty f > \la \}} 
	\leq \frac{C}{\la^\tau} \bignormto{f}{L^\tau}{\tau}
	\quant \la > 0
	\quant f \in L^\tau({\wh\Ga}).  
	$$
	By replacing $f$ by $f/\la$, we see that it suffices to prove the estimate above for $\la=1$.  
	Set $\Om := \big\{x \in {\wh\Ga}: f(x) \leq 1/2\big\}$, and, for integers $n$ and $r$, define
	$$
	E_n := \big\{x \in {\wh\Ga}: 2^{n-1} < f(x) \le 2^n\big\}
	\quad\hbox{and}\quad 
	F_r := \big\{x \in {\wh\Ga}: f(x) > a^r/2\big\}.
	$$ 
	We fix $r$, and decompose $f$ accordingly.  Denote by $N$ the biggest integer such that $2^n \le a^r$.  Clearly 
	$$
	{\wh\Ga} =  \Om \cup \Big(\bigcup_{n=0}^N E_n\Big) \cup F_r,
	$$
	so that 
	$$
	f 
	\le \frac{1}{2}\, \One_\Om + \sum_{n=o}^N  \, 2^n \, \One_{E_n}+f\, \One_{F_r}.
        $$ 
	Consequently
	$$
	A_rf 
	\le \frac{1}{2} + \sum_{n=o}^N  \, 2^n \, A_r\One_{E_n}+A_r\big(f\, \One_{F_r}\big).
	$$
	Define $\ds g_r := \sum_{n=o}^N  \, 2^n \, A_r\One_{E_n}$ and $h_r :=  A_r\big(f\, \One_{F_r}\big)$.  Clearly 
	$$
	\big\{A_rf >1\big\} \subseteq \big\{g_r + h_r > 1/2\big\},
	$$ 
	and the latter set is contained in 
	$$
	\big\{g_r > 1/2\} \cup \{ h_r > 0\big\}. 
	$$
	Denote by $I_1$ and $I_2$ the first and the second set in the union above.  Clearly $\mod{\big\{A_rf >1\big\}} \leq \mod{I_1} + \mod{I_2}$.  
	We estimate $\mod{I_1}$ and $\mod{I_2}$ separately.

	First we estimate $|I_2|$.  Observe that if $d_{\wh\Ga}(x,F_r) > r$, then $h_r(x) = 0$, so that $I_2$ is contained in $\ds \bigcup_{y\in F_r} \, {\bB}_r(y)$, 
	Therefore
	$$
	|I_2|
	\le \sum_{y \in F_r} \, |{\bB}_r(y)| 
	\le C \, b^r\,  \mod{F_r}
	=   C \, a^{\tau r} \, \mod{F_r}. 
	$$ 

	Now we estimate $\mod{I_1}$.  Set  $\be := (2-\tau)/4$: notice that $\be$ is positive, for $\tau <2$ by assumption.  It is convenient to set
	$\al :=a^{-\be r}\, \, (1-2^{-{\be}})$ and 
	$$
	V_n := \big\{x: 2^n\, A_r(\One_{E_n})(x) \geq  2^{n\be-1} \,  \al\big\}.
	$$ 
	By arguing almost \textit{verbatim} as in the proof of \cite[p. 501]{OR}, one can show that if $x$ belongs to $I_1$, 
	then $x$ belongs to $V_n$ for some integer $n$ in $\{0,\ldots,N\}$.  Thus, 
	\begin{equation} \label{f: est I1}
        	|I_1| 
		\leq \sum_{n=0}^N  \, |V_n|.  
    	\end{equation} 
	Now, since $2^n\, A_r(\One_{E_n})\geq  2^{n\be-1} \,  \al$ on $V_n$, we see that 
	\begin{equation}  \label{f: lower duality}
		\prodo{\One_{V_n}}{ 2^n\, A_r(\One_{E_n})}
		\ge |V_n|\, 2^{n\be-1} \,  \al,
	\end{equation} 
	where $\langle \cdot, \cdot\rangle$ denotes the inner product on $L^2(\mu_{\wh \Ga})$.  Also the lower estimate in \eqref{f: pinched exp} yields 
	$$
	\begin{aligned}
	\langle \One_{V_n}, 2^n\, A_r(\One_{E_n})\rangle 
		& \leq \frac{2^n}{c \, a^r} \, \int_{V_n} \wrt\mu_{\wh \Ga}(y) \int_{{\bB}_r(y)} \, \One_{E_n} (x) \wrt \mu_{\wh \Ga}(x) \\
		& = \frac{2^n}{c \, a^r} \,\, \bigmod{\big\{(x,y) \in E_n \times V_n \ : d_{\wh\Ga}(x,y)\leq r\big\}}
	\end{aligned}
	$$
	This, combined with  \rmi\ and the lower estimate \eqref{f: lower duality}, implies that 
	$$
	|V_n| 
	\leq C \, \frac{1}{2^{n\be} \,  \al} \, \, 2^n\, (\sqrt b/a)^r \, \sqrt{\mod{V_n}\, \mod{E_n}}.
	$$
	Now recall that $b = a^\tau$.  Then the previous estimate may be re-written as 
	$$
	|V_n| 
		\leq C \, \frac{1}{2^{2n\be} \,  \al^2} \, \, 2^{2n}\, (b/a^2)^r \,  \mod{E_n} 
		= C\, \Big(\frac{2^n}{a^r} \Big)^{1-\tau/2} \, 2^{n\tau} \, \mod{E_n},
	$$
	which, together with \eqref{f: est I1}, yields
	$$
        	|I_1| 
		\leq C\, \sum_{n=0}^N  \, \Big(\frac{2^n}{a^r} \Big)^{1-\tau/2} \, 2^{n\tau} \, \mod{E_n}.  
	$$
	Recall that $\big\{A_r f>1\big\}$ is contained in $I_1\cup I_2$, so that 
	\begin{equation} \label{f: Ar final}
		\bigmod{\big\{A_r f>1\big\}}
	\le C \, \sum_{n=0}^N  \, \Big(\frac{2^n}{a^r} \Big)^{1-\tau/2} \, 2^{n\tau} \, \mod{E_n} + C \, a^{\tau r} \, \mod{F_r}.  
	\end{equation}
	Now, $\ds\cM_\infty f = \sup_{r \ge 1}\, A_r f$.  Therefore if $\cM_\infty f(x) >1$, 
	then there exists a nonnegative integer $r$ such that $A_r f(x) > 1$.   Consequently
	$$
	\begin{aligned}
	\bigmod{\big\{\cM_\infty f>1\big\}}
		& \leq \sum_{r=0}^\infty \,  \bigmod{\big\{A_r f>1\big\}}  \\
		& \leq C \, \sum_{r=0}^\infty \,\sum_{n=0}^N  \, \Big(\frac{2^n}{a^r} \Big)^{1-\tau/2} \, 2^{n\tau} \, \mod{E_n} 
		    + C \, \sum_{r=0}^\infty \,a^{\tau r} \, \bigmod{F_r}.  
	\end{aligned}
	$$
	Now, trivially 
	$$
	\ds \sum_{r=0}^\infty \,a^{\tau r} \, \One_{F_r}  
	\leq \sum_{r=0}^{\lfloor \log_a (2 f) \rfloor} \,a^{\tau r} \, 
	\le  \frac{2^\tau b}{b-1} \, f^\tau,
	$$ 
	whence 
	$$
	\sum_{r=0}^\infty \,a^{\tau r} \, \mod{F_r}   
	= \int_{{\wh\Ga}} \,  \sum_{r=0}^\infty \,a^{\tau r} \, \One_{F_r}  \wrt  \mu_{\wh\Ga}  
	\leq \frac{2^\tau b}{b-1} \bignormto{f}{L^\tau({\wh\Ga})}{\tau}.  
	$$
	Similarly, $\ds \sum_{n=0}^\infty \,2^{\tau n} \, \One_{\wh E_n}  \leq 2^{\tau+1} \, f^\tau$, where we have set $\wh E_n := \{f>2^{n-1}\}$.  Clearly $E_n\subseteq \wh E_n$.
	Using these facts and reversing the order of summation, we see that 
	\begin{align*}
   		%|\{\cM f>1\}| 
		\sum_{r=0}^\infty \sum_{n=0}^N \Big(\frac{2^n}{a^r}\Big)^{1-\tau/2}\, 2^{\tau n}\, |E_n|
		& \leq  \int_{{\wh\Ga}}\sum_{n =0}^\infty 2^{\tau n}\,\One_{\wh E_n}  \sum_{r:a^r\ge 2^{n+1}}\bigg(\frac{2^n}{a^r}\bigg)^{1-\tau/2}\wrt\mu_{\wh\Ga}\\
		& \le C\,  \int_{{\wh\Ga}}\sum_{n=0}^\infty \, 2^{n\tau}\, \One_{\wh E_n} \wrt\mu_{\wh\Ga}\\
		& \le C\,  \int_{{\wh\Ga}} f^{\tau} \wrt\mu_{\wh\Ga}.
	\end{align*} 
	By combining these estimates, we find that 
	$$
	\mod{\{\cM_\infty f>1\}}
	\leq C \bignormto{f}{L^\tau({\wh\Ga})}{\tau},   
	$$
	as desired.

	This concludes the proof of \rmii, and of the theorem.
\end{proof}

\begin{remark}
	Recall that a tree $\mrT$ is a very special spider's web.  Thus, if $\mrT$ satisfies \eqref{f: pinched exp} for some $b<a^2$, then Theorem \ref{t: NT for spider web} 
	implies that the maximal operator $\cM$ on $\mrT$ is of weak type $(\tau,\tau)$. This improves \cite[Theorem 3.2]{LMSV}, in which it is shown that 
	$\cM$ is of {\it restricted} weak type $(\tau,\tau)$ under stronger assumptions.
\end{remark}

\section[Discretisation of Gromov]{Spider's webs ``embedded" in $\de$-hyperbolic length spaces} \label{s: Discretisation of Gromov} 

Given a positive number $\eta$, we say that a set $\fD$ of points in a metric space $(Y,d)$ is $\eta$-\textit{separated} if $d(y,z) \geq \eta$ 
for every pair of distinct points $y$ and $z$ in $\fD$;  the set $\fD$ is an \emph{$\eta$-discretisation} of~$Y$ if it is a maximal (with respect to inclusion)
$\eta$-separated set in $Y$.
It is straightforward to show that $\eta$-discretisations exist for every $\eta$.     

Notice that if $\fD$ a discretisation of $Y$, then 
$$
d(\fD, x) < \eta 
\quant x \in Y,
$$
for otherwise there would exist a point $x$ in $Y$ such that $d(x,\fD) \geq \eta$, thereby violating the maximality of $\fD$.

We denote by $\Upsilon_{a,b}^\de$ the class of all $\de$-hyperbolic spider's webs $(\wh\Ga,d_{\wh\Ga})$ with the property that the counting measure $\mu_{\wh\Ga}$
on $\wh\Ga$ satisfies the growth condition~\eqref{f: pinched exp}.  

The proof of the main result in this section, Theorem~\ref{t: sri spider} below, hinges on a nontrivial variant of the construction described in \cite[Section~5]{BI}.  
Recall the definition of the class $\cX_{a,b}^\de$ given just above Theorem~\ref{t: main}.

\begin{theorem} \label{t: sri spider}
	Suppose that $\de$, $a$ and $b$ are positive numbers satisfying the condition $1<a\leq b$.  Assume that $(X,d,\mu)$ is in 
	the class $\cX_{a,b}^\de$.  Then there exist a positive number $\de'$ and a spider's web $\wh\Ga$ in the class $\Upsilon_{a,b}^{\de'}$
	that is strictly roughly isometric to $(X,d)$.  
\end{theorem}
The proof consists of four big steps, each of which is further split up into smaller steps.  Before we dive into the details of the proof we shall outline 
the general scheme.  In order to construct the desired spider's web $\wh {\Ga}$ we begin by defining the underlying tree $\mrT$ and an associated graph $\Ga$, 
obtained from $\mrT$ by adding {\it horizontal} edges to the tree between points that are close enough in the metric space $X$.  The graph $\Ga$
turns out to have bounded valence (Step~I).

We observe that there is no \textit{a priori} reason for which the graph $\Ga$ should be a spider's web, because it might happen
that two points $x$ and $y$ in $\Ga$ with the same level are neighbours, but their predecessors $p(x)$ and $p(y)$ are not.  In Step~IV
we shall construct a spider's web~$\wh\Ga$, which has the same vertices as $\Ga$, but where each vertex $x$ may have 
more neighbours at its level than in $\Ga$.  

In order to do this, we need to establish a few preliminary facts, which occupy Steps~II-III below.  Recall the metric graphs $\Ga_0$ and
$\wt\Ga$ associated to~$\Ga$ (see Caveat~\ref{cav: metric graphs}).

Step~II is devoted to the proof that $\Ga_0$, $\wt\Ga$ and $\Ga$ are $\de'$-hyperbolic.  Step~III, which contains the most technical part of the proof,
is dedicated to the proof that $\Ga_0$ is strictly roughly isometric to $(X,d)$.  
Morse's lemma plays a crucial role in deriving the desired estimates here and in Step~IV.

The last remaining link is proving that, 
although $\Ga$ might not be a spider's web, because of hyperbolicity it is always a quasi-spider's web and therefore can be ``completed'' 
to a true spider's web $\wh \Ga$, which is still strictly roughly isometric to $(X,d)$ (Step~IV).

\begin{proof}
	\textbf{Step~I: construction of the graph $\Ga$.}  This step is split up into Step~I$_1$, where we construct the tree $\mrT$, and Step~I$_2$,
	where we define the graph $\Ga$ and prove that it has bounded valence.   

	\textit{Step~I$_1$: construction of the tree $\mrT$.}
	Fix a point $o$ in $X$.  We define a rooted tree $\mrT$ with root $o$ embedded in $X$ as follows.  For each positive integer~$n$ consider the sphere 
	$S_n(o) := \{x \in X: d(x,o) = n\}$, and a $1$-discretisation~$\Si_n$ thereof (with respect to the distance $d$).  In particular,
	$$
	\bigcup_{x\in \Si_n} \, B_1(x) \supseteq S_n(o).
	$$
	Notice that 
	\begin{equation} \label{f: covering}
	B_1(o)\cup \Big(\bigcup_{n=1}^\infty\,\bigcup_{x\in \Si_n} \, B_2(x)\Big) = X. 
	\end{equation}
	Indeed, consider a point $y$ in $X$.  If $d(y,o) < 1$, then $y$ is covered by $B_1(o)$, and if $d(y,o) = n$ for some positive integer $n$, 
	then we already know that $y$ is covered by $B_2(x)$ for some $x$ in $\Si_n$.  

	Otherwise denote by $n$ the positive integer such that $n < d(y,o) \leq n+1$. 
	Consider the point $y_n$ where the geodesic joining $o$ and $y$ intersects the sphere $S_n(o)$. We have that 
	\[
		d(y_n,y) = d(o,y)-d(o,y_n) = d(o,y)-n \le 1. 
	\]
	By construction of $\Si_n$, there exists a point $x$ in $\Si_n$ such that $d(x,y_n)<1$, whence $y$ belongs to $B_2(x)$.
 
	Set $\ds V := \{o\} \cup\Big(\bigcup_{n=1}^\infty \, \Si_n\Big)$.  The points in $V$ will be the vertices of the rooted tree $\mrT$.
	It remains to indicate, for each $x$ in $V\setminus\{o\}$, its predecessor~$p(x)$.  
	If $x$ belongs to $\Si_1$, we set $p(x) = o$.  If $x$ is in~$\Si_n$ for some $n\geq 2$, then its predecessor $p(x)$ is one, no matter which, of the points 
	in $\Si_{n-1}$ at minimum distance from $x$.  The construction above implies that 
	\begin{align} \label{f: d(x,p(x))}
	d\big(x,p(x)\big) \le 2 \quant x \in \mrT.    
	\end{align}
	Of course different choices of $p(x)$ will give rise to different trees.  All of them share the properties we shall need in the sequel.
	Note also that there may be vertices~$x$ in~$\mrT$ such that $s(x)$ is empty, i.e. $x$ has no successors. 

	\smallskip
	\textit{Step~I$_2$: definition of the graph $\Ga$.}  The vertices of $\Ga$ agree with the vertices of $\mrT$.  Set 
	\begin{equation} \label{f: te0}
	\boxed{
		\te := \max(15,4D_{2,4}+2\de),}
	\end{equation}
	where $\de$ is the hyperbolicity constant of $(X,d)$ and $D_{2,4}$ is the constant appearing in the definition of $(2,4)$-quasi-geodesics in~$(X,d)$.
	We \textit{declare} that two distinct vertices~$x$ and $y$ in $\Ga$ are neighbours if and only if one of the following holds:
	\begin{enumerate}
		\item[\itemno1]
			$x$ and $y$ are neighbours in $\mrT$;
		\item[\itemno2]
			they belong to the same level set $\Si_n$ for some positive integer $n$ and $d(x,y) \leq \te$.
	\end{enumerate}
	
	Observe that \textit{the valence function $\nu$ on $\Ga$ is bounded}.  
	Indeed, if $x\sim y$ in $\Ga$, then either $x$ and $y$ are neighbours in $\mrT$, whence $d(x,y) < 2$ by the construction of $\mrT$
	(see Step~I above), or they belong to the same level and $d(x,y) \leq \te$.  Altogether any neighbour of $x$ in $\Ga$ belongs
	to $B_{\te}(x)$.  Recall that the points in $\Ga$ are $1$-separated in $X$, so that 
	$$
	\mu\big(B_{\te+1}(x)\big)
	\geq \sum_{y\sim x} \mu\big(B_{1/2}(y)\big)
	\geq \nu(x) \, \min_{y\sim x} \mu\big(B_{1/2}(y)\big).  
	$$
	Now, the LDP (see Remark~\ref{r: geom I}) implies that $\nu(x) \leq L_{2\te+2,1/2}$, as required.

	\smallskip
	\textbf{The metric graphs $\Ga_0$ and $\wt\Ga$ are associated to $\Ga$ as in Caveat~\ref{cav: metric graphs}.}

	\smallskip
	\textbf{Step II: $(\wt\Ga,d_{\wt\Ga})$, $(\Ga_0,d_{\Ga_0})$ and $(\Ga,d_\Ga)$ are $\de'$-hyperbolic for some positive~$\de'$}.
	This step is split up into Step~II$_1$, where we prove that if two points in $\Ga$ are close enough with respect to $d$, then they are close
	enough also with respect to $d_\Ga$, and Step~II$_2$, where we show that $(\wt\Ga,d_{\wt\Ga})$ and $(X,d)$ are roughly isometric.  The 
	$\de'$-hyperbolicity of $(\wt\Ga,d_{\wt\Ga})$ then follows from Step~II$_2$ and Theorem~\ref{t: rqi hyp}. The $\de'$-hyperbolicity of $(\Ga_0,d_0)$ 
	and $(\Ga,d_\Ga)$ is a direct consequence of the $\de'$-hyperbolicity of $(\wt\Ga,d_{\wt\Ga})$ and of Proposition~\ref{p: Ga Ga_0 wtGa}.  Finally,
	the $\de'$-hyperbolicity of $(\Ga,d_{\Ga})$ follows from the fact that $d_\Ga$ is the restriction of $d_{\Ga_0}$ to $\Ga$.  

	\smallskip
	\textit{Step~II$_1$:  if $x$ and $y$ are in $\Ga$ and $d(x,y) \leq \te/3$, then $d_\Ga(x,y) \leq (\te/3)+1$.}

	The conclusion is trivial in the case where $x=y$.  Thus, we may assume that $x\neq y$.

	Observe that if $h(x) = h(y)$ and $d(x,y) \leq \te/3$, then $x\sim y$ (by \rmii\ in Step~I$_2$), so that $d_\Ga(x,y) = 1$, and the conclusion follows. 

	If, instead, $h(x) \neq h(y)$, then, without loss of generality, we may assume that $h(x) > h(y)$.  Set $\ell:= h(x) - h(y)$.  
	The triangle inequality implies that 
	$$
	d(x,y) \geq d(x,o) - d(o,y) = \ell.  
	$$
	Therefore $\ell\leq \te/3$.  Observe that
	$$
	d\big(p^{\ell}(x), x\big) 
	\leq \sum_{j=1}^{\ell} \, d\big(p^{j}(x), p^{j-1}(x)\big)  
	\leq 2\ell;
	$$
	the last inequality follows from \eqref{f: d(x,p(x))}. 
	The points $p^{\ell}(x)$ and $y$ belong to the same level, and the triangle inequality implies that
	$$
	d\big(p^{\ell}(x), y\big) 
	\leq  d\big(p^{\ell}(x), x\big) + d(x,y)
	\leq 2\ell + \te/3
	\leq \te.  
	$$
	Therefore $p^{\ell}(x)$ and $y$ are neighbours (see the definition of $\Ga$ in Step~I$_2$), and $[x,\ldots,p^{\ell}(x),y]$ is a path in $\Ga$ 
	of length at most $(\te/3)+1$, as required. 
	
	\smallskip
	\textit{Step~II$_2$:   the identity map is a $(\la,\be)$-rough isometry between $(\wt\Ga,d_{\wt\Ga})$ and $(X,d)$, where} 
	$$
	\boxed{\la :=\te\, (\te/3+1) 
	\quad\hbox{and}\quad
	\be:= \te\, (\te/3+5).
	}
	$$
	Observe that $d(x,\wt\Ga) \leq d(x,\Ga)$ for every $x$ in $X$.  By \eqref{f: covering},
	$$
		\sup\, \{d(x,\Ga): x \in X\}
		< 2,
	$$
	so that $\sup\, \{d(x,\wt\Ga): x \in X\} <2$.  Notice that 
	$$
	d(x,y) \le d_{\wt \Ga}(x,y) \qquad \forall x,y \in \wt \Ga.
	$$ 
	Proposition \ref{p: Ga Ga_0 wtGa}~\rmi\ implies that 
	$$
	d_{\wt \Ga}(x,y) 
	\le  \te \, d_{\Ga_0}(x,y) \quant x,y \in \wt \Ga.
	$$
	Hence, it suffices to prove that 
	\begin{equation} \label{f: Ga0 Ga}	
	\begin{aligned}
		d_{\Ga_0}(x,y) 
		\le (\te/3+1) \, d(x,y)+\te/3+5 \quant x,y \in \wt \Ga.
	\end{aligned}
	\end{equation}
	Suppose that $x$ and $y$ are distinct points in $\wt\Ga$, and 
	denote by $\ga$ a geodesic in $X$ joining them.  If $d(x,y)\le 1$ we set $a_0:=x$ and $a_1:=y$.  If $d(x,y) >1$, then we
	consider the points $\{a_0:=x, a_1,\ldots, a_{N-1}, a_N:=y\}$ in $\ga$ such that $d(a_j,a_{j+1})=1$ for every $j$ in $\{0,...,N-2\}$.  Clearly $d(x,y) \ge N-1$. 
	Next, for every $j$ in $\{0,...,N\}$ choose a point $z_j$ in $\Ga$ at minimum distance in $X$ from $a_j$.  In particular, $d(z_j,a_j) \le 2$. 
	The triangle inequality implies that 
	$$
	d(z_j,z_{j+1}) 
	\le d(z_j,a_j)+d(a_j,a_{j+1}) +d(a_{j+1},z_{j+1}) 
	\le 5
	$$
	for every $j$ in $\{0,...,N-1\}$.  Since $5\leq \te/3$ (see \eqref{f: te0}), Step~II$_1$ yields
	$$
	d_{\Ga_0}(z_j,z_{j+1}) 
	\le \te/3+1.  
	$$
	 By combining the previous estimates, we see that 
	$$
    	d_{\Ga_0}(x,y) 
	\le d_{\Ga_0}(x,z_0)+\sum_{j=0}^{N-1}d_{\Ga_0}(z_j,z_{j+1})+d(y,z_N) 
    	\le 4+N(\te/3+1).
	$$
	The right hand side above may be re-written as
	$$
     	(N-1)(\te/3+1)+\te/3+5  
     	\le d(x,y)(\te/3+1)+\te/3+5,
	$$
	thereby proving \eqref{f: Ga0 Ga}. 

	\smallskip
	\textbf{Step~III: $(\Ga_0,d_{\Ga_0})$ and $(X,d)$ are strictly roughly isometric.} 
	We split up Step~III into Steps~III$_1$-III$_4$.  Step~III$_1$ contains some preliminary estimates.  In Step~III$_2$ we prove that
	$d_{\Ga_0} \geq d-\eta$ for a suitable constant $\eta$.  Step~III$_3$ is devoted to the proof that geodesics 
	in $\Ga_0$ are quasi-geodesics in $X$.  Finally the upper estimate $d_{\Ga_0} \leq d+\eta$ is proved in Step~III$_4$.
 	
	Observe that 
	any geodesic in $\Ga_0$ connecting $o$ to a point $z$ in $\Ga_0$ coincides with the geodesic 
	in the metric subtree of $\Ga_0$ obtained from $\Ga_0$ by removing the horizontal edges introduced in Step~I$_2$.

	\smallskip
	\textit{Step~III$_1$: preliminary estimates.  Suppose that $z$ and $z'$ are points in~$\Ga_0$,  and that $z'$ lies on a geodesic 
	$\ga_{o,z}^0$ in $\Ga_0$  joining $o$ and $z$.  The following hold:
	\begin{enumerate}
		\item[\itemno1]
			if $z$ belongs to $\mrT$, then $\ga_{o,z}^0$ is a $(2,4)$-quasi-geodesic in $(X,d)$;
		\item[\itemno2]
			if $z$ and $z'$ belong to $\mrT$, then
			\begin{equation} \label{f: estimate I}
			d_\mrT(z,z') 
			\geq d(z,z') - 4 D_{2,4},  
			\end{equation}
			where $D_{2,4}$ is as in Step~I$_2$, and $d_\mrT$ denotes the tree distance on~$\mrT$;  
		\item[\itemno3]
			$d_{\mrT} (\floor z,\floor{z'}) \leq d_{\Ga_0} (z,z') + 2$;
			see Notation~\ref{nota: rfloor} for the definition of $\floor z$ and $\floor{z'}$;
		\item[\itemno4]
			$d_{\Ga_0} (z,z') \geq d(z,z') - 2 \te-4D_{2,4}-2$.
	\end{enumerate}
	}

	First we prove \rmi.  It suffices to prove \rmi\ in the case where $z\neq o$.  Denote~by  
	$$
	[a_0:=o, a_1, \ldots, a_{N-1}, a_N:=z] 
	$$
	the points in $\mrT$ that lie on $\ga_{o,z}^0$, and by $\si_j$ the geodesic segment in $\ga_{o,z}^0$ joining~$a_j$ and $a_{j+1}$.  
	Recall that $\si_j$ is also a geodesic in $(X,d)$.

	Suppose that $p$ and $q$ are distinct points in~$\ga_{o,z}^0$.  Without loss of generality we assume that $d(o,p)<d(o,q)$. 
	Then there are nonnegative integers $j_1$ and $j_2$, with $j_1\leq j_2\leq N-1$, such that $p$ and $q$ belong to $\si_{j_1}$ and to $\si_{j_2}$, respectively.

	If $j_1=j_2$, then $\ell_X\big(\ga_{o,z}^0([p,q])\big)=d(p,q)$; here $\ga_{o,z}^0([p,q])$ 
	denotes the segment in $\si_{j_1}$ with endpoints $p$ and $q$.

	If, instead, $j_1<j_2$, then 
	\begin{equation}
	\label{f: lgtpq} 
	\begin{aligned}
		\ell_X\big(\ga_{o,z}^0([p,q])\big)
		& =    d(p,a_{j_1+1})+\sum_{j=j_{1}+1}^{j_2-1}d(a_{j},a_{j+1})+ d(a_{j_2},q)\\
  		& \leq 2\, (j_2-j_1-1)+4;
	\end{aligned}
	\end{equation}
	we have used \eqref{f: d(x,p(x))} in the last inequality, and we agree that the second summand on the first line above vanishes if $j_2\leq j_1+1$. 
	The triangle inequality implies that 
	$$
	d(p,q) 
	\ge d(q,o)-d(p,o) 
	\ge d(a_{j_2},o)-d(a_{j_1+1},o)
	=j_2-j_1-1,
	$$ 
	which, combined with \eqref{f: lgtpq}, yields 
	\begin{align*}
		\ell_X\big(\ga_{o,z}^0([p,q])\big) 
		\le 2\, d(p,q)+4,
	\end{align*}
	as desired.

	Next we prove \rmii.  
	Observe that if $z'=o$ and $z$ belongs to $\Si_n$, then the construction of $\mrT$ (see Step~I) and the definition of~$\Si_n$ imply that 
	$d_\mrT(z,o) = n = d(z,o)$, so that \eqref{f: estimate I} trivially holds.  

	If, instead, $z'$ and $o$ are distinct points, consider the geodesic~$\ga$ in $X$ joining $o$ and $z$.  Since, by \rmi, $\ga_{o,z}^0$ is a 
	$(2,4)$-quasi-geodesic joining $o$ and $z$, Lemma~\ref{lem: Morse} ensures that there exists a positive constant $D_{2,4}$ such that $\ga_{o,z}^0$ lies
	in a $D_{2,4}$-neighbourhood of $\ga$.  Therefore there exists a point $x$ in $\ga$ such that $d(z',x) \leq D_{2,4}$, so that 
	$$
	\begin{aligned}
		\bigmod{d(z,z')-d_\mrT(z,z')}
		& = \bigmod{d(z,z')+d(z',o)-d(z,o)} \\
		& = \bigmod{d(z',o)-d(o,x)+d(z',z)-d(x,z)}\\
		& \leq 2D_{2,4};
	\end{aligned}
	$$
	the first equality follows from the fact that, by construction of the tree~$\mrT$, $d_\mrT(z,z') = d(o,z)-d(o,z')$,
	the second from the formula $d(o,z) = d(o,x) + d(z,x)$, which holds because $z$ lies on $\ga$, and the inequality 
	from the triangle inequality applied to both the triangles $oxz'$ and $z'xz$.  

	This proves \rmii.

	Observe that \rmiii\ is trivial in the case where $z$ and $z'$ are vertices of $\Ga$, for then~$z'$ lies on the geodesic in $\mrT$ joining $o$ and $z$,
	whence $d_{\Ga_0} (z,z') = d_{\mrT} (z,z')$. 

    	Notice that $d_{\Ga_0} (\floor z,\floor {z'}) = d_{\mrT} (\floor z,\floor {z'})$, 
	because $\lfloor z \rfloor$ and $\lfloor z' \rfloor$ are vertices of $\Ga$ lying on the geodesic in $\mrT$ joining $\floor z$ and $o$. 
	Furthermore $d_{\Ga_0}(z,\floor z)\leq 1$ and $d_{\Ga_0}(\floor {z'},z')\leq 1$, because $z$ lies on the edge in $\Ga_0$ joining $\floor z$ and
	one of its neighbours in $\Ga$, and similarly for $z'$.  Therefore
	$$
	d_{\mrT} (\floor z,\floor {z'})
	= d_{\Ga_0}(\floor z,\floor {z'}) 
	\le d_{\Ga_0}(\floor z,z)+d_{\Ga_0}(z,z')+d_{\Ga_0}(z',\floor {z'}),
	$$
	as required.

	Finally we prove \rmiv.   By combining \rmiii\ and \rmii, we see that 
	\begin{equation} \label{f: partial x xprimed}
		d_{\Ga_0} (z,z') 
		\geq d_{\mrT} (\floor z,\lfloor z'\rfloor) -2 
		\geq d (\floor z,\lfloor z'\rfloor) -4D_{2,4}-2.   
	\end{equation}
	Now, the point $z$ lies on a geodesic in $X$ joining $\lfloor z \rfloor$ and one of its neighbours.   
	Such geodesics in $X$ have length at most $\te$, so that 
	$d(z,\floor z)\leq \te$.  Similarly $d(\floor{z'},z')\leq \te $.  
	These estimates and the triangle inequality imply that
	$$ 
	d(\floor z,\floor z')
	\geq  d(z,z') - d(z,\floor z) - d(\floor{z'}, z')
	\geq d(z,z') - 2 \te, 
	$$
	Now, \rmiv\ follows directly from this and \eqref{f: partial x xprimed}.  
	
	\smallskip
	\textit{Step~III$_2$.  The following lower estimate holds: 
	\begin{equation} \label{f: lower}
		d_{\Ga_0} (x,y) \geq d(x,y) - \eta
		\quant x,y \in \Ga_0,
	\end{equation}
	where 
	$$
	\boxed{
	\eta 
	:= 2 \te \delta' + 2\de' + 4 \te + 2 + \max\big(8\, D_{2,4}  + 2,  6D_{\te\la,\te\be}\big).
	}
	$$  
	}  

	Consider the geodesic triangle  $oxy$ in $\Ga_0$ with sides $\ga_{o,x}^0$, $\ga_{o,y}^0$ and $\ga_{x,y}^0$.
	Since, by Step~II$_2$, $\Ga_0$ is a $\de'$-hyperbolic length space, Remark~\ref{rem: three points} implies the existence of a point $p$ in 
	in $(\Ga_0,d_{\Ga_0})$, and points $x'$ in $\ga_{o,x}^0$ and $y'$ in $\ga_{o,y}^0$ such that 
	\begin{equation} \label{f: p delta primed}
	\max\big(d_{\Ga_0}(p,x'), d_{\Ga_0}(p,y')\big) \leq \de'.  
	\end{equation}
	Since $p$ is a point on the geodesic $\ga_{x,y}^0$, 
	\begin{equation} \label{f: p middle}
	d_{\Ga_0} (x,y)
	= d_{\Ga_0} (x,p) + d_{\Ga_0} (p,y).  
	\end{equation}
	Now the triangle inequality, applied to the triangle $xx'p$, and \eqref{f: p delta primed} imply that  
	\begin{equation} \label{f: intermediate last}
	d_{\Ga_0} (x,p)
	\geq d_{\Ga_0} (x,x') - d_{\Ga_0} (x',p)
	\geq d_{\Ga_0} (x,x') - \de'.  
	\end{equation}
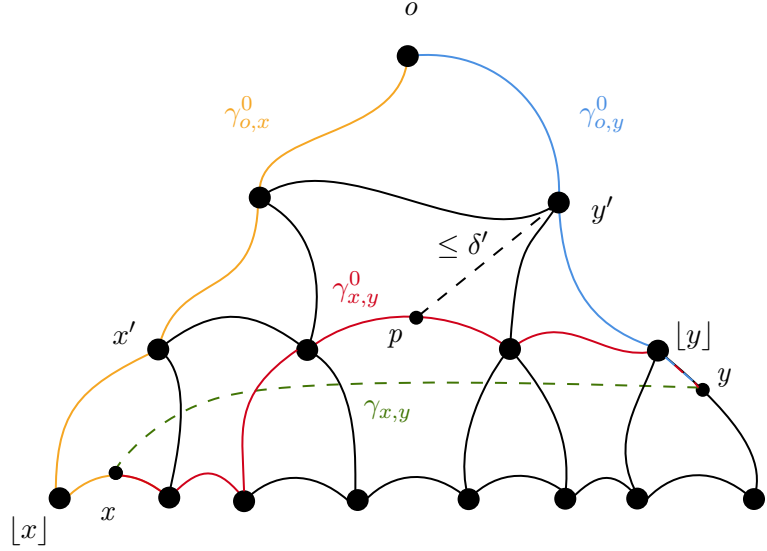
\begin{figure}\label{fig:2} 
\tikzset{every picture/.style={line width=0.75pt}} %set default line width to 0.75pt        

\begin{tikzpicture}[x=0.75pt,y=0.75pt,yscale=-1,xscale=1]
%uncomment if require: \path (0,503); %set diagram left start at 0, and has height of 503

%Shape: Circle [id:dp06645370913815196] 
\draw  [fill={rgb, 255:red, 0; green, 0; blue, 0 }  ,fill opacity=1 ] (164.33,326) .. controls (164.33,324.34) and (165.68,323) .. (167.33,323) .. controls (168.99,323) and (170.33,324.34) .. (170.33,326) .. controls (170.33,327.66) and (168.99,329) .. (167.33,329) .. controls (165.68,329) and (164.33,327.66) .. (164.33,326) -- cycle ;
%Shape: Circle [id:dp08949918642980803] 
\draw  [fill={rgb, 255:red, 0; green, 0; blue, 0 }  ,fill opacity=1 ] (386.6,339) .. controls (386.6,336.24) and (388.84,334) .. (391.6,334) .. controls (394.36,334) and (396.6,336.24) .. (396.6,339) .. controls (396.6,341.76) and (394.36,344) .. (391.6,344) .. controls (388.84,344) and (386.6,341.76) .. (386.6,339) -- cycle ;
%Shape: Circle [id:dp21521106188294015] 
\draw  [fill={rgb, 255:red, 0; green, 0; blue, 0 }  ,fill opacity=1 ] (338.37,339.23) .. controls (338.37,336.47) and (340.61,334.23) .. (343.37,334.23) .. controls (346.14,334.23) and (348.37,336.47) .. (348.37,339.23) .. controls (348.37,341.99) and (346.14,344.23) .. (343.37,344.23) .. controls (340.61,344.23) and (338.37,341.99) .. (338.37,339.23) -- cycle ;
%Shape: Circle [id:dp7050183382284283] 
\draw  [fill={rgb, 255:red, 0; green, 0; blue, 0 }  ,fill opacity=1 ] (480.75,339.17) .. controls (480.75,336.41) and (482.99,334.17) .. (485.75,334.17) .. controls (488.51,334.17) and (490.75,336.41) .. (490.75,339.17) .. controls (490.75,341.93) and (488.51,344.17) .. (485.75,344.17) .. controls (482.99,344.17) and (480.75,341.93) .. (480.75,339.17) -- cycle ;
%Shape: Circle [id:dp7595648330764789] 
\draw  [fill={rgb, 255:red, 0; green, 0; blue, 0 }  ,fill opacity=1 ] (422.5,338.92) .. controls (422.5,336.16) and (424.74,333.92) .. (427.5,333.92) .. controls (430.26,333.92) and (432.5,336.16) .. (432.5,338.92) .. controls (432.5,341.68) and (430.26,343.92) .. (427.5,343.92) .. controls (424.74,343.92) and (422.5,341.68) .. (422.5,338.92) -- cycle ;
%Shape: Circle [id:dp8307961695358936] 
\draw  [fill={rgb, 255:red, 0; green, 0; blue, 0 }  ,fill opacity=1 ] (283,339.67) .. controls (283,336.91) and (285.24,334.67) .. (288,334.67) .. controls (290.76,334.67) and (293,336.91) .. (293,339.67) .. controls (293,342.43) and (290.76,344.67) .. (288,344.67) .. controls (285.24,344.67) and (283,342.43) .. (283,339.67) -- cycle ;
%Shape: Circle [id:dp23480407114883373] 
\draw  [fill={rgb, 255:red, 0; green, 0; blue, 0 }  ,fill opacity=1 ] (457.18,284.23) .. controls (457.18,282.58) and (458.53,281.23) .. (460.18,281.23) .. controls (461.84,281.23) and (463.18,282.58) .. (463.18,284.23) .. controls (463.18,285.89) and (461.84,287.23) .. (460.18,287.23) .. controls (458.53,287.23) and (457.18,285.89) .. (457.18,284.23) -- cycle ;
%Straight Lines [id:da16758838671556853] 
\draw  [dash pattern={on 4.5pt off 4.5pt}]  (388.33,190) -- (317.23,247.95) ;
%Curve Lines [id:da6014852044451157] 
\draw [color={rgb, 255:red, 65; green, 117; blue, 5 }  ,draw opacity=1 ] [dash pattern={on 4.5pt off 4.5pt}]  (167.33,326) .. controls (167.48,325.11) and (167.64,324.25) .. (167.8,323.41) .. controls (167.97,322.57) and (185.85,296.94) .. (219.96,287.86) .. controls (254.08,278.78) and (345.75,279.42) .. (458.08,283.08) ;
%Curve Lines [id:da3788853991933314] 
\draw [color={rgb, 255:red, 245; green, 166; blue, 35 }  ,draw opacity=1 ]   (239.17,187.5) .. controls (235.33,152) and (310.33,160) .. (313,116.33) ;
%Curve Lines [id:da28349551232247505] 
\draw [color={rgb, 255:red, 245; green, 166; blue, 35 }  ,draw opacity=1 ]   (188.5,263.92) .. controls (198.33,222) and (238.33,242) .. (238.33,188) ;
%Curve Lines [id:da7678168933482334] 
\draw [color={rgb, 255:red, 245; green, 166; blue, 35 }  ,draw opacity=1 ]   (137.33,339) .. controls (137.33,296) and (146.33,284) .. (188.5,263.92) ;
%Curve Lines [id:da2318573389786529] 
\draw [color={rgb, 255:red, 245; green, 166; blue, 35 }  ,draw opacity=1 ]   (142.33,339) .. controls (149.03,332.05) and (156.81,328.05) .. (165.04,327.36) .. controls (165.8,327.29) and (166.56,327.26) .. (167.33,327.25) .. controls (176.04,327.17) and (185.18,330.75) .. (194,338.42) ;
%Curve Lines [id:da3770894321249769] 
\draw    (231.33,340.25) .. controls (245.67,325.36) and (265,324.03) .. (283,339.67) ;
%Curve Lines [id:da673354421667241] 
\draw    (291.71,339.81) .. controls (306.04,324.93) and (325.37,323.59) .. (343.37,339.23) ;
%Curve Lines [id:da7798658589068042] 
\draw    (343.37,339.23) .. controls (357.71,324.34) and (377.04,323.01) .. (395.04,338.65) ;
%Curve Lines [id:da25640084485404324] 
\draw    (432.5,338.92) .. controls (446.83,324.03) and (466.17,322.69) .. (484.17,338.33) ;
%Curve Lines [id:da4432610853891793] 
\draw [color={rgb, 255:red, 208; green, 2; blue, 27 }  ,draw opacity=1 ]   (194,338.42) .. controls (208.33,323.53) and (216.67,321.44) .. (229.83,337.17) ;
%Curve Lines [id:da8182867664758497] 
\draw    (392.92,340.67) .. controls (407.25,325.78) and (415.58,323.69) .. (428.75,339.42) ;
%Curve Lines [id:da5113829018507541] 
\draw [color={rgb, 255:red, 74; green, 144; blue, 226 }  ,draw opacity=1 ]   (313,116.33) .. controls (358,110.78) and (391.33,144.78) .. (388.33,190) ;
%Curve Lines [id:da038324915175611984] 
\draw [color={rgb, 255:red, 74; green, 144; blue, 226 }  ,draw opacity=1 ]   (388.33,190) .. controls (390.33,236.44) and (408,254.33) .. (439.5,264.08) ;
%Curve Lines [id:da9812031453463135] 
\draw    (437.83,264.5) .. controls (443.68,269.3) and (448.88,273.79) .. (453.48,278.01) .. controls (454.1,278.58) and (454.72,279.14) .. (455.32,279.7) .. controls (488.42,310.54) and (489.44,326.3) .. (485.75,339.17) ;
%Curve Lines [id:da25011929021823] 
\draw    (239.17,187.5) .. controls (279.17,157.5) and (348.33,220) .. (388.33,190) ;
%Curve Lines [id:da02947056689046723] 
\draw    (188.5,263.92) .. controls (218,240.33) and (236,243.67) .. (262.83,264.17) ;
%Curve Lines [id:da38668607070515615] 
\draw [color={rgb, 255:red, 208; green, 2; blue, 27 }  ,draw opacity=1 ]   (262.83,264.17) .. controls (292.33,240.58) and (337.17,243.17) .. (364,263.67) ;
%Curve Lines [id:da7386794302440163] 
\draw [color={rgb, 255:red, 208; green, 2; blue, 27 }  ,draw opacity=1 ]   (363.5,264.25) .. controls (393,240.67) and (409.33,271.67) .. (437.83,264.5) ;
%Shape: Circle [id:dp5683292303923805] 
\draw  [fill={rgb, 255:red, 0; green, 0; blue, 0 }  ,fill opacity=1 ] (359,263.67) .. controls (359,260.91) and (361.24,258.67) .. (364,258.67) .. controls (366.76,258.67) and (369,260.91) .. (369,263.67) .. controls (369,266.43) and (366.76,268.67) .. (364,268.67) .. controls (361.24,268.67) and (359,266.43) .. (359,263.67) -- cycle ;
%Shape: Circle [id:dp1583110341188877] 
\draw  [fill={rgb, 255:red, 0; green, 0; blue, 0 }  ,fill opacity=1 ] (257.83,264.17) .. controls (257.83,261.41) and (260.07,259.17) .. (262.83,259.17) .. controls (265.59,259.17) and (267.83,261.41) .. (267.83,264.17) .. controls (267.83,266.93) and (265.59,269.17) .. (262.83,269.17) .. controls (260.07,269.17) and (257.83,266.93) .. (257.83,264.17) -- cycle ;
%Shape: Circle [id:dp4989818192301376] 
\draw  [fill={rgb, 255:red, 0; green, 0; blue, 0 }  ,fill opacity=1 ] (432.83,264.5) .. controls (432.83,261.74) and (435.07,259.5) .. (437.83,259.5) .. controls (440.59,259.5) and (442.83,261.74) .. (442.83,264.5) .. controls (442.83,267.26) and (440.59,269.5) .. (437.83,269.5) .. controls (435.07,269.5) and (432.83,267.26) .. (432.83,264.5) -- cycle ;
%Shape: Circle [id:dp030196449488209987] 
\draw  [fill={rgb, 255:red, 0; green, 0; blue, 0 }  ,fill opacity=1 ] (314.23,247.95) .. controls (314.23,246.29) and (315.58,244.95) .. (317.23,244.95) .. controls (318.89,244.95) and (320.23,246.29) .. (320.23,247.95) .. controls (320.23,249.6) and (318.89,250.95) .. (317.23,250.95) .. controls (315.58,250.95) and (314.23,249.6) .. (314.23,247.95) -- cycle ;
%Shape: Circle [id:dp4117761682084272] 
\draw  [fill={rgb, 255:red, 0; green, 0; blue, 0 }  ,fill opacity=1 ] (226.33,340.25) .. controls (226.33,337.49) and (228.57,335.25) .. (231.33,335.25) .. controls (234.09,335.25) and (236.33,337.49) .. (236.33,340.25) .. controls (236.33,343.01) and (234.09,345.25) .. (231.33,345.25) .. controls (228.57,345.25) and (226.33,343.01) .. (226.33,340.25) -- cycle ;
%Curve Lines [id:da24059431023703404] 
\draw [color={rgb, 255:red, 208; green, 2; blue, 27 }  ,draw opacity=1 ]   (167.33,327.25) .. controls (168.05,327.26) and (168.76,327.28) .. (169.45,327.32) .. controls (178.58,327.85) and (185.25,331.37) .. (194,338.42) ;
%Shape: Circle [id:dp13322071124879253] 
\draw  [fill={rgb, 255:red, 0; green, 0; blue, 0 }  ,fill opacity=1 ] (189,338.42) .. controls (189,335.66) and (191.24,333.42) .. (194,333.42) .. controls (196.76,333.42) and (199,335.66) .. (199,338.42) .. controls (199,341.18) and (196.76,343.42) .. (194,343.42) .. controls (191.24,343.42) and (189,341.18) .. (189,338.42) -- cycle ;
%Straight Lines [id:da4150518330889217] 
\draw [color={rgb, 255:red, 74; green, 144; blue, 226 }  ,draw opacity=1 ] [dash pattern={on 4.5pt off 4.5pt}]  (441.58,267.83) -- (457.08,281.33) ;
%Straight Lines [id:da4427999423247704] 
\draw [color={rgb, 255:red, 208; green, 2; blue, 27 }  ,draw opacity=1 ]   (446.25,271.83) -- (450.33,275.58) ;
%Straight Lines [id:da9107896044590867] 
\draw [color={rgb, 255:red, 208; green, 2; blue, 27 }  ,draw opacity=1 ]   (456.1,280.48) -- (458.08,283.08) ;
%Shape: Circle [id:dp7746707131754059] 
\draw  [fill={rgb, 255:red, 0; green, 0; blue, 0 }  ,fill opacity=1 ] (383.33,190) .. controls (383.33,187.24) and (385.57,185) .. (388.33,185) .. controls (391.09,185) and (393.33,187.24) .. (393.33,190) .. controls (393.33,192.76) and (391.09,195) .. (388.33,195) .. controls (385.57,195) and (383.33,192.76) .. (383.33,190) -- cycle ;
%Shape: Circle [id:dp662447386168189] 
\draw  [fill={rgb, 255:red, 0; green, 0; blue, 0 }  ,fill opacity=1 ] (164.33,326) .. controls (164.33,324.34) and (165.68,323) .. (167.33,323) .. controls (168.99,323) and (170.33,324.34) .. (170.33,326) .. controls (170.33,327.66) and (168.99,329) .. (167.33,329) .. controls (165.68,329) and (164.33,327.66) .. (164.33,326) -- cycle ;
%Shape: Circle [id:dp8253477435668192] 
\draw  [fill={rgb, 255:red, 0; green, 0; blue, 0 }  ,fill opacity=1 ] (234.17,187.5) .. controls (234.17,184.74) and (236.41,182.5) .. (239.17,182.5) .. controls (241.93,182.5) and (244.17,184.74) .. (244.17,187.5) .. controls (244.17,190.26) and (241.93,192.5) .. (239.17,192.5) .. controls (236.41,192.5) and (234.17,190.26) .. (234.17,187.5) -- cycle ;
%Shape: Circle [id:dp0722338345092387] 
\draw  [fill={rgb, 255:red, 0; green, 0; blue, 0 }  ,fill opacity=1 ] (308,116.33) .. controls (308,113.57) and (310.24,111.33) .. (313,111.33) .. controls (315.76,111.33) and (318,113.57) .. (318,116.33) .. controls (318,119.09) and (315.76,121.33) .. (313,121.33) .. controls (310.24,121.33) and (308,119.09) .. (308,116.33) -- cycle ;
%Shape: Circle [id:dp16330480184208762] 
\draw  [fill={rgb, 255:red, 0; green, 0; blue, 0 }  ,fill opacity=1 ] (183.5,263.92) .. controls (183.5,261.16) and (185.74,258.92) .. (188.5,258.92) .. controls (191.26,258.92) and (193.5,261.16) .. (193.5,263.92) .. controls (193.5,266.68) and (191.26,268.92) .. (188.5,268.92) .. controls (185.74,268.92) and (183.5,266.68) .. (183.5,263.92) -- cycle ;
%Shape: Circle [id:dp1761463810291336] 
\draw  [fill={rgb, 255:red, 0; green, 0; blue, 0 }  ,fill opacity=1 ] (134.33,338.75) .. controls (134.33,335.99) and (136.57,333.75) .. (139.33,333.75) .. controls (142.09,333.75) and (144.33,335.99) .. (144.33,338.75) .. controls (144.33,341.51) and (142.09,343.75) .. (139.33,343.75) .. controls (136.57,343.75) and (134.33,341.51) .. (134.33,338.75) -- cycle ;
%Curve Lines [id:da10953232123157408] 
\draw    (188.5,263.92) .. controls (211.67,279.33) and (191.67,340.33) .. (194,338.42) ;
%Curve Lines [id:da481294723042475] 
\draw [color={rgb, 255:red, 208; green, 2; blue, 27 }  ,draw opacity=1 ]   (257.83,264.17) .. controls (224.67,287.33) and (230.67,305.33) .. (231.33,335.25) ;
%Curve Lines [id:da7658425358210537] 
\draw    (262.83,264.17) .. controls (295.67,280.33) and (283.67,342.33) .. (288,339.67) ;
%Curve Lines [id:da7222734605244217] 
\draw    (364,263.67) .. controls (395.67,300.33) and (392.67,322.33) .. (391.6,339) ;
%Curve Lines [id:da3014119350820096] 
\draw    (364,263.67) .. controls (359.67,267.33) and (333.67,304.33) .. (343.37,339.23) ;
%Curve Lines [id:da4027115124650905] 
\draw    (437.83,264.5) .. controls (429.67,290.33) and (411.67,309.33) .. (428.75,339.42) ;
%Curve Lines [id:da8559514942263157] 
\draw    (238.33,188) .. controls (277.67,206.33) and (267.67,253.33) .. (262.83,264.17) ;
%Curve Lines [id:da8235185433090569] 
\draw    (388.33,190) .. controls (373.67,214.33) and (367.67,205.33) .. (364,263.67) ;

% Text Node
\draw (157.83,342.01) node [anchor=north west][inner sep=0.75pt]   [align=left] {\textit{x}};
% Text Node
\draw (403,185.17) node [anchor=north west][inner sep=0.75pt]   [align=left] {${y'}$};

\draw (112,346.33) node [anchor=north west][inner sep=0.75pt]   [align=left] {$\lfloor x \rfloor$};

% Text Node 
%\draw (185,346.33) node [anchor=north west][inner sep=0.75pt]   [align=left] {$\lceil x \rceil$};
% Text Node
\draw (308.8,87.97) node [anchor=north west][inner sep=0.75pt]   [align=left] {\textit{o}};
% Text Node
\draw (220.07,135.17) node [anchor=north west][inner sep=0.75pt]   [align=left] {\textcolor[rgb]{0.96,0.65,0.14}{$\gamma_{o,x}^0$}};
% Text Node
\draw (397.2,135.17) node [anchor=north west][inner sep=0.75pt]   [align=left] {\textcolor[rgb]{0.29,0.56,0.89}{$\gamma^0_{o,y}$}};
% Text Node
\draw (273.87,223.3) node [anchor=north west][inner sep=0.75pt]   [align=left] {\textcolor[rgb]{0.82,0.01,0.11}{ $\gamma^0_{x,y}$}};
% Text Node
\draw (164.13,248.2) node [anchor=north west][inner sep=0.75pt]   [align=left] {${x'}$};
% Text Node
\draw (301.2,252.4) node [anchor=north west][inner sep=0.75pt]   [align=left] {\textit{p}};
% Text Node
\draw (465.2,271.97) node [anchor=north west][inner sep=0.75pt]   [align=left] {\textit{y}};
% Text Node
\draw (326,203) node [anchor=north west][inner sep=0.75pt]   [align=left] {$\le \delta'$};
% Text Node
\draw (444,248) node [anchor=north west][inner sep=0.75pt]   [align=left] {$\lfloor y\rfloor$};
% Text Node 
%\draw (484,348) node [anchor=north west][inner sep=0.75pt]   [align=left] {$\lceil y\rceil$};
% Text Node
\draw (289.65,288.12) node [anchor=north west][inner sep=0.75pt]   [align=left] {\textcolor[rgb]{0.25,0.46,0.02}{$\gamma_{x,y}$}};

\end{tikzpicture}
\caption{The figure represents the points $o$, $x$, $x'$, $\lfloor x\rfloor$, $y, y', [y]$, and $p$, as well as the geodesics in $\Gamma_0$ connecting $o$ to $x$ (orange), $o$ to $y$ (blue), and $x$ to $y$ (red). The geodesic in $X$ connecting $x$ to $y$ is represented by a green dashed curve.}
\end{figure}
	\noindent
	By combining Step~III$_1$~\rmiv\ and \eqref{f: intermediate last} yields 
	$$
	d_{\Ga_0} (x,p)
	\geq d(x,x') - \de'-4D_{2,4}- 2 \te  -2.    	 
	$$
	By arguing similarly, we see that  
	$$
	d_{\Ga_0} (p,y)
	\geq d(y,y') - \de'-4D_{2,4}- 2 \te-2.  
	$$
	Using \eqref{f: p middle} and the last two inequalities, we see that 
	$$
	\begin{aligned}
	d_{\Ga_0} (x,y)
		& \geq d(x,x') + d(y,y')- 2\de'-  8\, D_{2,4}- 4 \te-4   \\
		& \geq d(x,y) - d(x',y') - 2\de'-8\, D_{2,4}- 4 \te-4  \\
		& \geq d(x,y) -2 \te \delta' -  2\de'-8\, D_{2,4}- 4 \te-4;
	\end{aligned}
	$$
	the penultimate inequality follows from the triangle inequality 
	$$
	d(x,y)
	\leq d(x,x') + d(x',y') + d(y',y), 
	$$
	and the last inequality from the fact that 
	$$
	d(x',y') \leq d_{\wt{\Ga}}(x',y') \leq \te \, d_{\Ga_0}(x',y')  \leq 2 \te \delta';
	$$  
	we have used \eqref{f: p delta primed} in the last inequality above.   

	The proof of \eqref{f: lower} is complete. 

	\medskip
	\textit{Step~III$_3$:  geodesics in $\Ga_0$ are $\big(\te\la, \te\be\big)$-quasi-geodesics in $X$, where $\la$ and~$\be$ are
	as in Step~II$_2$.} 

	Suppose that $x$ and $y$ are in $\Ga_0$, and denote by $\ga^0_{x,y}$ a geodesic in $\Ga_0$ joining~$x$ and $y$.  
	Observe that if $\ga^0_{x,y}$ is contained in an edge, then for any pair of points $u$ and $u'$ in $\ga^0_{x,y}$ 
	$$
	\ell_X(\ga^0_{x,y}([u,u']))
	= d_{\wt \Ga}(u,u')
	= d(u,u');
	$$ 
	here $\ga^0_{x,y}([u,u'])$ denotes the segment in $\ga^0_{x,y}$ joining $u$ to $u'$. 
	Otherwise, denote  by $u_1$ and $u_1'$ the vertices in $\ga^0_{x,y}$ at minimum distance in  $\Ga_0$ from~$u$ and $u'$, respectively.  Clearly
	\begin{equation} \label{f: length subgeod}
	\ell_X \big(\ga^0_{x,y}([u,u'])\big)
	= d_{\wt \Ga}(u,u_1) + \ell_X \big(\ga^0_{x,y}([u_1,u_1'])\big) + d_{\wt \Ga}(u'_1,u').
	\end{equation}
	Note that $\ell_X \big(\ga^0_{x,y}([u_1,u_1'])\big)$ is equal to the sum of the lengths in $X$ of the edges in $\ga^0_{x,y}([u_1,u_1'])$.  Since
	the length of each of these is dominated by $\te$,
	$$
	\ell_X \big(\ga^0_{x,y}([u_1,u_1'])\big) 
	\leq \te \, \#\big\{\hbox{edges in $\ga^0_{x,y}([u_1,u_1'])$}\big\}
	=    \te\, d_{\Ga_0}(u_1,u_1').   
	$$
	Furthermore $d_{\wt \Ga}(u,u_1) \le \te\,d_{\Ga_0}(u,u_1)$ and $d_{\wt \Ga}(u'_1,u') \leq \te\,d_{\Ga_0}(u_1',u')$ by the right inequality
	in \eqref{f: comparison distances}.  Now. these estimates and \eqref{f: length subgeod} imply that 
	$$
	\begin{aligned}
	\ell_X \big(\ga^0_{x,y}([u,u'])\big)
	 	& \le  \te\,\big[ d_{\Ga_0}(u,u_1)+ d_{\Ga_0}(u_1,u_1')+d_{\Ga_0}(u_1',u')\big] \\
		& =    \te\, d_{\Ga_0}(u,u').
	\end{aligned}
	$$
	The construction of $\Ga$ implies that the length in $X$ of any geodesic joining any two neighbours in $\Ga$ is at least $1$.  Therefore 
	$d_{\Ga_0}(u,u') \le d_{\wt \Ga}(u,u')$, whence 
	$$
	\ell_X \big(\ga^0_{x,y}([u,u'])\big)
	\leq \te\, d_{\wt\Ga}(u,u').
	$$

	Since, by Step II$_2$, the identity map is a $(\la,\be)$-rough isometry between $(\wt \Ga,d_{\wt \Ga})$ and $(X,d)$, we can conclude that
	$$
	\ell_X \big(\ga^0_{x,y}([u,u'])\big) 
	\le \te \la\, d(u,u') + \te\be,
	$$
	i.e. $\ga_{x,y}^0$ is a $(\te\la,\te\be)$-quasi-geodesic in~$X$, as desired.  

	\medskip
	\textit{Step~III$_4$:  the following upper estimate holds:} 
	\begin{equation} \label{f: upper}
		d_{\Ga_0} (x,y) \leq d(x,y) + \eta
		\quant x,y \in \Ga_0,
	\end{equation}
    	where $\eta$ is as in Step III$_2$.

	Preliminarily, observe that 
	\begin{equation} \label{f: d dwtGa and dGa0}
		d \leq d_{\wt \Ga} \leq \te \, d_{\Ga_0} 
	\end{equation}
	on $\Ga_0$.  The left inequality is trivial, and the right inequality follows from Proposition~\ref{p: Ga Ga_0 wtGa}~\rmi\ (with $\te$
	in place of $A_2$).

	Note that if $x=y$, then \eqref{f: upper} holds trivially.  If $y$ agrees with $o$, then 
	$$
	d_{\Ga_0}(x,o) 
	\leq d_{\Ga_0}(x,\floor x) + d_{\Ga_0}(\floor x,o) 
	\leq 1 + d_{\mrT}(\floor x,o) 
	=    1 + d(\floor x,o). 
	$$
	Now, 
	$$
	d(\floor x,o) 
	\leq d(\floor x,x) + d(x,o) 
	\leq \te + d(x,o).
	$$
	Thus, 
	$$
	d_{\Ga_0}(x,o) 
	\leq 1+\te + d(x,o), 
	$$
	and \eqref{f: upper} holds, because $1+\te < \eta$.

	Thus, we can assume that $x$ and $y$ are distinct points both different from~$o$.
	Denote by $\ga_{x,y}$, $\ga_{o,x}$ and $\ga_{o,y}$ the edges of the geodesic triangle $oxy$ in $X$, and 		
	by $\ga_{x,y}^0$, $\ga_{o,x}^0$ and $\ga_{o,y}^0$ the edges of the corresponding geodesic triangle $oxy$ in $\Ga_0$.   

	Since, by Step~II$_2$, the metric graph $\Ga_0$ is a $\de'$-hyperbolic length space, Remark~\ref{rem: three points} guarantees the existence of points $p$, $v$ 
	and $w$ in $\ga_{x,y}^0$, $\ga_{o,x}^0$ and $\ga_{o,y}^0$, respectively, such that 
	\begin{equation} \label{f: comparison distances}
	d_{\Ga_0}(p,v) < \de'\quad\hbox{and}\quad d_{\Ga_0}(p,w) < \de'.
	\end{equation}

	Step~III$_3$ and Morse Lemma ensure that $\ga_{o,x}^0$, $\ga_{x,y}^0$ and $\ga_{o,y}^0$ are contained in $D_{\te\la,\te\be}$-neighbourhoods
	of~$\ga_{o,x}$, $\ga_{o,y}$ and $\ga_{x,y}$, respectively.  In particular, there exist points~$v'$ in $\ga_{o,x}$, $w'$ in $\ga_{o,y}$ and 
	$p'$ in $\ga_{x,y}$ such that 
	\begin{equation}  \label{f: Dte0}
		\max\big(d(v,v'), d(w,w'), d(p,p')\big) < D_{\te\la,\te\be}.
	\end{equation}
	As an intermediate step, we derive a few consequences from \eqref{f: d dwtGa and dGa0}, \eqref{f: comparison distances} and \eqref{f: Dte0}. 
	The triangle inequality and \eqref{f: Dte0} imply that 
	$$
		d(p',v')
		\leq d(p',p) + d(p,v) + d(v,v')
		<    d(p,v) + 2 D_{\te\la,\te\be}.
	$$
	Moreover $d(p,v) \leq \te\, d_{\Ga_0}(p,v)\leq \te\, \de'$ by \eqref{f: d dwtGa and dGa0} and 
	\eqref{f: comparison distances}.  By combining the last two estimates we see that
	\begin{equation} \label{f: p primed v primed}
		d(p',v') 
		\leq \te\, \de' + 2 D_{\te\la,\te\be}; 
	\end{equation}
	A similar argument yields 
	\begin{equation} \label{f: p primed w primed}
		d(p',w') 
		\leq \te\, \de' + 2 D_{\te\la,\te\be}. 
	\end{equation}

	Now,
	$$
	\begin{aligned}
		d_{\Ga_0}(x,y)
		& = d_{\Ga_0}(x,p) + d_{\Ga_0}(p,y),
	\end{aligned}
	$$
	because $p$ is a point on the geodesic in $\Ga_0$ joining $x$ and $y$.  We use the triangle inequality applied to the  triangle $xpv$ 
	in $\Ga_0$ and \eqref{f: comparison distances}, and obtain that $d_{\Ga_0}(x,p) \leq d_{\Ga_0}(x,v) + \de'$.   A similar argument applied 
	to the triangle $ypw$ in $\Ga_0$ yields $d_{\Ga_0}(p,y) \leq d_{\Ga_0}(w,y) + \de'$.  By combining these inequalities, we obtain that 
	\begin{equation} \label{f: x0 y0}
		d_{\Ga_0}(x,y)
		\leq  d_{\Ga_0}(x,v) + d_{\Ga_0}(w,y) + 2\de'.  
	\end{equation}
	Consider first $x$ and $v$.  Since $v$ belongs to the geodesic $\ga_{o,x}^0$,
	\begin{equation} \label{f: xvo}
		d_{\Ga_0}(x,v)
		=   d_{\Ga_0}(x,o)- d_{\Ga_0}(v,o).
	\end{equation}
	Now, $\floor v$ lies on the geodesic in $\Ga_0$ joining $o$ and $v$, and $\floor{x}$ lies on the geodesic in $\Ga_0$ joining $o$ and $x$. 
	Furthermore, $x$ belongs to a geodesic segment joining $\floor{x}$ and one of its neighbours.  Therefore 
	$$
	d_{\Ga_0}(v,o) \geq d_{\Ga_0}(\floor{v},o)
	\quad\hbox{and}\quad
	d_{\Ga_0}(\floor{x},o)\leq d_{\Ga_0}(x,o) \leq d_{\Ga_0}(\floor{x},o) +1.
	$$ 
	These inequalities and \eqref{f: xvo} imply that 
	$$
		d_{\Ga_0}(x,v)
			\leq   d_{\Ga_0}(\floor x,o)- d_{\Ga_0}(\floor v,o) +1
			=      d_{\Ga_0}(\floor x,\floor v)+1:
	$$
	in the equality above we have used the fact that $\floor v$ lies on the geodesic joining $o$ and $\floor x$.

	A similar inequality holds if we replace $x$ and $v$ with $y$ and $w$.  These observations and \eqref{f: x0 y0} imply that 
	$$
		d_{\Ga_0}(x,y)
		\leq  d_{\Ga_0}(\floor x,\floor v) + d_{\Ga_0}(\floor w,\floor y) + 2\de'+2.
	$$
	Now, observe that 
	$$
	d_{\Ga_0}(\floor x,\floor v)
	= d_{\Ga_0}(\floor x, o) -  d_{\Ga_0}(o,\floor v)
	= d(\floor x, o) -  d(o,\floor v);
	$$
	the second equality follows from the construction of the tree $\mrT$ (see Step~I).  Similarly, 
	$$
	d_{\Ga_0}(\floor y,\floor w)
	= d_{\Ga_0}(\floor y, o) -  d_{\Ga_0}(o,\floor w)
	= d(\floor y,o) -  d(o,\floor w).
	$$ 
	Now, the triangle inequality implies that
	$$
	d(\floor x,o)   
	\leq  d(\floor x, x) +  d(x,o) 
	\leq  d_{\wt\Ga}(\floor x, x) +  d(x,o) 
	\leq \te + d(x,o).
	$$
	Similarly 
	$$
	d(\floor y,o) \le  d(\floor y,y) + d(y,o) \le \te + d(y,o).
	$$
	For much the same reason we see that 
	$$
	d(o,\floor v) \geq d(o,v)-\te
	\quad\hbox{and}\quad
	d(o,\floor w) \geq d(o,w)-\te.
	$$ 
	Now, the triangle inequality applied to the triangles $ovv'$ and $oww'$ in $X$, together with the estimate \eqref{f: Dte0},  imply that 
	$$
	d(o,v) \geq d(o,v') - d(v,v') \geq d(o,v') - D_{\te\la,\te\be}
	$$
	and 
	$$
	d(o,w) \geq d(o,w') - d(w,w') \geq d(o,w') - D_{\te\la,\te\be}.  
	$$
	By combining the estimates above, we find that 
	\begin{equation} \label{f: penultimate}
		d_{\Ga_0}(x,y)
		\leq d(x,o)-d(v',o) + d(y,o) - d(w',o) + 4\te + 2D_{\te\la,\te\be} + 2 \de'+2.
	\end{equation}
	Observe that $d(x,o)-d(v',o) = d(x,v')$, because $v'$ belongs to $\ga_{o,x}$, and similarly
	$d(y,o) - d(w',o) = d(y,w')$.  The triangle inequality applied to the triangles $p'xv'$ and $p'yw'$ in $X$ and the fact that 
	$d \le d_{\wt \Ga} \le \te \, d_{\Ga_0}$ on $\Ga_0$ give
	\begin{align*}
		d(x,v') 
		&\leq d(x,p') + d(p',p)+d(p,v)+d(v,v') \\ 
		&\leq d(x,p') +2D_{\te\la,\te\be} +\te\,d_{\Ga_0}(p,v) \\ 
		&\leq d(x,p') +2D_{\te\la,\te\be} +\te\,\de',    	
	\end{align*}
	and similarly,
	$$
	d(y,w')  \leq d(y,p') +2D_{\te\la,\te\be} +\te \,\de'.  
	$$
	Now, \eqref{f: penultimate} and the fact that $d(x,y) = d(x,p') + d(p',y)$ (which holds, because $p'$ belongs to $\ga_{x,y}$) imply that
	$$
	\begin{aligned}
		d_{\Ga_0}(x,y)
		& \leq d(x,p') + d(p',y) + 2\te \de'+ 4\te +6D_{\te\la,\te\be} + 2 \de'+2 \\
		& =    d(x,y) + 2\te  \de' + 4\te +6D_{\te\la,\te\be} + 2 \de'+2,
	\end{aligned}
	$$
	as required.    

	\smallskip
	\textbf{Step~IV: construction of the spider's web $\wh\Ga$ and conclusion of the proof.} 
	We split up Step~IV into Step~IV$_1$, where we prove that $\Ga$ is a quasi-spider's web. and Step~IV$_2$, where we show that
	$\Ga$ can be completed to a spider's web $\wh\Ga$, which is still strictly roughly isometric to $X$.

	\textit{Step~IV$_1$: $(\Ga,d_\Ga)$ is a quasi-spider's web.  Specifically, if $x$ and $y$ are neighbours in $\Ga$, and $j > D_{2,4}+\de+ \te$, 
	then $p^j(x)$ and $p^j(y)$ either coincide or are neighbours in $\Ga$}.

	Suppose that $x$ and $y$ are neighbours in $\Ga$ belonging to $\Si_n$, and consider the edges $\ga_{x,y}$,~$\ga_{o,x}$ and $\ga_{o,y}$ 
	of the geodesic triangle $oxy$ in $(X,d)$.  Consider also the geodesics $\ga_{o,x}^0$ and $\ga_{o,y}^0$ 
	in $(\Ga_0,d_{\Ga_0})$, and for $j$ in $\{1,\ldots, n-1\}$, the predecessors $p^j(x)$ and $p^j(y)$ of $x$ and $y$, respectively.  

	By \eqref{f: d(x,p(x))}, $d\big(p^j(x),p^{j+1}(x)\big) \le 2$ for every $j$.  
	We have already proved in Step III$_1$~\rmi\ that $\ga_{o,x}^0$ and $\ga_{o,y}^0$ 
	are $(2,4)$-quasi-geodesics in $(X,d)$.  By Morse Lemma (see Lemma \ref{lem: Morse}), there exists a constant $D_{2,4}$, 
	depending on $\de$, such that~$\ga_{o,x}^0$ is contained in the $D_{2,4}$-neighbourhood of $\ga_{o,x}$ and $\ga_{o,y}^0$ is contained in 
	the~$D_{2,4}$-neighbourhood of~$\ga_{o,y}$.  In particular, there exist points $v$ in $\ga_{o,x}$ and $w$ in $\ga_{o,y}$ such that 
	\begin{equation} \label{f: DD}
		d\big(p^j(x),v\big) < D_{2,4}
	\quad\hbox{and}\quad
		d\big(w,p^{j}(y)\big) < D_{2,4}.
	\end{equation}
	Since $(X,d)$ is a $\de$-hyperbolic length space, the geodesic $\ga_{o,y}$ is contained in a $\de$-neighbourhood of $\ga_{o,x} \cup \ga_{x,y}$.
	In particular, there exists a point $z$ in $\ga_{o,x}\cup \ga_{x,y}$ such that $d(z,w) < \de$.

	We claim that if $j>D_{2,4}+\de+\te$, then $z$ belongs to $\ga_{o,x}$.  It suffices to show that $d(w,\ga_{x,y}) > \de$.  Indeed, if $p$ is in $\ga_{x,y}$, then
	$$
	d(w,p)
	\geq d(w,y) - d(y,p)
	\geq d\big(p^j(y),y\big) - d\big(w, p^j(y)\big) - d(y,p);
	$$
	the first inequality above follows from the triangle inequality applied to the triangle $wpy$ and the second from the triangle
	inequality applied to the triangle $p^j(y)wy$.  

	Now, $d(y,p) \leq \te$, because $y$ and $p$ belong to the geodesic $\ga_{x,y}$ and
	such geodesic has length at most $\te$ (for $x$ and $y$ are neighbours in $\Ga$ belonging to the same level $\Si_n$), $d\big(w, p^j(y)\big) < D_{2,4}$
	by \eqref{f: DD}, and $d\big(p^j(y),y\big) \geq j$.   Hence 
	$$
	d(w,p) 
	\geq j-\te-D_{2,4}
	> \de,
	$$
	as claimed.  

	Finally, we shall show that if $j> \te+D_{2,4}+\de$, then $d\big(p^j(x),p^j(y)\big) \leq \te$, whence $p^j(x)$ and $p^j(y)$ are connected in $\Ga$. 

	Observe that $p^j(x)$ and $p^j(y)$ belong to $\Si_{n-j}$, hence to $S_{n-j}(o)$.  Therefore the triangle inequality (applied to the triangle $op^j(y)w$ in $X$)
	implies that 
	$$
	n-j-D_{2,4} 
	\leq d(o,w)
	\leq n-j+D_{2,4}.  
	$$
	Furthermore, the triangle inequality (applied to the triangle $ozw$ in $X$) implies that 
	$$
	d(w,o) - \de
	\leq d(o,z) 
	\leq d(w,o) + \de,
	$$
	so that 
	\begin{equation} \label{f: DD I}
		n-j - D_{2,4} - \de
	\leq d(o,z) 
		\leq n-j +D_{2,4} + \de.  
	\end{equation}
	A similar reasoning gives that 
	\begin{equation} \label{f: DD II}
		n-j-D_{2,4} \leq d(v,o) \leq n-j+D_{2,4}.
	\end{equation}
	Now, since $v$ and $z$ belong to $\ga_{o,x}$, 
	$$
	d(v,z)
	= \bigmod{d(v,o)-d(z,o)}
	\leq 2D_{2,4}+\de;
	$$
	the inequality above follows by combining \eqref{f: DD I} and \eqref{f: DD II}.  Finally, notice that, by the triangle inequality,
	$$
	\begin{aligned}
	d\big(p^j(x),p^j(y)\big)
		& \leq d\big(p^j(x),v\big) +  d(v,z)+d(z,w) + d\big(w,p^j(y)\big) \\
		& \leq 4D_{2,4}+2\de \\
		& \leq \te;
	\end{aligned}
	$$
	the last inequality follows from the definition of $\te$ (see \eqref{f: te0}).
	Therefore the points $p^j(x)$ and $p^j(y)$ either coincide or are connected in~$\Ga$, as required.

	\smallskip
	\textit{Step~IV$_2$: the spider's web $(\wh\Ga,d_{\wh\Ga})$ is strictly roughly isometric to $(X,d)$}.
	The graph $\wh\Ga$ has the same set of vertices as $\Ga$.  We endow $\wh\Ga$ the structure of a spider's web by adding edges to the quasi-spider's web $\Ga$. 
	We declare that two distinct vertices $x$ and $y$ in $\Ga$ are neighbours in $\wh\Ga$ if and only if one of the following holds:
	\begin{enumerate}
		\item[\itemno1]
			$x$ and $y$ are neighbours in $\Ga$;
		\item[\itemno2]
			$x$ and $y$ are not neighbours in $\Ga$, they belong to the same level, %i.e. they have the same distance from the rooot $o$,
			and there exists a positive integer $j$ and vertices $v$ and $w$, which are neighbours in $\Ga$, such that $x=p^j(v)$ and $y=p^j(w)$.
	\end{enumerate}

	\noindent
	Thus, we are adding edges to $\Ga$ in order to get $\wh\Ga$, but, loosely speaking,  not too many, for $\Ga$ is already a quasi-spider's web (see 
	Step~IV$_1$).  Therefore there exists a positive integer $m$ such that if $v$ and $w$ are neighbours in~$\Ga$ and belong to the same level,~$\Si_n$ say, 
	then for all $j$ in $\{m,\ldots,n\}$ either the vertices~$p^j(v)$ and~$p^j(w)$ are neighbours in $\Ga$ or they coincide.
	
	We endow $\wh\Ga$ with the graph distance $d_{\wh\Ga}$ and set $K:=\lceil D_{2,4}+\de+\te+1 \rceil$.  We claim that
	\begin{equation} \label{f: ultima}
	d_{\wh\Ga}(x,y) \leq d_\Ga(x,y) \leq d_{\wh\Ga} (x,y)+ 2K.
	\quant x,y\in \Ga.
	\end{equation}
	The left hand inequality above is trivial.  We prove the right hand inequality.
	Suppose that $x$ and $y$ are in $\Ga$.  By Proposition~\ref{p: geodesics}~\rmi, there exists a standard geodesic $\wh \ga$ in $\wh \Ga$ connecting~$x$ to $y$.  
	We shall define a path $\ga$ in $\Ga$ such that 
	$$
	\ell_{\Ga} (\ga) \leq \ell_{\wh\Ga}(\wh\ga) + 2K,
	$$
	which clearly implies \eqref{f: ultima}.  
	
	Denote by $[x,\tilde x]$, $[\tilde x, \tilde y],$ and $[\tilde y, y]$ the ascending, the horizontal and the descending parts of $\wh \ga$, respectively.  
	   Observe that $\wt x$ and $\wt y$ belong to the same level, $\Si_n$ say, so that their distance 
	from $o$ both is equal to $n$ in $\Ga$ and in $\wh\Ga$.

	If $n \le K$, then define $\ga$ to be the union of the ascending geodesic $[x,o]$ and the descending 
	geodesic $[o,y]$.  

	If, instead, $n > K$, then define $\ga$ as the union of the ascending geodesic $[x,p^{K}(\wt x)]$, 
	the horizontal path $\ga_K$, defined as the projection on $\Sigma_{h(\wt x)-K}$ of $[\wt x, \wt y]$ and the descending geodesic $[p^{K}(\wt y),y]$. 

	Observe that, by definition of $\wh \Ga$ (see Step~IV$_1$), if two neighbours $u$ and $v$ in $\wh \Ga$ belong to the same level, then $p^K(u)$ and $p^K(v)$ 
	either coincide, or are neighbours in $\Ga$. 
	This implies that $\ga$ is a path in $\Ga$ connecting $x$ to $y$.  We conclude that 
	$$
	\begin{aligned}
		d_{\Ga}(x,y) \le \ell_\Ga(\ga) 
	  &\le \ell_{\wh\Ga}(\wh \ga)+2K 
		= d_{\wh \Ga}(x,y)+2K,
	\end{aligned}
	$$
	as claimed.

	By Step~III, $(\Ga_0,d_{\Ga_0})$ and $(X,d)$ are strictly roughly isometric.  It is straightforward to check that $(\Ga_0,d_{\Ga_0})$ and  
	$(\Ga,d_{\Ga})$ are strictly roughly isometric.  Furthermore \eqref{f: ultima} implies that $(\Ga,d_{\Ga})$ and $(\wh\Ga,d_{\wh\Ga})$ are 
	strictly roughly isometric. 
	Since the relation of being strictly roughly isometric is an equivalence relation, in particular it is transitive, 
	we can conclude that $(\wh\Ga,d_{\wh\Ga})$ and $(X,d)$ are strictly roughly isometric, as required to conclude the proof of Step~IV$_2$, and 
	of the theorem.  
\end{proof}

\begin{remark}
	\rm{
	Observe that the lower estimate in Step~III$_1$~\rmii\ holds for any pair of points $z$ and $z'$ in $\mrT$.  Indeed, denote by $z\wedge z'$ 
	the confluent of~$z$ and~$z'$ in $\mrT$ with respect to the root $o$, i.e. the last point in common between the geodesics in $\mrT$ joining $o$ to $z$ and $z'$.  Clearly 
	$$
	d_\mrT(z,z') 
	= d_\mrT(z,z\wedge z') +  d_\mrT(z',z\wedge z').  
	$$
	Now, the point $z\wedge z'$ belongs to both the tree geodesics joining $o$ with $z$ and~$o$ with $z'$.  Therefore, by case \rmii, applied twice, we see that 
	$$
	d_\mrT(z,z\wedge z') \geq d(z,z\wedge z') - 2\, D_{2,4} 
	\quad\hbox{and}\quad
	d_\mrT(z',z\wedge z') \geq   d(z',z\wedge z') - 2\, D_{2,4}.  
	$$
	By combining the formulae above, we see that 
	$$
	d_\mrT(z,z') \geq d(z,z') - 4\, D_{2,4},
	$$ 
	as required. 
	}
\end{remark}

\section[Proof of the main result]{The main result and its consequences} \label{s: main}
In this section we prove Theorem~\ref{t: main} and some of its consequences.  
We start with a well-known result concerning $\cM_0$ (see \eqref{f: local max}  for its definition).  For the sake of completeness, we include a simple proof thereof.

\begin{proposition} \label{p: weak cM0}
	Suppose that $(X,d,\mu)$ is a locally doubling metric measure space.  
	Then the operator $\cM_0$ is of weak  type $(1,1)$ and bounded on $\lp{X}$ for all $p$ in $(1,\infty]$.
\end{proposition}

\begin{proof}
	The operator $\cM_0$ is trivially bounded on $\ly{X}$.   We prove the weak type $(1,1)$ estimate: the boundedness of $\cM_0$ on $\lp{X}$ 
	for $p$ in $(1,\infty)$ will follow from this by interpolation.   

	Consider a $1$-discretisation $\fD$ of $X$. 
	Then every point in $X$ is covered by a ball $B_2(z)$ for some $z$ in $\fD$.  Set $\ds \psi := \sum_{z\in \fD} \, \One_{B_2(z)}$.  Observe that 
	$$
	1\leq \psi \leq L_{12,1/2};
	$$
	see, for instance, \cite[Lemma~1~\rmi]{MVo}.  The constant $L_{\tau,s}$ is defined in Remark~\ref{r: geom I}. 

	Suppose that $f$ is in $\lu{X}$.  For every $z$ in $\fD$, set $\ds f_z := f\, \frac{\One_{B_2(z)}}{\psi}$.  Observe that 
	$\ds f = \sum_{z\in \fD} \, f_z$.  By the sublinearity of $\cM_0$ we have that 
	\begin{equation} \label{f: sublin cM0}
		\cM_0f 
		\leq \sum_{z\in \fD} \, \cM_0 {f_z}  
	\end{equation}
	Observe that the support of $\cM_0 {f_z}$ is contained in $\OV{B_4(z)}$.  Thus, given a point~$x$, the function $\cM_0 {f_z}$
	(possibly) does not vanish at $x$ only if $d(x,z) < 4$, i.e. only if $z$ belongs to $B_4(x)$.  Now
	$$
		\sharp \big(\fD\cap B_4(x)\big) \leq \frac{L_{12,1/2}}{L_{16}}
	$$
	(see, for instance, \cite[Lemma~1~\rmii]{MVo}).  Denote by $N$ the constant on the right hand of the inequality above.
	Then for each point $x$ there are at most~$N$ summands of the series in \eqref{f: sublin cM0} (possibly) nonvanishing at~$x$.   
	Consequently, for each positive number $\la$ the following containment holds:
	\begin{equation} \label{f: level set cM0}
		\big\{\cM_0f > \la \big\}
		\subseteq \bigcup_{z\in \fD} \, \big\{\cM_0 {f_z} > \la/N \big\}.   
	\end{equation}
	We shall prove that there exists a constant $C$, independent of $z$, such that 
	\begin{equation} \label{f: level set summand}
		\bigmod{\big\{\cM_0 {f_z} > \si \big\}} 
		\leq \frac{L_5^3}{\si} \, \bignorm{f_z}{1}
		\quant \si >0.  
	\end{equation}
	This will imply that 	
	$$
	\bigmod{\big\{\cM_0 f > \la\big\}}  
	\leq \frac{L_5^3N}{\la} \, \sum_{z\in \fD} \, \bignorm{f_z}{1}
	=    \frac{L_5^3N}{\la} \bignorm{f}{1},
	$$
	as required.   

	Thus, it remains to prove \eqref{f: level set summand}.   Denote by $E_\si$ the level set $\big\{\cM_0 {f_z} > \si\big\}$.  If $x$ belongs to $E_\si$, 
	then there exists a ball $B_x$, centred at $x$ and of radius at most $1$, such that 
	$$
		\mod{B_x} \leq \frac{1}{\si} \, \int_{B_x} \, \mod{f_z} \wrt \mu.   
	$$
	The collection of balls $\cF := \{B_x: x \in E_\si\}$ covers $E_\si$.  Now, we can argue \textit{verbatim} as in the proof of 
	\cite[Lemma, p,~9]{St} and select from $\cF$ a (possibly finite) sequence $\{B_j\}$ of mutually disjoint balls such that $\{5B_j\}$ covers $E_\si$.
	Therefore 
	$$
		\mod{E_\si} \leq \sum_j \, \mod{5B_j} \leq L_{5,1} \, \sum_j \, \mod{B_j} \leq \frac{L_{5,1}}{\si} \, \int_{B_j} \, \mod{f_z} \wrt \mu
		\leq \frac{L_{5,1}}{\si} \bignorm{f_z}{1},     
	$$
	as required.   
\end{proof}

We prove our main result, Theorem~\ref{t: main}, which we restate for the reader's convenience.

\begin{theorem*} %\label{t: main}
	Suppose that $1<a\leq b< a^2$ and that $\de$ is a nonnegative number.  Assume that $(X,d)$ is a $\de$-hyperbolic complete length space and 
	$\mu$ is locally doubling  Borel measure on $X$ such that \eqref{f: pinched exp} holds.  Then the centred HL maximal operator $\cM$
	is bounded on $\lp{X}$ for all $p>\tau$, and it is of weak type $(\tau,\tau)$.  
\end{theorem*}

\begin{proof}
	By Proposition~\ref{p: weak cM0}, $\cM_0$ is of weak type $(1,1)$ and bounded on $\lp{X}$ for all $p$ in $(1,\infty]$.   Thus, it suffices to prove
	that $\cM_\infty$ has the required mapping properties.  
	
	Since $X$ belongs to the class $\cX_{a,b}^\de$, Theorem~\ref{t: sri spider} implies that $X$ is strictly roughly isometric to a spider's web $\wh\Ga$ 
	in the class $\Upsilon_{a,b}^{\de'}$,  endowed with its graph distance.	 More precisely, the proof of Theorem~\ref{t: sri spider} 
	shows that the set of vertices of $\wh\Ga$ is a discrete subset of $X$, and that the identity operator $\iota: \wh\Ga\to X$ is 
	a strict $\be$-rough isometric embedding of $(\wh\Ga,d_{\wh\Ga})$ into $(X,d)$, i.e. 
	$$
	d_{\wh\Ga}(x,y) -\be 
	\leq d(x,y)
	\leq d_{\wh\Ga}(x,y) + \be
	\quant x,y\in \wh\Ga.  
	$$
	Now, by Theorem~\ref{t: NT for spider web}, the operator
	$\cM_\infty$ is of weak type $(\tau,\tau)$ and bounded on $\lp{\wh\Ga}$ for all $p$ in $(\tau,\infty]$.  

	It remains to prove that the operator $\cM_\infty$ on $X$ inherits the abovementioned mapping properties of the operator $\cM_\infty$ on $\wh\Ga$.

	Recall that $\cC := \{B_2(x): x \in \wh\Ga\}$ is a covering of $X$ (see \eqref{f: covering}).  For each point $z$ in $X$ denote by $\wh\Ga_z$ the set of 
	all~$x$ in $\wh\Ga$ such that $z$ belongs to $B_2(x)$, and set $\ds \om := \sup_{z\in X} \sharp (\wh\Ga_z)$.  We call $\om$ the \textit{overlapping number}
	of the covering~$\cC$.

	We \textit{claim} that $\om$ is finite. 
	Indeed, observe that  $d(z,x) < 2$ for every $x$ in $\wh\Ga_z$.  Since the vertices of $\wh\Ga$ consitute a $2$-discretisation of $X$, 
	the balls in $\{B_{1}(x): x \in \wh\Ga\}$ are mutually disjoint and the balls in $\{B_{1}(x): x \in \wh\Ga_z\}$ are contained in~$B_{3}(z)$.  
	This and the upper estimate in \eqref{f: pinched exp} imply that 
	$$
	\sharp (\wh\Ga_z) \cdot \inf\, \big\{\mu\big(B_{1}(x)\big): x \in \wh\Ga_z\big\}
	\leq \sup\, \big\{\mu\big(B_{3}(z)\big): z \in X\big\} 
	\leq C\, b^{3}.  
	$$
	Now, the LDP and the lower estimate in \eqref{f: pinched exp} imply that 
	$$
	c\, a \leq \mu\big(B_1(x)\big).
	$$
	Altogether, we see that 
	$$
	\sharp (\wh\Ga_z) \leq C \, \frac{b^{3}}{c\,a}
	\quant z \in X,
	$$
	and the claim follows.
	
	For each nonnegative measurable function $f$ on $X$, consider the function~$\pi f$ on $\wh\Ga$, defined by the formula
	$$
	(\pi f)(x)
	:= \int_{{B_{2}(x)}} \, f \wrt \mu
	\quant x \in \wh\Ga.  
	$$
	For each $z$ in $X$, denote by $x_z$ a point in $\wh\Ga$ at minimum distance in $X$ from~$z$; in particular, $d(z,x_z)  <2$.  There are, at most, 
	$\om$ such points, so that the choice of $x_z$ is somehow arbitrary.  However, this will have no consequences in what follows. 
	For each $x$ in $\wh\Ga$, denote by $\Om_x$ the set of all $z$ in $X$ such that $x_z=x$.  Clearly~$\Om_x$ is contained in ${B_{2}(x)}$, and 
	$$
	X = \bigcup_{x\in \wh\Ga} \, \Om_x.
	$$
	Simple geometric considerations show that for each $z$ in $X$ and for every $R>1$ the following containment holds
	$$
	B_R(z) 
	\subseteq \bigcup_{x\in \wh\Ga \cap{B_{R+2}(z)}} \, B_{2}(x) 
	\subseteq {B_{R+4}(z)}  
	\subseteq {B_{R+6}(x_z)}, 
	$$
	so that 
	$$
	\One_{B_R(z)}  
	\leq \sum_{x\in \wh\Ga \cap {B_{R+2}(z)}} \, \One_{{B_{2}(x)}} 
	\leq \om \, \One_{B_{R+{6}}(x_z)}.  
	$$
	Consequently
	$$
	\begin{aligned}
	\int_{B_R(z)} f \wrt \mu
		& \leq \sum_{x\in \wh\Ga \cap {B_{R+2}(z)}} \, \int_X f \, \One_{{B_{2}(x)}} \, \wrt \mu \\ 
		& =    \sum_{x\in \wh\Ga \cap {B_{R+2}(z)}} \, \pi f (x) \\ 
		& \leq \sum_{x\in \wh\Ga \cap {B_{R+4}(x_z)}} \, \pi f (x).
	\end{aligned}
	$$
	Since the identity is a strict $\be$-rough isometry between $(\wh\Ga,d_{\wh\Ga})$ and $(X,d)$,   
	$$
	d_{\wh\Ga}(x,x_z) \leq d(x,x_z) + \be,
	$$
	so that $\wh \Ga \cap {B_{R+4}(x_z)} \subseteq {B_{R+4+\be}^{\wh\Ga}(x_z)}$.  Thus, we have proved that 
	\begin{equation} \label{f: ineq integral}
	\int_{B_R(z)} f \wrt \mu
		\leq \int_{B_{R+\be+{4}}^{\wh\Ga}(x_z)} \pi f \, \wrt \mu_{\wh\Ga},
	\end{equation}
	where $\mu_{\wh\Ga}$ denotes the counting measure on the vertices of $\wh\Ga$.   
	Now observe that 
	\begin{equation} \label{f: ineq measures}
		\begin{aligned}
			\mu_{\wh\Ga}\big(B_{R+\be+{4}}^{\wh\Ga}(x_z)\big) 
			& \leq  \sum_{x\in B_{R+\be+{4}}^{\wh\Ga}(x_z)} \frac{\mu\big({B_{1}(x)\big) }}{ca} \\
			& \leq \om \,(c\, a)^{-1} \, \mu\big(B_{R+{2}\be+{5}}(x_z)\big) \\  
			& \leq \om \, C_\be \, (c\, a)^{-1} \, \mu\big(B_{R}(z)\big),
		\end{aligned}
	\end{equation} where in the last line we have used the following easy observation: 
	\begin{align*}
     	B_{R+2\be+5}(x_z) \subset B_{R+2\be+7}(z) \subset \bigcup_{x \in \wh \Ga \cap B_{R}(z)} B_{2\beta +11}(x),
	\end{align*} 
	that implies 
	\begin{align*}
     	\mu(B_{R+\be+5}(x_z)) 
		 & \le \sum_{x \in \wh \Ga \cap B_{R}(z)} \mu\big(B_{2\beta +11}(x)\big) \\
		 & \le C\, b^{2\be}\, \mu_{\wh\Ga} \big(\{x \in \wh \Ga \cap B_{R}(z)\}\big) \\ 
		 & \le C_\beta \, \mu\big(B_R(x)\big), 
	\end{align*} 
	as desired.  By combining \eqref{f: ineq integral} and \eqref{f: ineq measures}, we find that to each $z$ in $X$, we can associate a point~$x_z$ in 
	$\wh\Ga$ such that $d(z,x_z) \leq 2$ and 
	$$
	\cM_\infty f (z)
	\leq C_{\be}\,\, \cM_\infty (\pi f) (x_z).
	$$
	For each $\al>0$ set
	\begin{align*}
 	\wh\Ga(\al)
		:= \big\{x\in \wh\Ga: \hbox{ there exists $z \in E_{\cM_\infty f}(\al)$ such that $x_z=x$}.\}
	\end{align*}
	Observe that 
	$$
	E_{\cM_\infty f}(\al)
	= \bigcup_{x\in \wh\Ga(\al)} \, \big(E_{\cM_\infty f}(\al) \cap \Om_x\big)
	\subseteq \bigcup_{x\in \wh\Ga(\al)} \, \big(E_{\cM_\infty f}(\al) \cap {B_{{2}}(x)} \big),
	$$
	so that 
	\begin{equation} \label{f: transfer weak I}
	\mu\big(E_{\cM_\infty f}(\al)\big)
	\leq \sum_{x\in \wh\Ga(\al)} \, \mu\big({B_{{2}}(x)} \big)
	\leq C\, b^2 \cdot \sharp\big(\wh\Ga(\al)\big).  
	\end{equation}
	Now, if $x$ belongs to $\wh\Ga(\al)$, then $C_{\be}\,\, \cM_\infty (\pi f) (x) > \al$, i.e. $x$ belongs to $E_{\cM_\infty (\pi f)} (\al/C_\be)$,
	whence 
	$$
	\sharp\big(\wh\Ga(\al)\big) 
	\leq \mu_{\wh\Ga}\big(E_{\cM_\infty (\pi f)} (\al/C_\be)\big).
	$$
	Therefore the following inequality between the distribution functions of $\cM_\infty f$ and $\cM_\infty (\pi f)$ holds 
	$$
	\mu\big(E_{\cM_\infty f}(\al)\big)
	\leq C\, b^2 \cdot \mu_{\wh\Ga}\big(E_{\cM_\infty (\pi f)} (\al/C_\be)\big)
	\quant \al >0. 
	$$
	Now, observe that if $f$ is in $L^\tau(X)$, then $\pi f$ is in $L^\tau(\wh\Ga)$.  Indeed, 
	\begin{equation} \label{f: transfer weak II}
	\begin{aligned}
		\bignormto{\pi f}{L^\tau(\wh\Ga)}{\tau}
		& =    \sum_{x\in \wh\Ga} \, \Big(\int_{{B_{{2}}(x)}} f \wrt \mu\Big)^\tau \\
		& \leq \sum_{x \in \wh\Ga} \mu\big({B_{2}(x)}\big)^{\tau/\tau'} \,\int_{X}  f^\tau \, \One_{{B_2(x)}}  \wrt \mu \\
		& \leq (Cb^2)^{\tau/\tau'}  \, \om \bignormto{f}{L^\tau(X)}{\tau};
	\end{aligned}
	\end{equation}
	the first inequality follows from H\"older's inequality (with exponents $\tau$ and~$\tau'$), and the second from condition \eqref{f: pinched exp}
	and the fact that $\om$ is the overlapping number of the cover $\cC$.    
	
	Now, $\cM_\infty$ is of weak type $(\tau,\tau)$ on $\wh\Ga$ by assumption.   This, together with the estimates \eqref{f: transfer weak I} and 
	\eqref{f: transfer weak II}, implies that 
	$$
	\mu\big(E_{\cM_\infty f}(\al)\big)
	\leq C_\beta\, (Cb^2)^{1+\tau/\tau'}  \, \frac{\om}{\al^\tau} \, \bignormto{f}{L^\tau(X)}{\tau},
	$$ 
	i.e. $\cM_\infty$ is of weak type $(\tau,\tau)$ on $X$, as required.  

	This concludes the proof of the theorem.  
\end{proof}

\begin{remark}
    	Observe that when $a=b$ in the assumptions of the above theorem, $\cM$ is of weak type (1,1). 
	This case includes all symmetric spaces of noncompact type of rank 1.
\end{remark}

\begin{corollary} \label{c: CH manifolds}
	Suppose that $A$ and $B$ are positive numbers such that $A\leq B <2A$, and $M$ is a Cartan--Hadamard Riemannian manifold with pinched curvature 
	sectional curvature, i.e. $-B^2\leq K\leq -A^2$.  Then $\cM$ is of weak type $(B/A, B/A)$ and it is bounded on $\lp{M}$ for all $p>B/A$.
\end{corollary}

\begin{proof}
	It is well known that $M$ is a $\de$-hyperbolic space.  
	By comparison results \cite[Corollary~3.2~(ii)]{Sa}, 
	$$
	\e^{(n-1)Ar}
	\leq \bigmod{B_r(x)} 
	\leq \e^{(n-1)Br}
	\quant x \in M \quant r \geq 1.  
	$$
	Thus, if we set $a := \e^{(n-1)A}$ and $a := \e^{(n-1)B}$, then the condition $B<2A$ transforms to $b<a^2$, so that
	$M$ belongs to the class $\cX_{a,b}^\de$.  Therefore, by Theorem~\ref{t: main}, $\cM$ is of weak type $(\tau,\tau)$ and 
	it is bounded on $\lp{M}$ for all $p$ in $(\tau,\infty]$.   Observe that 
	$$
	\log_ab
	= \frac{\log b}{\log a} 
	= \frac{B}{A}. 
	$$
	Hence $\cM$ has the required mapping properties.  
\end{proof}


\begin{thebibliography}{MMV2}

\bibitem[ADY]{ADY} J.-Ph. Anker, E. Damek, C. Yacoub, Spherical analysis on harmonic $A N$ groups, 
	\textit{Ann. Scuola Norm. Sup. Pisa Cl. Sci.} \textbf{23} (1996), 643--679.

\bibitem[Bo]{Bo} M. Bonk,  Quasi-geodesic segments and Gromov hyperbolic spaces, {\it Geom. Dedicata} \textbf{62} (1996) 281--298.

\bibitem[BH]{BH} M.R.~Bridson, A.~Haefliger, \textit{Metric spaces of non-positive curvature}, Springer, 1999.   

\bibitem[BBI]{BBI} D.~Burago, Y.~Burago, S. Ivanov, {\it A course in metric geometry}, Graduate Studies in Math. \textbf{33},
American Mathematical Society, 2001.

\bibitem[BI]{BI} D.~Burago, S. Ivanov, Uniform approximation of metrics by graphs, 
\textit{Proc. Amer. Math. Soc.} \textbf{143} (2015), 1241--1256.

\bibitem[CDP]{CDP} M. Coornaert, T. Delzant, A. Papadopoulos, \textit{Géométrie et théorie des
	groupes. Les groupes hyperboliques de Gromov}, Lecture Notes in Mathematics, vol. \textbf{1441}, Springer-Verlag, 1990. 

\bibitem[CMS1]{CMS1} M. Cowling, S. Meda, A. G. Setti, 
An overview of harmonic analysis on the group of isometries of a homogeneous tree, \textit{Exposition. Math.} {16} (1998), 385--423.

\bibitem[CMS2]{CMS2} M. Cowling, S. Meda, A. G. Setti, 
A weak type $(1, 1)$ estimate for a maximal operator on a group of isometries of homogeneous trees,
\emph{Coll. Math.} \textbf{118} (2010), 223--232.

\bibitem[dLHG]{dLHG} P.~de la Harpe, E.~Ghys,  \textit{Sur les groupes hyperboliques d'après Mikhael Gromov}, Birkhäuser, 1990. 

\bibitem[Gr]{Gr} M. Gromov,  {\it Hyperbolic groups}, Gersten, S.M. (eds) Essays in Group Theory. Mathematical Sciences Research Institute Publications, 
vol 8. Springer, New York, 1987.

\bibitem[He]{He} J. Heinonen, {\it Lectures on Analysis on Metric Spaces}, London Math.
Society Lecture Notes Series, n.~\textbf{162},  Springer Verlag, 2001.

\bibitem[K]{K} M. Kanai, Rough isometries, and combinatorial approximation of non-compact
Riemannian manifolds, 
\textit{J. Math. Soc. Japan} \textbf{37} (1985), 391--413.

\bibitem[Kn]{Kn} G. Knieper, New results on noncompact harmonic manifolds. \textit{Comment. Math. Helv.} \textbf{87} (2012), 669--703.

\bibitem[LMSV]{LMSV} M. Levi, S. Meda, F. Santagati and M. Vallarino, Hardy--Littlewood maximal operators on trees with bounded geometry, 
	available at arXiv:2308.07128v1 [math FA], to appear in \textit{Trans. Amer. Math. Soc.}.

\bibitem[LS]{LS} M. Levi and F. Santagati, Hardy--Littlewood fractional maximal operators on homogeneous trees, \textit{Math. Z.} \textbf{308} (2024).

\bibitem[L1]{L1} H.-Q. Li, La fonction maximale de Hardy--Littlewood sur une classe d'espaces m\'etriques measurables, 
\textit{C. R. Acad. Sci. Paris, Ser. I} \textbf{338} (2004), 31--34.

\bibitem[L2]{L2} H.-Q. Li, Les fonctions maximales de Hardy--Littlewood pour des measures sur les vari\'et\'es de type cuspidale, 
\textit{J. Math. Pures Appl.} \textbf{88} (2007), 261--275.

\bibitem[Lo]{Lo} N. Lohoué, Fonction maximale sur les variétés de Cartan--Hadamard, \textit{C. R. Acad. Sci. Paris Sér. I} \textbf{300} (1985), 213--216.

\bibitem[MPSV]{MPSV} S. Meda, S. Pigola, A.G. Setti, G. Veronelli,  Hardy--Littlewood maximal operators on certain manifolds with bounded geometry, available at arXiv:2502.13109,
preprint.

\bibitem[MS]{MS} S. Meda and F. Santagati, Triangular maximal operators on locally finite trees, \textit{Mathematika}  \textbf{70} (2024),  https://doi.org/10.1112/mtk.12253.

\bibitem[MVo]{MVo} S. Meda and S. Volpi, Spaces of Goldberg type on 
certain measured metric spaces,  \emph{Ann. Mat. Pura Appl.} \textbf{196} (2017), 947--981.

\bibitem[NT]{NT} A. Naor and T. Tao, Random martingales and localization of maximal inequalities, 
\emph{J. Funct. Anal.} \textbf{259} (2010), 731--779.

\bibitem[OR]{OR} S. Ombrosi and I.P. Rivera-Rios, Weighted $L^p$ estimates on the infinite rooted $k$-ary tree,
\textit{Math. Ann.} \textbf{384} (2022), 491--510.

\bibitem[RT]{RT} R.~Rochberg and M.~Taibleson, Factorization of the Green's operator and weak-type estimates for a random walk on a tree,
\emph{Publ. Math.} \textbf{35} (1991), 187--207.

\bibitem[Sa]{Sa} T. Sakai, \emph{Riemannian geometry}, Translations of Mathematical
Monographs, Vol.~\textbf{149}, American Mathematical Society, 1992.

\bibitem[ST]{ST}
J. Soria and P. Tradacete, Geometric properties of infinite graphs and the Hardy--Littlewood maximal operator,
\emph{J. Anal. Math.} \textbf{137} (2019), 913--937.

\bibitem[St]{St} E.M. Stein, \emph{Singular Integrals and Differentiability Properties of Functions}, Princeton Univ. Press,
Princeton, N.J., 1970.

\bibitem[Str]{Str} J.-O.~Str\"omberg,
Weak type $L^1$ estimates for maximal functions on non-compact
symmetric spaces, \emph{Ann. of Math.}
\textbf{114} (1981), 115--126.

\end{thebibliography}
\end{document}